\documentclass[11pt]{article}

\usepackage[colorlinks = true, pdfstartview = FitV, linkcolor = blue, citecolor = blue, urlcolor = blue, hypertexnames=false]{hyperref}
\usepackage{booktabs}
\usepackage{graphicx}
\usepackage{amsmath}
\usepackage{amssymb}
\usepackage{latexsym}
\usepackage{subfigure}
\usepackage{crop}
\usepackage{algorithmic,algorithm}
\usepackage{multirow}
\usepackage{bm}
\usepackage{bbm}
\usepackage{enumerate}
\usepackage{url}
\usepackage{array}
\usepackage{paralist}
\usepackage{diagbox}
\usepackage{booktabs}
\usepackage[dvipsnames]{xcolor}

\usepackage{rotating}
\usepackage[capitalise]{cleveref}
\crefname{equation}{}{}
\crefname{figure}{Figure}{Figures}
\creflabelformat{equation}{\textup{(#2#1#3)}}
\crefname{assumption}{Assumption}{Assumptions}
\crefname{condition}{Condition}{Conditions}

\usepackage{xspace}

\renewcommand\th{\textsuperscript{th}\xspace}

\usepackage{fullpage}
\usepackage{multirow}

\usepackage[sort,nocompress]{cite}

\usepackage{arydshln}
\usepackage{enumitem}

\usepackage{pifont}
\usepackage{accents}

\usepackage{stackengine}
\stackMath
\newcommand\tsup[2][2]{%
	\def\useanchorwidth{T}%
	\ifnum#1>1%
	\stackon[-.5pt]{\tsup[\numexpr#1-1\relax]{#2}}{\scriptscriptstyle\sim}%
	\else%
	\stackon[.5pt]{#2}{\scriptscriptstyle\sim}%
	\fi%
}

\definecolor{forestgreen}{rgb}{0.13, 0.55, 0.13}

\definecolor{amber}{rgb}{1.0, 0.75, 0.0}

\definecolor{bananayellow}{rgb}{.8, 0.6, 0}

\definecolor{cuhkpl}{RGB}{152,24,147}

\newcounter{comment}

\usepackage{amsthm}
\usepackage[framemethod=TikZ]{mdframed}

\newmdtheoremenv[%
linewidth = 1pt,%
roundcorner = 10pt,%
leftmargin = 0,%
rightmargin = 0,%
backgroundcolor = green!3,%
outerlinecolor = blue!70!black,%
splittopskip = \topskip,%
ntheorem = true,%
]{theorem}{Theorem}

\newmdtheoremenv[%
linewidth = 1pt,%
roundcorner = 10pt,%
leftmargin = 0,%
rightmargin = 0,%
backgroundcolor = green!3,%
outerlinecolor = blue!70!black,%
splittopskip = \topskip,%
ntheorem = true,%
]{corollary}{Corollary}

\newmdtheoremenv[%
linewidth = 1pt,%
roundcorner = 10pt,%
leftmargin = 0,%
rightmargin = 0,%
backgroundcolor = green!3,%
outerlinecolor = blue!70!black,%
splittopskip = \topskip,%
ntheorem = true,%
]{lemma}{Lemma}

\newmdtheoremenv[%
linewidth = 1pt,%
roundcorner = 10pt,%
leftmargin = 0,%
rightmargin = 0,%
backgroundcolor = green!3,%
outerlinecolor = blue!70!black,%
splittopskip = \topskip,%
ntheorem = true,%
]{proposition}{Proposition}

\newmdtheoremenv[%
linewidth = 1pt,%
roundcorner = 10pt,%
leftmargin = 0,%
rightmargin = 0,%
backgroundcolor = blue!3,%
outerlinecolor = blue!70!black,%
splittopskip = \topskip,%
ntheorem = true,%
]{definition}{Definition}

\newmdtheoremenv[%
linewidth = 1pt,%
roundcorner = 10pt,%
leftmargin = 0,%
rightmargin = 0,%
backgroundcolor = yellow!3,%
outerlinecolor = blue!70!black,%
splittopskip = \topskip,%
ntheorem = true,%
]{condition}{Condition}

\newmdtheoremenv[%
linewidth = 1pt,%
roundcorner = 10pt,%
leftmargin = 0,%
rightmargin = 0,%
backgroundcolor = yellow!3,%
outerlinecolor = blue!70!black,%
splittopskip = \topskip,%
ntheorem = true,%
]{assumption}{Assumption}

\theoremstyle{definition}
\newmdtheoremenv[%
linewidth = 1pt,%
roundcorner = 10pt,%
leftmargin = 0,%
rightmargin = 0,%
backgroundcolor = cyan!3,%
outerlinecolor = blue!70!black,%
splittopskip = \topskip,%
ntheorem = true,%
]{example}{Example}

\theoremstyle{definition}
\newmdtheoremenv[%
linewidth = 1pt,%
roundcorner = 10pt,%
leftmargin = 0,%
rightmargin = 0,%
backgroundcolor = red!3,%
outerlinecolor = blue!70!black,%
splittopskip = \topskip,%
ntheorem = true,%
]{remark}{Remark}

\mdtheorem[%
linewidth = 1pt,%
roundcorner = 10pt,%
leftmargin = 0,%
rightmargin = 0,%
backgroundcolor = green!3,%
outerlinecolor = blue!70!black,%
splittopskip = \topskip,%
ntheorem = true,%
]{informal}{Contributions and Main Results (Informal)}

\mdtheorem[%
linewidth = 1pt,%
roundcorner = 10pt,%
leftmargin = 0,%
rightmargin = 0,%
backgroundcolor = yellow!3,%
outerlinecolor = blue!70!black,%
splittopskip = \topskip,%
ntheorem = true,%
]{NPC}{NPC Condition}

\mdtheorem[%
linewidth = 1pt,%
roundcorner = 10pt,%
leftmargin = 0,%
rightmargin = 0,%
backgroundcolor = yellow!3,%
outerlinecolor = blue!70!black,%
splittopskip = \topskip,%
ntheorem = true,%
]{SOL}{Inexactness Condition}

	\usepackage{tikz}
	\usepackage{xparse}
	
	\NewDocumentCommand\DownArrow{O{2.0ex} O{black}}{%
		\mathrel{\tikz[baseline] \draw [<-, line width=0.5pt, #2] (0,0) -- ++(0,#1);}
	}
	
	\usepackage{listings} 
	
	\definecolor{mygreen}{rgb}{0,0.6,0}
	\definecolor{mygray}{rgb}{0.5,0.5,0.5}
	\definecolor{mymauve}{rgb}{0.58,0,0.82}
	
	\usepackage{dsfont}

	\newcommand{\df}{\mathrm{d}}
	\newcommand{\real}{\mathbb{R}}

	\newcommand{\Klv}{\mathcal{K}}
	
	\newcommand{\dt}{\text{D}_{\text{type}}}
	\newcommand{\Klvt}{\mathcal{K}_{t}}
	\newcommand{\T}{\intercal}
	\DeclareMathOperator*{\argmin}{arg\,min}

	\newcommand {\dotprod}[1]{\left\langle #1\right\rangle}
	\newcommand {\vnorm}[1]{\left\| #1\right\|}
	\newcommand {\abs}[1]  {\left| #1 \right|}

	\newcommand {\bigO}[1]{\mathcal O\left( #1\right)}
	
	\newcommand {\tbigO}[1]{\tilde{\mathcal{O}} \left( #1\right)}

	\newcommand\bbR{\ensuremath{\mathbb{R}}} 
	
	\newcommand {\range}  {\textnormal{Range}}
	\newcommand {\Null}  {\textnormal{Null}}
	\newcommand {\sign}  {\textnormal{Sign}}
	\newcommand {\Span}  {\textnormal{Span}}

	\newcommand {\zero}   {\mathbf{0}}

	\renewcommand {\aa}   {\mathbf{a}}
	
	\newcommand {\dd}   {\mathbf{d}}
	\newcommand {\ddk}  {{\dd}_{k}}

	\newcommand {\dn}   {d}
	
	\newcommand {\ee}   {\mathbf{e}}

	\renewcommand {\gg}  {\mathbf{g}}
	\newcommand {\ggk}  {{\gg}_{k}}
	\newcommand {\ggkk}  {{\gg}_{k+1}}

	\newcommand {\tgg}  {\tilde{\gg}}
	
	\newcommand {\tg}  {\tilde{g}}

	\newcommand {\tgd}  {{\tilde{g}}_{d}}
	\newcommand {\HH}  {\mathbf{H}}
	\newcommand {\HHk}  {\HH_{k}}
	
	\newcommand {\tHH}  {\tilde{\HH}}
	\newcommand {\HHd}  {\HH^{\dagger}}
	\newcommand {\HHdk}  {\left[\HH_{k}\right]^{\dagger}}
	\newcommand {\tHHk}  {\tilde{\HH}_{k}}

	\newcommand {\eye}  {\mathbf{I}}

	\newcommand {\kry}  {\mathcal{K}_{t}}
	\newcommand {\Kt}[1]{\mathcal{K}_{t}\left(#1\right)}

	\newcommand {\TN}  {T_\text{N}}
	\newcommand {\TNL}  {T_\text{L}}
	\newcommand {\TP}  {T_\text{P}}
	\newcommand {\TS}  {T_\text{S}}
	\newcommand {\KN}  {K_\text{N}}
	
	\newcommand {\KS}  {K_\text{S}}
	
	\newcommand {\Lg}  {L_{\gg}}
	\newcommand {\LH}  {L_{\HH}}

	\newcommand {\mi}  {m_{i}}
	
	\newcommand {\pp}   {\mathbf{p}}

	\newcommand {\QQ}  {\mathbf{Q}}

	\newcommand {\qq}   {{\mathbf{q}}}

	\newcommand {\RR}  {\mathbf{R}}

	\newcommand {\bRR}  {\bar{\RR}}

	\newcommand {\rr}   {\mathbf{r}}

	\newcommand {\rrkt}  {{\rr}_{k}^{(t)}}
	\newcommand {\rrktp}  {{\rr}_{k}^{(t-1)}}
	\newcommand {\rrt}  {{\rr}_{t}}
	\newcommand {\rrtp}  {{\rr}_{t-1}}

	\newcommand {\rrg}  {\rr_{g}}

	\renewcommand {\ss}  {\mathbf{s}}
	\newcommand {\sst}  {\ss_{t}}
	\newcommand {\sskt}  {\ss_{k}^{(t)}}
	\newcommand {\ssktp}  {\ss_{k}^{(t-1)}}
	\newcommand {\sstp}  {\ss_{t-1}}

	\newcommand {\TT}  {\mathbf{T}}
	\newcommand {\TTt}  {\mathbf{T}_{t}}
	\newcommand {\TTg}  {\mathbf{T}_{g}}
	
	\newcommand {\tTTt}  {\tilde{\mathbf{T}}_{t}}
	
	\newcommand {\bTT}  {\bar{\TT}}

	\newcommand {\ttt}  {\mathbf{t}}
	
	\newcommand {\UU}  {\mathbf{U}}
	
	\newcommand {\UUT}  {\UU^{\intercal}}
	\newcommand {\UUp}  {\UU_{\perp}}
	\newcommand {\UUpT}  {\UU_{\perp}^{\intercal}}

	\newcommand {\uu}   {\mathbf{u}}

	\newcommand {\VV}  {\mathbf{V}}
	
	\newcommand {\VVt}  {\VV_{t}}

	\newcommand {\VVg}  {\VV_{g}}
	
	\newcommand {\VVtT}  {\VV_{t}^{\T}}

	\newcommand {\vv}   {\mathbf{v}}

	\newcommand {\WW}  {\mathbf{W}}

	\newcommand {\ww}   {\mathbf{w}}

	\newcommand {\xx}   {\mathbf{x}}

	\newcommand {\xxs}  {\xx^{\star}}

	\newcommand {\xxk}  {{\xx}_{k}}
	
	\newcommand {\xxkn}  {{\xx}_{k+1}}

	\newcommand {\YY}  {\mathbf{Y}}

	\newcommand {\yy}   {\mathbf{y}}

	\newcommand {\zz}  {\mathbf{z}}

	\newcommand {\alphak}  {{\alpha}_{k}}

	\newcommand {\talpha}  {\tilde{\alpha}}

	\newcommand {\tbeta}  {\tilde{\beta}}

	\newcommand {\epsg}  {\varepsilon_{\gg}}

	\newcommand {\epsH}  {\varepsilon_{\HH}}
	
	\newcommand {\beps}  {\bar{\varepsilon}}

	\newcommand {\kappap}  {\kappa^{+}}
		\newcommand {\kappan}  {\kappa^{-}}

	\newcommand {\nuj}  {{\nu}_{j}}
	\newcommand {\tnu}  {\tilde{\nu}}
	
	\newcommand {\psip}  {{\psi}_{+}}
	\newcommand {\psiz}  {{\psi}_{0}}
	\newcommand {\psin}  {{\psi}_{-}}

	\newcommand{\hf}{\frac12}

	\newcommand{\defeq}{\triangleq}

	\renewcommand{\Pr}{\hbox{\bf{Pr}}}

\begin{document}
	
	\title{A Newton-MR algorithm with complexity guarantees for nonconvex smooth unconstrained optimization}
	
	\author{Yang Liu\thanks{Mathematical Institute, University of Oxford, UK. {Email: \tt{yang.liu@maths.ox.ac.uk}}.}
		\and
		Fred Roosta\thanks{School of Mathematics and Physics, University of Queensland, Australia, and International Computer Science Institute, Berkeley, USA. Email: \tt{fred.roosta@uq.edu.au}}
	}
	
	\date{\today}
	\maketitle
	
	\abstract{In this paper, we consider variants of Newton-MR algorithm, initially proposed in \cite{roosta2018newton}, for solving unconstrained, smooth, but non-convex optimization problems. Unlike the overwhelming majority of Newton-type methods, which rely on conjugate gradient algorithm as the primary workhorse for their respective sub-problems, Newton-MR employs minimum residual (MINRES) method. Recently, \cite{liu2022minres} establishes certain useful monotonicity properties of MINRES as well as its inherent ability to detect non-positive curvature directions as soon as they arise. We leverage these recent results and show that our algorithms come with desirable properties including competitive first and second-order worst-case complexities. Numerical examples demonstrate the performance of our proposed algorithms.
		
    \section{Introduction}
\label{sec:intro}
Consider the unconstrained optimization problem
\begin{equation}
	\label{eq:obj}
	\min_{\xx \in \real^{\dn}} f(\xx),
\end{equation}
where $ f:\real^{\dn} \rightarrow \real $ is sufficiently smooth.  Over the last few decades, many
optimization algorithms have been developed to solve\cref{eq:obj} in a variety of settings. Among them, the class of second-order algorithms play a prominent role.
The canonical example of such algorithms is arguably the classical Newton's method whose iterations are typically written as $ \xx_{k+1} = \xxk + \alphak \ddk $ with 
\begin{align}
	\label{eq:linear_system}
	\HHk \ddk = - \ggk,
\end{align}
where $ \xxk$, $\ggk = \nabla f(\xxk)$, $\HHk = \nabla^{2} f(\xxk)$, and $\alphak $ are respectively the current iterate, the gradient, the Hessian matrix, and the step-size that is often chosen using an Armijo-type line-search to enforce sufficient decrease in $ f $ \cite{nocedal2006numerical}. When $ f $ is smooth and strongly convex, it is well known that the local and global convergence rates of the classical Newton's method are, respectively, quadratic and linear \cite{nesterov2004introductory,boyd2004convex,nocedal2006numerical}. 
For such problems, if forming the Hessian matrix explicitly and/or solving \cref{eq:linear_system} exactly is prohibitive, the update direction $ \ddk $ is obtained using conjugate gradient (CG) to approximately solve \cref{eq:linear_system}, resulting in the celebrated Newton-CG method \cite{nocedal2006numerical,fasano2009nonmonotone,gould2000exploiting}. 

However, in non-convex problems, where the Hessian matrix can be singular and/or indefinite, the classical Newton's method and its straightforward Newton-CG variant, lack any convergence guarantees. 
To address this, many Newton-type variants have been proposed which aim at extending Newton-CG beyond strongly convex problems to more general settings, e.g., trust-region \cite{conn2000trust} and cubic regularization \cite{nesterov2006cubic,cartis2011adaptiveI,cartis2011adaptiveII}. However, many of these extensions involve sub-problems that are non-trivial to solve, e.g., the sub-problems of trust-region and cubic regularization methods are non-linear and non-convex \cite{xuNonconvexEmpirical2017,xuNonconvexTheoretical2017}. For instance, the CG-Steihaug method \cite{steihaug1983conjugate} is not guaranteed to solve the trust-region sub-problem to an arbitrary accuracy. This can negatively affect the convergence speed of the trust-region method. Indeed, if either of the negative curvature or the boundary is encountered too early, the CG-Steihaug method terminates and the resulting step is only slightly, if at all, better than the Cauchy direction \cite{conn2000trust}.

In this light, there has been recent efforts in studying non-convex Newton-type methods whose sub-problems involve elementary linear algebra problems, which are much better understood. 
Notably among these is \cite{royer2020newton}, which enhances the classical Newton-CG approach with safeguards to detect non-positive curvature (NPC) direction in the Hessian, that may arise during the iterations of CG.  
By exploiting NPC directions, the method in \cite{royer2020newton} can be applied to general non-convex settings and enjoys favorable complexity guarantees that have been shown to be optimal in certain settings. Stochastic variants of the method in \cite{royer2020newton} has also recently been considered in \cite{yao2021inexact}.
In fact, beyond CG, all conjugate direction methods such as conjugate residual (CR) \cite{saad2003iterative} involve iterations that allow for ready access to NPC directions. Using this property, \cite{dahito2019conjugate} studies variants of Newton's algorithm where CR replaces CG as the sub-problem solver, and by employing NPC direction, obtains asymptotic convergence guarantees for non-convex problems.
Arguably, however, the iterative method-of-choice when it comes to real symmetric but potentially indefinite matrices is the celebrated Minimum Residual (MINRES) algorithm \cite{paige1975solution,choi2011minres}. Inspired by this, and by reformulating the sub-problems as simple symmetric linear least-squares, a new variant of Newton's method, called Newton-MR, is introduced in \cite{liu2021convergence,roosta2018newton} where the classical CG algorithm is replaced with MINRES. However, \cite{liu2021convergence,roosta2018newton} are limited in their scope in that, unlike the Newton-CG alternatives in \cite{royer2020newton}, the proposed Newton-MR algorithms can only be applied to a sub-class of non-convex objectives, known as invex problems \cite{mishra2008invexity}. This shortcoming lies in the fact that the algorithms proposed in \cite{liu2021convergence,roosta2018newton} do not have the ability to leverage NPC directions.

Recently, \cite{liu2022minres} has established several non-trivial properties of MINRES that mimic, and in fact surpass, those found in CG. These include certain useful monotonicity properties as well as an inherent ability to detect NPC directions, when they arise. These properties allow MINRES to be considered as an alternative sub-problem solver to CG for many Newton-type non-convex optimization algorithms. Equipped with the results of \cite{liu2022minres}, in this paper we aim to address the shortcomings of the algorithms in \cite{liu2021convergence,roosta2018newton} and introduce variants of Newton-MR algorithms, which can be applied to general non-convex problems and enjoy highly desirable complexity guarantees. As compared with Newton-CG alternatives in \cite{royer2020newton}, we will show that not only do our proposed Newton-MR variants have advantageous convergence properties, but they also have superior empirical performance in a variety of examples.

\begin{informal*}
	Our contributions and main results are informally summarized below.
	\begin{enumerate}[label = {\bfseries (\arabic*)}]
		\item For optimization of non-convex optimization \cref{eq:obj}, we propose variants of Newton-MR (\cref{alg:NewtonMR_1st,alg:NewtonMR_2nd}), which under certain assumptions, provide the following guarantees, which have been shown to be optimal. A table comparing \cref{alg:NewtonMR_1st,alg:NewtonMR_2nd} with other state-of-the-art second-order methods can be found in \cref{table:complexity}.
		\begin{enumerate}[label = {\bfseries (1.\roman*)}]
			\item After at most $\mathcal{O}(\varepsilon^{-3/2})$ iterations, \cref{alg:NewtonMR_1st} guarantees approximate first-order optimality where $ \vnorm{\nabla f(\xx)} \leq \varepsilon $ (\cref{thm:complexity_first_alg_opt}). Furthermore, this is guaranteed in at most  $\mathcal{O}(\varepsilon^{-3/2})$ gradient and Hessian-vector product evaluations\footnote{Following \cite{royer2018complexity,royer2020newton}, we mainly consider gradient and Hessian-vector product evaluations as dominant costs in every iteration, which is henceforth referred to as operational complexity.}  (\cref{thm:complexity_first_Hv_opt}).
			\item After at most $\mathcal{O}(\varepsilon^{-3/2})$ iterations, \cref{alg:NewtonMR_2nd} guarantees, with probability one, approximate second-order optimality where $ \vnorm{\nabla f(\xx)} \leq \varepsilon $ and $ \lambda_{\min} \left(\nabla f(\xx)\right) \geq -\sqrt{\varepsilon}$ (\cref{thm:complexity_second_alg}). With high probability, we then obtain second-order operation complexity of order $\mathcal{O}(\varepsilon^{-7/4})$ (\cref{thm:complexity_second_Hv}), which is then improved to $\tilde{\mathcal{O}}(\varepsilon^{-3/2})$ for functions with a certain structural property in regions near saddle points (\cref{cor:complexity_second_Hv_saddle}).
		\end{enumerate}
		
		\item Of potential independent interests, we give a few novel properties of MINRES. 
		\begin{enumerate}[label = {\bfseries (2.\roman*)}]
			\item We build on the results of \cite{liu2022minres} and provide further refined theoretical guarantees of MINRES in terms of obtaining a NPC direction and certificate of positive (semi-)definiteness (or lack thereof) for $\nabla^2 f(\xx)$ (\cref{thm:T_indefinite,thm:H_PSD}). 
			\item We also perform a novel convergence analysis for MINRES, which improves upon the existing bounds in terms of dependence on the spectrum for indefinite matrices (\cref{lemma:MINRES_complexity_Sol}).
		\end{enumerate}
	\end{enumerate}
\end{informal*}

\begin{table}[htbp!]
	\centering
	\scalebox{0.75}{
	\begin{tabular}{|c|c|c|c|c|}
		\hline &&&& \\[-1em]
		\multirow{2}{*}{Algorithms} & \multicolumn{2}{|l|}{\hspace*{3cm} $ \| \ggk \| \leq \epsg $} & \multicolumn{2}{|l|}{\hspace*{2.5mm} $ \| \ggk \| \leq \epsg $ and $\lambda_{\min}(\HHk) \geq - \epsH $}  \\
		\cline{2-5} \\[-1em]
		& Iteration Comp. & Operation Comp. & Iteration Comp. & Operation Comp. \\
		\hline \hline &&&& \\[-1em]
		TR2 \cite{cartis2022evaluation} & $ \bigO{\epsg^{-2}} $ & --- & $ \bigO{\max \{\epsg^{-2}\epsH^{-1}, \epsH^{-3}\}} $ & --- \\ 
		\hline &&&& \\[-1em]
		TRACE\cite{curtis2017trust} & $ \bigO{\epsg^{-3/2}} $ & --- & $ \bigO{\max \{\epsg^{-3/2}, \epsH^{-3}\}} $ & --- \\ 
		\hline &&&& \\[-1em]
		I-TRACE\cite{curtis2022worst} & $ \bigO{\epsg^{-3/2}} $ & $ \tbigO{\epsg^{-3/2}} $ & --- & --- \\ 
		\hline &&&& \\[-1em]
		TR$\epsilon$\cite{curtis2021trust} & --- & --- & $ \tbigO{\max \{\epsg^{-2}\epsH, \epsH^{-3}\}} $ & $ \tbigO{\max \{\epsg^{-2}\epsH^{1/2}, \epsH^{-7/2}\}} $ \\ 
		\hline &&&& \\[-1em]
		LS$\epsilon$\cite{royer2020newton} & $ \bigO{\max \{\epsg^{-3}\epsH^{3}, \epsH^{-3}\}} $ & $ \tbigO{\max \{\epsg^{-3}\epsH^{5/2}, \epsH^{-7/2}\}} $ & $ \bigO{\max \{\epsg^{-3}\epsH^{3}, \epsH^{-3}\}} $ & $ \tbigO{\max \{\epsg^{-3}\epsH^{5/2}, \epsH^{-7/2}\}} $ \\
		\hline &&&& \\[-1em]
		Newton-MR & $ \bigO{\epsg^{-3/2}} $ & $\bigO{\epsg^{-3/2}}$ & $ \bigO{\max \{\epsg^{-3/2}, \epsH^{-3}\}} $ & $ \max \{ \tbigO{\epsg^{-3/2}}, \bigO{\epsH^{-7/2}} \} $ \\ 
		\hline &&&& \\[-1em]
		ARC(AR2) \cite{cartis2022evaluation} & $ \bigO{\epsg^{-3/2}} $ & --- & $ \bigO{\max \{\epsg^{-3/2}, \epsH^{-3}\}} $ & --- \\ 
		\hline
	\end{tabular}
	}
	\caption[Worst-case Complexity Comparison for Second-order Methods]{This table provides a comparison among a representative selection of second-order methods in terms of both the worst-case iteration complexity (overall iterations of the main algorithm) and the operation complexity (overall gradient and Hessian-vector product evaluations) for reaching the first- and second-order approximate optimality. The TR$\epsilon$ and LS$\epsilon$ refer to Newton-type methods whose underlying sub-problems contain an $\epsH$-dependent perturbation of the Hessian matrix, i.e., $ \HH + 2\epsH\eye$.}\label{table:complexity}
\end{table}

The rest of this paper is organized as follows. We end this section by introducing some definitions and notation used in this paper, as well as a a brief review of the complexity guarantees for some related works. In \cref{sec:MINRES_main}, we first discuss descent properties of the directions obtained from MINRES. Subsequently, we study in details the complexity of MINRES for obtaining such directions. We then introduce our Newton-MR variants in \cref{sec:NewtonMR}, accompanied by their convergence guarantees and complexity analysis. We then evaluate the empirical performance of our proposed algorithms in \cref{sec:exp}. Conclusions and further thoughts are gathered in \cref{sec:conc}.  To be self-contained, we bring some algorithmic details of  MINRES as well as its relevant properties in \cref{sec:MINRES_review}.

\subsubsection*{Notation and Definitions}
\label{sec:notation}
Throughout the paper, vectors and matrices are denoted by bold lower-case and bold upper-case letters, respectively, e.g., $ \vv $ and $ \VV $. 
We use regular lower-case and upper-case letters to denote scalar constants, e.g., $ \dn $  or $ L $. 
For a real vector, $ \vv $, its transpose is denoted by $ \vv^{\T} $. 
For two vectors $ \vv,\yy $, their inner-product is denoted by $ \dotprod{\vv, \yy}  = \vv^{\T} \yy $. For a vector $\vv$ and a matrix $ \VV $, $ \|\vv\| $ and $ \|\VV\| $ denote vector $ \ell_{2} $ norm and matrix spectral norm, respectively. 
The subscripts denote the iteration counter for the Newton-MR algorithm, while we use superscripts to denote the iterations of MINRES. For example, $ \sskt $ refers to the $ t
\th $ iterate of MINRES using $ \nabla f(\xxk) $ and $ \nabla^{2} f(\xxk)$ where $ \xxk $ denotes the $ k\th $ iterate of Newton-MR. 
A $ ^{[2]} $ as superscripts, e.g., $ \gamma_{t}^{[2]} $, represents the second edition of the element in MINRES iteration $ t $.
For notational simplicity, we use $ \gg(\xx) \triangleq \nabla f\left(\xx\right) \in \real^{\dn} $ and $ \HH(\xx) \triangleq \nabla^{2} f\left(\xx\right) \in \real^{\dn \times \dn} $ for the gradient and the Hessian of $ f $ at $ \xx $, respectively. At times, we may drop their dependence on $ \xx $ and/or the subscript, e.g., $ \gg = \ggk = \gg(\xxk) $ and $ \HH = \HHk = \HH(\xxk) $. 
For any $ t \geq 1 $, the Krylov sub-space of degree t generated using $ \ggk $ and $ \HHk $ is denoted by $\mathcal{K}_{t}(\HHk, \ggk)$. We also denote $ \rrkt = -\ggk - \HHk \sskt $ to be the residual corresponding to $ \sskt $. When the context is clear, we may also drop the dependence of MINRES iterates on $ \xxk $ and use $ \sst $ in lieu of $ \sskt $, and $ \rrt $ in lieu of $ \rrkt $.
%
%
Given two sets $ \mathcal{A} $ and $\mathcal{B}$, $ \mathcal{A} \setminus \mathcal{B}$ denotes the set subtraction $ \mathcal{A}  - \mathcal{B}  = \mathcal{A}  \cap \mathcal{B}^{\complement}$. The cardinality of a set $\mathcal{A}$ is denotes by $ |\mathcal{A}| $. Natural logarithm is denoted by $ \log(.) $. Logarithmic factors in the ``big-O'' notation are represented by $ \tbigO{.} $.

Clearly, the interplay between $ \HH $ and $ \gg $ is the main underlying factor affecting the overall convergence behavior of iterative methods based on the Krylov sub-space $ \Klvt(\HH,\gg) $. This interplay is entirely encapsulated in the notion of the grade of $ \gg $ with respect to $ \HH $; see \cite{saad2003iterative} for more details.
\begin{definition}[Grade of $ \gg $ w.r.t.\ $ \HH $]
	\label{def:grade}
	The grade of $ \gg $ with respect to $ \HH $ is the positive integer $ g \in \mathbb{N} $ such that 
	\begin{align*}
		\text{dim}\left(\mathcal{K}_{t}(\HH, \gg)\right) = \begin{cases}
			t, &\quad t \leq g, \\
			g, &\quad t > g.
		\end{cases}
	\end{align*}
\end{definition}

Since the Hessian matrix encodes information about the geometry of the optimization landscape, \cref{def:NPC} describes what is often referred to as NPC direction. This simply amounts to any direction that lies in the eigenspace corresponding to the nonpositive eigenvalues of $ \HH $.

\begin{definition}[Nonpositive Curvature direction]
	\label{def:NPC}
	Any non-zero $ \zz \in \real^{\dn} $ for which $ \dotprod{\zz, \HH \zz} \leq 0 $ is called a nonpositive curvature (NPC) direction.
\end{definition}

Under non-convex settings, we often seek to obtain solutions that satisfy certain approximate local optimality. Throughout this paper, we make use of the following popular notion of approximate optimality.
\begin{definition}[$(\varepsilon_g, \varepsilon_H)$-Optimality]
	\label{def:opt} 
	Given $\left(\varepsilon_g, \varepsilon_H \right) \in (0,1) \times (0,1)$, a point $\xx \in \real^{\dn} $ is an $(\varepsilon_g, \varepsilon_H)$-optimal solution to the problem~\eqref{eq:obj}, if 
	\begin{subequations}
		\label{eq:termination_second_order}  
		\begin{align}
			\vnorm{\gg(\xx)} &\le \epsg, \label{eq:small_g}\\
			\lambda_{\min} (\HH(\xx)) &\ge -\epsH. \label{eq:small_eig}
		\end{align}
	\end{subequations}
\end{definition}

\subsubsection*{Complexity for Non-convex Second-order Algorithms}
For optimization of \cref{eq:obj}, and in order to obtain approximate first-order optimality \cref{eq:small_g}, the iteration complexity of the classical Newton's method, when iterations are well-defined, has shown to be of the same order as that of steepest descent, namely $ \mathcal{O}(\epsg^{-2}) $ \cite{cartis2010complexity}.  In fact, an inexact trust-region method, which ensures at least a Cauchy (steepest-descent-like) decrease at each iteration, is shown to be of the same order \cite{gratton2008recursiveI,gratton2008recursiveII,blanchet2016convergence,gratton2017complexity,cartis2012complexity}. Recent non-trivial modifications of the classical trust-region methods have also been proposed which improve upon the iteration complexity to $ \mathcal{O}(\epsg^{-3/2}) $ \cite{curtis2017trust,curtis2017inexact,curtis2022worst}. These bounds can be shown to be tight \cite{cartis2010complexity} in the worst case. From the worst-case complexity point of view, cubic regularization methods have generally a better dependence on $ \epsg $ compared to the trust region algorithms. More specifically, \cite{nesterov2006cubic} showed that, under global Lipschitz continuity assumption on the Hessian, if the sub-problem is solved exactly, then the resulting cubic regularization algorithm achieves \cref{eq:small_g} with an iteration complexity of $ \mathcal{O}(\epsg^{-3/2})$. These results were extended by the seminal works of \cite{cartis2011adaptiveI,cartis2011adaptiveII} to an adaptive variant. In particular, the authors showed that the worst case complexity of $\mathcal{O}(\epsg^{- 3 / 2})$ can be achieved without requiring the knowledge of the Hessian's Lipschitz constant, access to the exact Hessian, or multi-dimensional global optimization of the sub-problem. These results were further refined in \cite{cartis2012complexity} where it was shown that, not only, multi-dimensional global minimization of the sub-problem is unnecessary, but also the same complexity can be achieved with mere one or two dimensional search. This $ \mathcal{O}(\epsg^{-3/2})$ bound has been shown to be tight \cite{cartis2011optimal}. 

To achieve approximate second-order optimality \cref{eq:termination_second_order}, iteration complexity bounds in the orders of $ \mathcal{O}(\max\{\epsH^{-1}\epsg^{-2}, \epsH^{-3}\}) $ and $ \mathcal{O}(\max\{\epsg^{-3}, \epsH^{-3}\}) $ have been obtained in the context of trust-region methods in \cite{cartis2012complexity} and \cite{grapiglia2016worst}, respectively. Similar bounds were also given in \cite{gratton2017complexity} under probabilistic model. Bounds of this order have shown to be optimal with certain choices of $ \epsg $ and $ \epsH $ \cite{cartis2012complexity}. For cubic regularization methods, \cite{cartis2012complexity} showed that at most $ \mathcal{O}(\max\{\epsg^{-2}, \epsH^{-3}\}) $ iterations is required to obtain \cref{eq:termination_second_order}. With further assumptions on the inexactness of the sub-problem solution, \cite{cartis2011adaptiveII,cartis2012complexity} also show that one can achieve the iteration complexity $ \mathcal{O}(\max\{\epsg^{-3/2}, \epsH^{-3}\}) $, which is shown to be tight for any choice of $ \epsg $ and $ \epsH $ \cite{cartis2010complexity}. Variants of adaptive cubic regularization under inexact function, gradient, and or Hessian information have been shown to achieve the same convergence rate \cite{xuNonconvexTheoretical2017,yao2018inexact,tripuraneni2017stochastic,kohler2017subsampledcubic,bellavia2021quadratic,bellavia2022adaptive}. Related algorithms and extensions with similar complexity guarantees have also been studied \cite{birgin2017use,martinez2017cubic,jiang2022cubic,birgin2017worst,bellavia2021adaptive,bellavia2019adaptive}.

Beyond trust-region and cubic-regularization frameworks, and most related to our work here, several line-search based methods have also been recently shown to achieve similar first and second-order complexity guarantees \cite{royer2020newton,royer2018complexity,royer2022nonlinear}. Better dependence on optimality tolerance can be obtained if one assumes additional structure, such as invexity \cite{roosta2018newton,liu2021convergence}.
In addition to iteration complexity, the computational cost of each iteration has also been considered in several recent works, e.g., \cite{royer2020newton,royer2018complexity,agarwal2017finding,yao2021inexact,curtis2021trust}. Taking $ \epsH^{2} = \epsg = \varepsilon  $, the overall operation complexities in these works achieve the rate of $ \tilde{\mathcal{O}}(\epsg^{-7/4}) $.

\section{MINRES: Newton-MR's Workhorse}
\label{sec:MINRES_main}
Recall that the original variant of Newton-MR, as studied in \cite{roosta2018newton}, relies on replacing the sub-problems of Newton-CG, i.e., the linear system in \cref{eq:linear_system}, with the ordinary least-squares formulation as
\begin{align}
	\label{eq:lieast-squares}
	\min_{\dd \in \real^{\dn}} \vnorm{\HHk \ss + \ggk}^{2}.
\end{align} 
In fact, the term ``MR'' refers to the sub-problems being in the form of \emph{\textbf{M}inimum \textbf{R}esidual}, i.e., least-squares. While there is a plethora of algorithms for approximately solving \cref{eq:lieast-squares}, it turns out that MINRES offers a catalog of beneficial properties that play a central role in optimization of \cref{eq:obj}. Among them, the most immediately relevant properties are those which guarantee descent for \cref{eq:obj} and are discussed in \cref{sec:MINRES_Descent}. Further applicable properties are gathered in \cref{sec:MINRES_property}.

\begin{algorithm}
	\caption{\textbf{MINRES($\HH,\gg,\eta$)}}
	\begin{algorithmic}[1]
		\label{alg:MINRES}
		\STATE \textbf{Input:} Hessian $ \HH $, gradient $\gg$, and inexactness tolerance $\eta > 0$
		\vspace{1mm}
		\STATE $ \phi_0 = \tbeta_1 = \vnorm{\gg} $, $ \rr_{0} = -\gg $, $ \vv_1 = \rr_{0} /\phi_0 $, $ \vv_0 = \ss_0 = \ww_0 = \ww_{-1} = \zero $,
		\vspace{1mm}
		\STATE $ s_0 = 0 $, $ c_0 = -1 $, $ \delta_1 = \tau_0 = 0 $, $ t = 1 $, $ \dt = \text{`SOL'}$,
		\vspace{1mm}
		\WHILE { \text{True} }
		\vspace{1mm}
		\STATE $ \qq_{t} = \HH \vv_{t}$, $\talpha_{t} = \vv^{\intercal}_{t} \qq_{t} $, $ \qq_{t} = \qq_{t} - \tbeta_{t} \vv_{t-1} $, $ \qq_{t} = \qq_{t} - \talpha_{t} \vv_{t} $, $ \tbeta_{t+1} = \vnorm{\qq_{t}} $
		\vspace{2mm}
		\STATE  $\displaystyle 
		\begin{bmatrix}
			\delta^{[2]}_{t} & \epsilon_{t+1} \\
			\gamma_{t} & \delta_{t+1}
		\end{bmatrix} = \begin{bmatrix}
			c_{t-1} & s_{t-1} \\
			s_{t-1} & -c_{t-1}
		\end{bmatrix} \begin{bmatrix}
			\delta_{t} & 0 \\
			\talpha_{t} & \tbeta_{t+1}
		\end{bmatrix}$
		\vspace{2mm}
		\IF{$ c_{t-1} \gamma_{t} \geq 0 $}  \label{alg:MINRES:NPC}
		\vspace{1mm}
		\STATE $ \dt = \text{`NPC'}$
		\vspace{1mm}
		\RETURN $ \rr_{t-1} $, $ \dt $.
		\vspace{1mm}
		\ENDIF
		\vspace{1mm}
		\IF{$ \phi_{t-1} \sqrt{\gamma_{t}^{2} + \delta_{t+1}^{2}} \leq \eta \sqrt{\phi_{0}^{2} - \phi_{t-1}^{2}} $} \label{alg:MINRES:SOL}
		\vspace{1mm}
		\STATE $ \dt = \text{`SOL'}$
		\vspace{1mm}
		\RETURN $ \ss_{t-1} $, $ \dt $
		\vspace{1mm}
		\ENDIF
		\vspace{1mm}
		\STATE $ \gamma_{t}^{[2]} = \sqrt{\gamma_{t}^2 + \tbeta_{t+1}^2} $,
		\vspace{1mm}
		\IF{ $ \gamma^{[2]}_{t} \neq 0 $ }
		\vspace{1mm}
		\STATE $c_{t} = \gamma_{t} / \gamma_{t}^{[2]} $, $ s_{t} = \tbeta_{t+1} / \gamma_{t}^{[2]} $, $ \tau_{t} = c_{t} \phi_{t-1} $, $ \phi_{t} = s_{t} \phi_{t-1} $, 
		\vspace{1mm}
		\STATE $ \ww_{t} = \left(\vv_{t} - \delta^{[2]}_{t} \ww_{t-1} - \epsilon_{t} \ww_{t-2} \right) / \gamma^{[2]}_{t} $, $ \ss_{t} = \ss_{t-1} + \tau_{t} \ww_{t} $
		\vspace{1mm}
		\IF{$ \tbeta_{t+1} \neq 0 $}
		\STATE $ \vv_{t+1} = \qq_{t} / \tbeta_{t+1} $, $ \rr_{t} = s_{t}^2 \rr_{t-1} - \phi_{t} c_{t} \vv_{t+1} $,
		\ENDIF
		\ELSE
		\vspace{1mm}
		\STATE $ c_{t} = 0 $, $ s_{t} = 1 $, $ \tau_{t} = 0 $, $ \phi_{t} = \phi_{t-1} $, $ \rr_{t} = \rr_{t-1} $, $ \ss_{t} = \ss_{t-1} $,  
		\vspace{1mm}
		\ENDIF
		\vspace{1mm}
		\STATE $ t \leftarrow t+1 $,
		\vspace{1mm}
		\ENDWHILE
	\end{algorithmic}
\end{algorithm}

\subsection{Descent Properties}
\label{sec:MINRES_Descent}
For solving \cref{eq:lieast-squares}, recall that the $ t\th $ iteration of MINRES, detailed in \cref{alg:MINRES},  yields a direction, $ \sskt $, which satisfies
\begin{align}
	\label{eq:MINRES}
	\sskt = \argmin_{\ss \in \mathcal{K}_{t}(\HHk, \ggk)}  \vnorm{\HHk \ss + \ggk}^{2}, \quad t = 1,2,\ldots,
\end{align}
where $ \mathcal{K}_{t}(\HHk, \ggk) = \text{Span}\{\ggk,\HHk \ggk, \ldots, \left[\HHk\right]^{t-1} \ggk \}$ denotes the Krylov sub-space of order $ t $ \cite{golub2013matrix}.
Suppose $ \ggk \neq \zero $ (we will treat the case for saddle points later). Since $ \bm{0} \in \mathcal{K}_{t}(\HHk, \ggk)$, it follows that 
\begin{align*}
	\dotprod{\sskt,\HHk \ggk} \leq - \vnorm{\HHk \sskt}^{2}/2 < 0, \quad t = 1,2,\ldots,
\end{align*} 
as long as $ \sskt \notin \Null(\HHk) $ (which can be guaranteed if, for example, $ \ggk \notin \Null(\HHk) $).
In other words, $ \sskt$ from MINRES, for all $ t \geq 1 $, is a descent direction for $ \vnorm{\gg(\xx)}^{2} $ at $ \xxk $. 
This observation forms the basis of the plain Newton-MR variant considered in \cite{roosta2018newton}. Due to such descent property, once $ \sskt $ satisfies a certain condition, the iterations of MINRES are terminated, and the update direction $ \ddk = \sskt $ is used within an Armijo-type line-search to obtain a step-size, guaranteeing decrease in $ \vnorm{\gg(\xx)}^{2} $.

By its construction, the iterates of such a na\"{i}ve variant of Newton-MR converge to the zeros of the gradient field, $ \gg(\xx) $. However, unless the objective function $ f(\xx) $ is invex \cite{mishra2008invexity}, such points include saddle points or even (local) maxima. 
To design a new Newton-MR variant for more general non-convex optimization, a natural question perhaps is ``\emph{when can MINRES generate a descent direction for the objective function, $ f(\xx) $?}''. This question has very recently been thoroughly studied in \cite{liu2022minres}. The answer is based on the inherent ability of MINRES to detect NPC directions, if they exist, and to otherwise provide a certificate for positive semi-definiteness of the Hessian matrix. In fact, MINRES provides two complementary mechanisms to generate descent directions for $ f(\xx) $: one prior to the detection of non-positive curvature and one based on the detected NPC direction.

\subsubsection*{Prior to Detecting NPC}
As long as no NPC direction has been detected, any iterate of MINRES is guaranteed to yield descent for $ f(\xx) $ (\cref{lemma:MINRES_NPC_detector,lemma:sTr}). For such iterates, we have
\begin{align*}
	\dotprod{\sskt, \ggk} < 0, \quad \text{and} \quad \dotprod{\sskt, \HHk \ggk} < 0.
\end{align*}
For example, if $ \HHk \succ \zero$, all iterates can be used to obtain descent, not only in the norm of its gradient, but also in the objective function itself. 
As a result, prior to the detection of any NPC direction, if a certain inexactness condition is met, the MINRES iterations can be terminated, returning an inexact solution as an update direction.
\begin{SOL*}
	The most widely used termination criterion for the sub-problem is that based on the residual, i.e., we terminate when $ \|\rrkt\| \leq \eta $ for some $ \eta > 0 $. However, in non-convex, if $ \ggk \notin \range(\HHk) $, then $\|\rrkt\|$ is lower bounded by $ \|(\eye - \HHk\HHdk)\ggk\| $, which is a priori unavailable. Hence, the acceptable range for $ \eta $ relies on an unavailable lower-bound, which is often unfortunately overlooked or ignored in many related works, making this approach rather impractical. In \cref{alg:MINRES}, at any iteration $ t $, we instead check the inexactness condition 
	\begin{align}
		\label{eq:inexactness}
		\vnorm{\HHk \rrktp} \leq \eta \vnorm{\HHk \ssktp},
	\end{align}
	for some $ \eta > 0 $ with the iterate from the \emph{previous} iteration, and return $ \ssktp$ if \cref{eq:inexactness} is satisfied. Note that, in \cref{alg:MINRES}, $ \|{\HHk \rrkt}\| $ is decreasing to zero while $ \|{\HHk \sskt}\| $ is monotonically increasing from zero to $ \|{\HHk \HHdk \ggk}\| $ \cite[Lemma 3.11]{liu2021convergence}, and hence for any $ \eta > 0 $, the inexactness condition \cref{eq:inexactness} will always be satisfied at some iteration. 
	Furthermore, \cref{eq:inexactness} can be verified without any additional matrix vector products; see \hyperref[alg:MINRES:SOL]{Step 7} of \cref{alg:MINRES} as well as \cref{eq:Hs}.
\end{SOL*}

\subsubsection*{After Detecting NPC} 
Once a NPC direction arises for the first time, say at iteration $ t $, the subsequent iterates cease to enjoy a descent guarantee for $ f $, and may in fact even yield ascent (although they always provide descent for $ \vnorm{\gg(\xx)}^{2} $). However, it turns out that the detected NPC direction, $ \rrktp $, is itself naturally a descent direction for $ f(\xx) $; see \cref{eq:rTg} in \cref{lemma:MINRES_properties}. In addition, $ \rrktp $ forms an interesting angle with the gradient of $ \vnorm{\gg(\xx)}^{2} $ at $ \xxk $; see \cref{eq:rTHr_conjugated} in \cref{lemma:MINRES_properties}. Indeed, with such an NPC direction, we have
\begin{align*}
	\dotprod{\rrktp, \ggk} < 0, \quad \text{and}  \quad 
	\left\{\begin{array}{lll}
		\dotprod{\rr_{k}^{(0)}, \HHk \ggk} \leq 0,& &  \text{(if NPC is detected at $t = 1$)}\\
		\\
		\dotprod{\rrktp, \HHk \ggk} = 0,&  & \text{(if NPC is detected at $t > 1$)}
	\end{array}\right..
\end{align*} 

\begin{NPC*}
	To determine whether or not a NPC direction is available at iteration $ t $ for MINRES, it suffices to monitor the condition 
	\begin{align}
		\label{eq:NPC}
		\dotprod{\rrktp, \HHk \rrktp} \leq 0,
	\end{align}
	where $ \rrktp = -\ggk - \HHk \ssktp $ is the residual vector of the \emph{previous} iteration. It also turns out that the NPC condition \cref{eq:NPC} can be readily verified as part of the MINRES iterations without explicitly computing $ \HHk \rrktp $; see \hyperref[alg:MINRES:NPC]{Step 11} of \cref{alg:MINRES}, as well as \cref{eq:NPC_cond,lemma:MINRES_properties}. In fact, the existence of a NPC direction for $ \HHk $ at iteration $ t $ is equivalent to the NPC condition \cref{eq:NPC}; see \cref{lemma:MINRES_NPC_detector}.
\end{NPC*}
\begin{remark}[Contrast Between NPC Detection in MINRES and Capped CG of \cite{royer2020newton}]
	Within the Capped CG procedure, \cite[Algorithm 1]{royer2020newton}, if a NPC direction does not arise naturally, an additional test is employed to extract such a direction. This is done by comparing the observed convergence rate of the residual norm of the iterates with the theoretical one under the positive definiteness assumption. A slower convergence rate than what is theoretically expected serves as an indication that the underlying matrix contains non-positive eigenvalues. This is then exploited to extract a NPC direction. Similar approaches are also used in \cite[Lemma 3.1, Lemma 3.2]{curtis2021trust} for the Truncated CG solver within the trust region framework. In contrast, the NPC directions always arise naturally within the MINRES iterations, as soon as they become available; see \cref{lemma:MINRES_NPC_detector}. In this light, monitoring the NPC condition \cref{eq:NPC} is necessary and sufficient to detect, or to otherwise rule out, the existence of directions.
\end{remark}

As discussed above, one can always construct a descent direction within the MINRES iterations. Indeed, if the inexactness condition \cref{eq:inexactness} is satisfied for $ \ssktp $ before the NPC condition \cref{eq:NPC} is detected, we let $ \ddk = \ssktp $. Otherwise, once the NPC direction is detected at iteration $ t $, we set $ \ddk = \rrktp $. In either case, we are guaranteed that the search direction $ \ddk $ satisfies
\begin{align}
	\label{eq:sufficient_decrease_dTg}
	\dotprod{\ddk, \ggk} < 0, \quad \text{and} \quad \dotprod{\ddk, \HHk \ggk} \leq 0.
\end{align}

\begin{remark}[Contrast Between CG and MINRES Directions]
	One can make several intriguing observations about the quality of the search directions obtained from MINRES in comparison with those from CG.
	\begin{itemize}
		\item Just as is the case for the classical CG, the output of the Capped CG procedure in \cite{royer2020newton} can be a direction of ascent for $ \vnorm{\gg}^{2} $. This could result in a, rather, chaotic trajectory for the iterates of Newton-CG in terms of the magnitude of the gradient. This is in fact reminiscent of the non-monotonic behavior of the residuals during the iterations of CG for solving a positive definite linear system. 
		In contrast, by constructing a search direction from MINRES as in \cref{eq:sufficient_decrease_dTg}, we obtain a direction that not only implies descent in the function, $ f(\xx) $, but it also forms a ``friendly'' angle with the direction of the greatest decrease of $\vnorm{\gg(\xx)}^{2}$.     
		
		\item   Recall that the $ t\th $ iteration of CG is equivalent to the formulation
		\begin{align*}
			\argmin_{\ss \in \mathcal{K}_{t}(\HHk, \ggk)} \dotprod{\ggk, \ss}  + \hf \dotprod{ \ss, \HHk \ss}.
		\end{align*}
		As long as non-positive curvature of $ \HHk $ has not been encountered, any solution $ \sskt $ to above optimization problem is guaranteed to satisfy $ \langle \ggk, \sskt \rangle  \leq -\langle \sskt, \HHk \sskt\rangle/2 \leq 0 $, and is hence a descent direction. 
		Similarly, as discussed above, any iterate of MINRES prior to the detection of non-positive curvature, provides the same guarantee. However, as shown in \cite{liu2022minres} (see also \cref{lemma:sTr}), for such an iterate, we also have $ \langle \ggk, \sskt \rangle  \leq -\langle \sskt, \HHk \sskt \rangle \leq 0 $.
		In other words, the directions from MINRES could yield sharper descent as compared with CG.
	\end{itemize}
\end{remark}

\subsection{Complexity Analysis}
\label{sec:MINRES_complexity}
In this section, we study the complexity of MINRES in obtaining a direction satisfying the conditions \cref{eq:inexactness} and \cref{eq:NPC}. For notational simplicity, we drop the dependence of $ \gg(\xxk) $ and $ \HH(\xxk) $ on $ \xxk $ and denote $ \gg = \ggk = \gg(\xxk) $ and $ \HH = \HHk = \HH(\xxk) $. Subsequently, we use $ \sst $ and $ \rrt $ to denote $ \sskt $ and $ \rrkt  $, respectively.

It turns out that the interplay between $ \HH $ and $ \gg $ is paramount in establishing the complexity of MINRES. For example, if $ \gg $ lies entirely in some eigen-space corresponding to certain eigenvalues of $ \HH $, the induced Krylov sub-space will never contain any of the remaining eigenvectors of $ \HH $. In other words, we only need to consider the portion of the spectrum of $ \HH $ that is relevant to $ \gg $.
\begin{definition}[$\gg$-Relevant Eigenvalues of $ \HH $]
	\label{def:eigpairs}
	Consider the set of eigenvalues of $ \HH $, i.e., $ \Theta(\HH) \defeq \left\{\lambda \in \real \mid \det(\HH - \lambda \eye) = 0\right\} $ and recall the eigenspace corresponding to an eigenvalue $ \lambda \in \Theta(\HH) $, i.e., $ \mathcal{E}_{\lambda}(\HH) \defeq \Null(\HH - \lambda \eye) $. The set of $\gg$-relevant eigenvalues of $ \HH $ is defined as 
	\begin{align*}
		\Psi(\HH, \gg) \defeq \left\{\lambda \in \Theta(\HH) \mid \gg \not\perp \mathcal{E}_{\lambda}(\HH)\right\}.
	\end{align*}
\end{definition}
%
In other words, $ \Psi(\HH, \gg) $ is the set of all eigenvalues of $ \HH $ whose associated eigenvectors are not orthogonal to $ \gg $.
Let $ \psi \defeq |\Psi(\HH, \gg)| = \psip + \psin + \psiz$ where $ \psip $, $\psin$, and $\psiz$, respectively, denote the number of positive, negative, and zero $\gg$-relevant eigenvalues of $ \HH $ (clearly, $ \psiz \in \{0,1\} $). 
Note that $ \Theta(\HH) \setminus \Psi(\HH, \gg) $ may contain positive, negative, and/or zero eigenvalues.
Without loss of generality, we label the $\gg$-relevant eigenvalues in $ \Psi(\HH, \gg) $ such that $ \lambda_1 > \lambda_2 > \ldots > \lambda_{\psip} > 0$, and  $0 > \lambda_{\psip+\psiz+1} > \ldots > \lambda_{\psi} $. 
%

For a given $ \lambda_{i} \in \Psi(\HH,\gg), \; 1 \leq i \leq \psi $, let $ \dim(\mathcal{E}_{\lambda_{i}}(\HH)) = \mi $ where $ 1 \leq \mi \leq d $. If $\mi > 2$, we can always consider a basis of $ \mathcal{E}_{\lambda_{i}}(\HH) $ for which $ \gg $ has a non-zero projection onto \emph{only one} of the basis vectors, i.e., for any orthonormal basis $ \WW_{i} \in \real^{\dn \times \mi} $ of $\mathcal{E}_{\lambda_{i}}(\HH)$, we can consider the unit eigenvectors given by $ \uu_{i}^{\{1\}} = \WW_{i}\WW_{i}^{\T} \gg/\vnorm{\WW_{i}\WW_{i}^{\T} \gg} $ and its non-zero orthogonal complements $ \uu_{i}^{\{2\}}, \ldots, \uu_{i}^{\{\mi\}} $ in $\mathcal{E}_{\lambda_{i}}(\HH)$. Using $ \uu_{i} $ instead of $ \uu_{i}^{\{1\}} $ for notational simplicity, we define 
\begin{subequations}
	\label{eq:H_decomp_PSD} 
	\begin{align}
		\label{eq:U_i}
		\UU_{i} &= \begin{bmatrix}
			\uu_{i} & \uu_{i}^{\{2\}} & \cdots & \uu_{i}^{\{\mi\}}
		\end{bmatrix}, \quad i = 1,\ldots,\psi.
	\end{align} 
	Putting the above all together, we will consider the full eigenvalue decomposition of $ \HH $ as 
	\begin{align}
		\HH &= \begin{bmatrix}
			\UU & \UUp
		\end{bmatrix} \begin{bmatrix}
			\Lambda & \\
			& \Lambda_{\perp}
		\end{bmatrix} \begin{bmatrix}
			\UU & \UUp
		\end{bmatrix}^{\T},
	\end{align}
	where
	\begin{align}
		\UU &= \begin{bmatrix}
			\UU_{+} & \UU_{-} 
		\end{bmatrix}, \quad \Lambda = \begin{bmatrix}
			\Lambda_{+} & \\
			& \Lambda_{-} 
		\end{bmatrix},
	\end{align}
	and $ \Lambda_{+} $ and $ \Lambda_{-}, $ are the diagonal matrices respectively containing the strictly positive and strictly negative $\gg$-relevant eigenvalues in $ \Psi(\HH, \gg) $ (with their multiplicities), with $ \UU_{+} $ and $ \UU_{-} $ being their corresponding eigenvectors as discussed above, i.e., 
	\begin{align}
		\label{eq:U_p}
		\UU_{+} = \begin{bmatrix}
			\UU_{1} & \UU_{2} & \cdots & \UU_{\psip},
		\end{bmatrix}, \quad \UU_{-} = \begin{bmatrix}
			\UU_{\psip+\psiz+1} & \UU_{\psip+\psiz+1} & \cdots & \UU_{\psip},
		\end{bmatrix}
	\end{align}
	where $ \UU_{i} $ is as in \cref{eq:U_i}. 
\end{subequations}
Also, $ \Lambda_{\perp} $ contains all the remaining eigenvalues of $ \HH $ (with multiplicities) with the corresponding eigenvectors collected in $ \UUp $. It follows that the decomposition of $ \gg $ in the basis $\begin{bmatrix} \UU & \UUp \end{bmatrix}$ is given by 
\begin{align}
	\label{eq:H_lambda_U}
	\gg = \begin{bmatrix}
		\UU & \UUp
	\end{bmatrix} \begin{bmatrix}
		\UU & \UUp
	\end{bmatrix}^{\T} \gg = \sum_{i=1}^{\psi} \dotprod{\uu_{i},\gg} \uu_{i},
\end{align}
where $ \uu_{i} $ is as in \cref{eq:U_i}.

Let us also define a few quantities related to \cref{eq:H_decomp_PSD}, which we use in the proofs of this paper. For some $ 1 \leq i \leq \psip $ and $ \psip+\psiz+1 \leq j \leq \psi $, we partition $ \Lambda_{+} $, $ \UU_{+} $, $ \Lambda_{-} $, and $ \UU_{-} $ in \cref{eq:H_decomp_PSD} as
\begin{subequations}
	\label{eq:H_decomp_PSD_ij}
	\begin{align}
		\Lambda_{+} &= \begin{bmatrix}
			\Lambda_{i+} & \\
			& \Lambda_{i+}^{\complement}
		\end{bmatrix}, \quad \UU_{+} = \begin{bmatrix}
			\UU_{i+} & \UU_{i+}^{\complement}
		\end{bmatrix} \\
		\Lambda_{-} &= \begin{bmatrix}
			\Lambda_{j-}^{\complement} & \\
			& \Lambda_{j-}
		\end{bmatrix}, \quad \UU_{-} = \begin{bmatrix}
			\UU_{j-}^{\complement} & \UU_{j-}
		\end{bmatrix}.
	\end{align}
\end{subequations}
In other words, $ \Lambda_{i+} $ is the diagonal sub-matrix of $ \Lambda_{+} $ in \cref{eq:H_decomp_PSD} associated with the largest positive $\gg$-relevant eigenvalues $ \lambda_{1} \geq \lambda_{1} \geq \ldots \geq \lambda_{i} $. The corresponding columns of $ \UU_{+} $ are gathered in $ \UU_{i+} $. Similarly, $ \Lambda_{j-} $ is the diagonal sub-matrix of $ \Lambda_{-} $ in \cref{eq:H_decomp_PSD} associated with the most negative $\gg$-relevant eigenvalues $ \lambda_{j} \geq \ldots \geq \lambda_{\psi} \geq \lambda_{\psi} $, with $ \UU_{j-} $ containing the corresponding eigenvectors.

\cref{lemma:Lanczos_grade} establishes a crucial connection between the grade of $ \gg $ with respect to $ \HH $ defined in \cref{def:grade} and the cardinality of $ \Psi(\HH,\gg) $.
\begin{lemma}
	\label{lemma:Lanczos_grade}
	We have $ g = \psip + \psin + \psiz $, where $ g $ be the grade of $ \gg $ with respect to $ \HH $, and $ \psip $, $\psin$, and $\psiz$, respectively, denote the number of positive, negative, and zero $\gg$-relevant eigenvalues of $ \HH $.
\end{lemma}
\begin{proof}
	The proof follows a similar line of reasoning as that in \cite[Theorem 2.6.2]{bjorck2015numerical}. 
	Let $\uu_{i}$ denote the eigenvector corresponding to $ \lambda_{i} \in \Psi(\HH,\gg) $ as described above. From \cref{eq:H_lambda_U}, for $ j\geq 0 $, we have 
	\begin{align*}
		\zz_{j} \defeq \HH^{j} \gg = \sum_{i=1}^{\psi} \lambda_{i}^{j} \dotprod{\uu_{i},\gg} \uu_{i},
	\end{align*}
	where we have defined $ 0^0 = 1 $.
	Denoting $ \yy_{i} \defeq \dotprod{\uu_{i},\gg} \uu_{i} $, we have
	\begin{align}
		\label{eq:vandermonde}
		\begin{bmatrix}
			\zz_{0} & \zz_{1} & \ldots & \zz_{\psi-1} 
		\end{bmatrix} = \overbrace{\begin{bmatrix}
				\yy_{1} & \yy_{2} & \ldots & \yy_{\psi} 
		\end{bmatrix}}^{\YY_{\psi}} \begin{bmatrix}
			1 & \lambda_{1} & \cdots & \lambda_{1}^{\psi-1}  \\
			1 & \lambda_{2} & \cdots & \lambda_{2}^{\psi-1} \\
			\vdots & \vdots & \vdots & \vdots \\
			1 & \lambda_{\psi} & \cdots & \lambda_{\psi}^{\psi-1}
		\end{bmatrix} 
	\end{align}
	Since $ \lambda_{i} \neq \lambda_{j},\; i \neq j $, the Vandermonde matrix is non-singular. The matrix $ \YY_{\psi} $ is also non-singular, and hence it follows that $ g = \psi = \psip + \psin + \psiz $. 
\end{proof}

Recall that, as part of MINRES, we have $ \TTt = \VVtT \HH \VVt $ where $ \TTt \in \real^{t \times t}$ is a symmetric tridiagonal matrix and  $ \VVt \in \real^{d \times t} $ is the basis obtained from the Lanczos process for $ \Klvt(\HH,\gg) $; see \cref{sec:MINRES_review} for more details. \cref{lemma:T_PSD} characterizes the implications of $ \psiz \in \left\{0,1\right\} $ on the structure of $ \TTg $.
\begin{lemma}
	\label{lemma:T_PSD}
	Let $ g $ be the grade of $ \gg $ with respect to $ \HH $.
	\begin{enumerate}[label = {\bfseries (\roman*)}]
		\item If $ \gg \in \range(\HH) $, i.e., $ \psiz = 0 $, then $ \TT_{g} $ is non-singular and $ \rr_{g} = \zero $. Furthermore, if the NPC condition \cref{eq:NPC} is never detected for any $ 1 \leq t \leq g $, then $ \TTg \succ \zero $.
		\item \label{lemma:T_PSD_ii} If $ \gg \notin \range(\HH) $, i.e., $ \psiz = 1 $, then $ \TT_{g} $ is singular. Furthermore, if the NPC condition \cref{eq:NPC} is first detected only at the very last iteration, then $ \TTg \succeq \zero $, and $ \rr_{g-1} = \rr_{g} $ is a zero curvature direction.
	\end{enumerate}
\end{lemma}
\begin{proof}   
	\hfill
	\begin{enumerate}[label = {\bfseries (\roman*)}]
		\item When $ \gg \in \range(\HH) $, we obviously have $ \rrg = \zero $. Also, by \cite[Appendix A]{choi2011minres}, we must have $ \TT_{g} $ is non-singular. Further, if the NPC condition \cref{eq:NPC} is never detected for any $ 1 \leq t \leq g $, then from \cref{lemma:MINRES_NPC_detector}, we also have $ \TT_{g} \succ \zero $.
		
		\item When $ \gg \notin \range(\HH) $, by \cite[Throem 3.2]{choi2011minres}, we have $ \TT_{g} $ is necessarily singular and $ \gamma^{[2]}_{g} = 0 $. This implies that $ \gamma_{g} = 0 $ and hence $ \rr_{g-1} $ must be a zero curvature direction, i.e., $ \dotprod{\rr_{g-1}, \HH \rr_{g-1}} =0 $. Since the NPC condition \cref{eq:NPC} is only detected for the first time at the last iteration, then by \cref{lemma:MINRES_NPC_detector} we have $ \TT_{g-1} \succ \zero$. But this, in turn, implies that $ \dotprod{\rr_{t}, \HH \rr_{t}} > 0$ for all $0 \leq t < g-1 $ (recall that $\rr_{t} \in \Klv_{g-1}(\HH, \gg), \; 0 \leq t \leq g-2 $). Now, assume $ \lambda_{\min}(\TT_{g}) < 0 $. Then, there must exist some non-zero vector $ \pp \in \Klv_{g}(\HH, \gg) = \Span\{\rr_{0},\rr_{1},\ldots, \rr_{g-1}\} $ such that $ \dotprod{\pp, \HH \pp} < 0 $. Let $ \xi_{1}, \ldots, \xi_{g-1} $ be scalars such that $ \pp = \sum_{i=0}^{g-1} \xi_{i} \rr_{i} $. Note that since $\dotprod{\rr_{g-1}, \HH \rr_{g-1}} = 0$, at least one of $ \xi_{1}, \ldots, \xi_{g-2} $ must be non-zero. Otherwise $ \pp = \rr_{g-1} $ and $\dotprod{\rr_{g-1}, \HH \rr_{g-1}} < 0$, which contradicts the fact that $ \rr_{g-1} $ is a zero curvature direction. By \cref{eq:rTHr_conjugated}, we have
		\begin{align*}
			0 > \pp^{T} \HH \pp = \sum_{i=0}^{g-1} \sum_{j=0}^{g-1} \xi_{i} \xi_{j} \dotprod{\rr_{i}, \HH \rr_{j}} =  \sum_{i=0}^{g-1} \xi_{i}^2 \dotprod{\rr_{i}, \HH \rr_{i}} =  \sum_{i=0}^{g-2} \xi_{i}^2 \dotprod{\rr_{i}, \HH \rr_{i}} > 0,
		\end{align*}
		which is a contradiction. Hence, we must have $ \TTg \succeq \zero $. By \cite[Remark 3.10-(ii)]{liu2022minres}, we also have $ \rr_{g-1}= \rr_{g} $.
	\end{enumerate}
\end{proof}

\cref{lemma:T_PSD} indicate that if $ \psiz = 1 $, the NPC condition \cref{eq:NPC} will be detected at the last iteration. However, if in addition $ \psin \geq 1 $, we expect to detect the NPC condition \cref{eq:NPC} earlier in the iterations. It turns out considering the interplay between $ \gg $ and $ \HH $ through the lens of $\gg$-relevant eigenvalues is paramount in establishing the complexity of MINRES for returning a direction satisfying the NPC condition \cref{eq:NPC}. For example, if $ \psin = \psiz = 0 $ and $ \psip \geq 1 $, i.e., $ \Psi(\HH, \gg) $ only contains positive eigenvalues of $ \HH $, then the non-positive curvature of $ \HH $ is simply not visible through the lens of $ \gg $.  However, when $ \psin + \psiz \geq 1 $, i.e., $ \Psi(\HH, \gg) $ contains at least one non-positive eigenvalue of $ \HH $, then we are able to detect a NPC direction. For example, when $ \psin = \psip = 0 $, then $ \gg \in \Null(\HH) $, which implies $ \dotprod{\gg,\HH\gg} = 0 $, i.e., $ \gg $ is a NPC direction, which is detected at the very first iteration. \cref{thm:T_indefinite} characterizes this property of MINRES.

\begin{theorem}
	\label{thm:T_indefinite}
	The NPC condition \cref{eq:NPC} is detected for some $ 1 \leq t \leq g $ if and only if $ \psin + \psiz  \geq 1 $. Furthermore, the NPC condition \cref{eq:NPC} is detected with strict inequality for some $ 1 \leq t \leq g $ if and only if $ \psin \geq 1 $. 
\end{theorem}
\begin{proof}
	By \cref{lemma:Lanczos_grade}, we recall that $ g = \psi = \psip + \psin + \psiz$. 
	Suppose $ \psin + \psiz  \geq 1 $. Let $1 \leq t \leq g$ be the first iteration where $ \Klvt(\HH,\gg) $ contains an eigenvector, say $ \uu $, corresponding to a non-positive eigenvalue in $\Psi(\HH,\gg)$. From \cref{eq:vandermonde}, it is clear that such an iteration exists. Let $ \VVt \in \real^{d \times t} $ denote the basis obtained from the Lanczos process for $ \Klvt(\HH,\gg) $. Hence, we can write $ \uu = \VVt \ww $ for some non-zero $ \ww \in \real^{t} $.
	It follows that 
	\begin{align*}
		\dotprod{\ww, \TTt \ww} = \dotprod{\VVt \ww, \HH \VVt \ww} = \dotprod{\uu, \HH \uu} \leq 0,
	\end{align*}
	which implies that $ \TT_{t} \not\succ \zero $. In other words, by \cref{lemma:MINRES_NPC_detector}, as soon as $ \uu \in \Klvt(\HH,\gg) $, the NPC condition \cref{eq:NPC} is satisfied.
	
	Conversely, suppose $ \psin = \psiz  = 0 $. Clearly, no eigenvector corresponding to any non-positive eigenvalue of $ \HH $ can belong to $ \Klvt(\HH,\gg) $ for any $ 1 \leq t \leq g$. Indeed, suppose for some $ \lambda \leq 0 $, its corresponding eigenvector, $ \vv $, and some $ 1 \leq t \leq g $, we have 
	\begin{align*}
		\vv = \sum_{i = 0}^{t-1} c_{i}\HH^{i} \gg.
	\end{align*}
	In this case, since $ \lambda \notin \Phi(\HH,\gg) $, we must have 
	\begin{align*}
		\dotprod{\vv,\vv} = \sum_{i = 0}^{t-1} c_{i} \dotprod{\vv,\HH^{i} \gg} = \sum_{i = 0}^{t-1} c_{i} \lambda^{i} \dotprod{\vv, \gg} = 0,
	\end{align*}
	which contradicts the assumption that $ \vv $ is an eigenvector and hence non-zero\footnote{More generally, it can be shown that, the eigenvectors corresponding to any $ \lambda \in \Theta(\HH) \setminus \Psi(\HH,\gg) $ will not belong to $ \Klvt(\HH,\gg) $ for any $ 1 \leq t \leq g$.}. Now, from \cref{eq:vandermonde}, it follows that $ \range(\YY_{\psi}) = \Span\{\uu_{1},\ldots,\uu_{\psi}\} = \Klv_{g}(\HH,\gg) $ where $\uu_{i}$'s are as in \cref{eq:H_lambda_U}. 
	For any $ 1 \leq t \leq g $, consider $ \VVt \in \real^{d \times t} $ to be the basis obtained from the Lanczos process for $ \Klvt(\HH,\gg) $. For any non-zero $ \yy \in \real^{t} $, let $ \zz = \VVt \yy \in \Klvt(\HH,\gg) \subseteq \Klv_{g}(\HH,\gg) $ and suppose $ \xi_{1}, \ldots, \xi_{\psi} $ are some scalars such that $ \zz = \sum_{i=1}^{\psi} \xi_{i} \uu_{i} $. We have 
	\begin{align*}
		\dotprod{\yy, \TTt \yy} &= \dotprod{\VVt \yy, \HH \VVt \yy} = \dotprod{\zz, \HH \zz} \\
		&= \sum_{i=1}^{\psi} \sum_{j=1}^{\psi} \xi_{i} \xi_{j} \dotprod{\uu_{i}, \HH \uu_{j}} =  \sum_{i=1}^{\psi} \xi_{i}^{2} \lambda_{i} \dotprod{\uu_{i}, \uu_{i}} > 0,
	\end{align*}
	which implies $ \TT_{t} \succ \zero$.
	
	Finally, if instead of $ \psin + \psiz  \geq 1 $, we have $ \psin \geq 1 $, i.e., we have at least one negative $\gg$-relevant eigenvalue, then the proof can be modified accordingly.
\end{proof}

Clearly $\gg \notin \range(\HH)$ if and only if $ \psiz = 1 $. Now, if the NPC condition \cref{eq:NPC} is first detected only at the very last iteration, then from \cref{lemma:T_PSD}-\labelcref{lemma:T_PSD_ii} and the proof of \cref{thm:T_indefinite}, it follows that we must have $ \psin = 0 $. Conversely, if $ \psin = 0 $, then just as in the proof of \cref{thm:T_indefinite}, we can show that $ \TT_{t} \succeq \zero$ for all $ 1 \leq t \leq g $. Now, by the Sturm Sequence Property \cite[Theorem 8.4.1]{golub2012matrix}, since the smallest eigenvalue of $ \TT_{t} $ decreases to zero monotonically, it follows that we must have $ \TT_{g} \succeq \zero $ but $ \TT_{t} \succ \zero, \; 1 \leq t \leq g-1 $. If, in addition, $ \psiz = 0 $, then necessarily we have $ \TT_{g} \succ \zero $ also.

Building upon \cref{thm:T_indefinite}, it turns out, we can also obtain a more refined certificate for the positive (semi) definiteness of $ \HH $ than \cite[Theorem 3.5]{liu2022minres}, which might be of independent interest than the rest of this paper. 
\begin{theorem}
	\label{thm:H_PSD}
	Assume $ \Psi(\HH, \gg) = \Theta(\HH) $, i.e., all eigenvalues of $ \HH $ are $\gg$-relevant. 
	\begin{enumerate}
		\item $ \HH \succ \zero $ if and only if the NPC condition \cref{eq:NPC} is never detected for all $ 1 \leq t \leq g $.
		\item $ \HH \succeq \zero $ if and only if the NPC condition \cref{eq:NPC} is first detected at the last iteration and evaluates to zero.
	\end{enumerate}
\end{theorem}
\begin{proof}
	\hfill
	\begin{enumerate}
		\item   Clearly if $ \HH \succ \zero $, then by \cref{eq:NPC_curve}, the NPC condition \cref{eq:NPC} is never detected. Conversely, suppose $ \HH \not \succ \zero $. If $ t \leq g $ is the first iteration such that $ \TTt \not \succ \zero$, then by \cref{lemma:MINRES_NPC_detector}, the NPC condition \cref{eq:NPC} must hold. So, suppose $ \TTt \succ \zero$ for all $ 1 \leq t \leq g $. Since $ \Psi(\HH, \gg) = \Theta(\HH) $, the decomposition \cref{eq:H_lambda_U} must involve an eigenvector, $\uu$, corresponding to a non-positive eigenvalue of $ \HH $. From \cref{eq:vandermonde}, it follows that $ \uu \in \Klv_{g}(\HH,\gg) $. Let $ \VVg \in \real^{d \times g} $ denote the basis obtained from the Lanczos process for $ \Klv_{g}(\HH,\gg) $. Hence, we can write $ \uu = \VVg \ww $ for some non-zero $ \ww \in \real^{g} $.
		It follows that 
		\begin{align*}
			\dotprod{\ww, \TTg \ww} = \dotprod{\VVg \ww, \HH \VVg \ww} = \dotprod{\uu, \HH \uu} \leq 0,
		\end{align*}
		which implies that $ \TT_{g} \not\succ \zero $, and we arrive at a contradiction. 
		
		\item If $ \HH \succeq \zero $, then we must have $ \TTt = \VVtT \HH \VVt \succeq \zero $. Now, by the Sturm Sequence Property, it follows that we must have $ \TT_{g} \succeq \zero $ but $ \TT_{t} \succ \zero, \; 1 \leq t \leq g-1 $. The results follows from applying \cref{lemma:MINRES_NPC_detector}. 
		Now, suppose the NPC condition \cref{eq:NPC} is first detected at the last iteration, and $ \dotprod{\rr_{g-1}, \HH \rr_{g-1}}  = 0$ but $ \HH \not \succeq \zero $. By \cref{lemma:MINRES_NPC_detector,lemma:T_PSD}, we must have $ \TTg \succeq \zero $. Similar to the proof of the first part, consider $ \uu \in \Klv_{g}(\HH,\gg) $ to be an eigenvector in the decomposition \cref{eq:H_lambda_U}, which corresponds to a negative eigenvalue of $ \HH $. Letting $ \uu = \VVg \ww $ for some non-zero $ \ww \in \real^{g} $, it follows that 
		\begin{align*}
			\dotprod{\ww, \TTg \ww} = \dotprod{\VVg \ww, \HH \VVg \ww} = \dotprod{\uu, \HH \uu} < 0,
		\end{align*}
		which implies that $ \TT_{g} \not\succeq \zero $, and we arrive at a contradiction.
	\end{enumerate}
\end{proof}

We end this section with an important implication of \cref{eq:H_decomp_PSD}, which is used heavily in our analysis later in this paper.
\begin{lemma}
	\label{lemma:HUUS}
	Let $ 1 \leq t \leq g $. For any $ \zz \in \Klv_{t}(\HH, \gg) $, we have $ \HH \zz = \HH\UU\UUT \zz $ where $ \UU $ is as in \cref{eq:H_decomp_PSD}.
\end{lemma}
\begin{proof}
	We first note that that since $ \Lambda_{\perp} \UUpT \gg = \zero $, it follows that $ \HH \UUp \UUpT \gg = \UUp \UUpT \HH \gg = \zero $. Hence, for any $ \zz \in \kry(\HH, \gg) $, we have
	\begin{align*}
		\HH \zz &= \HH \left( \sum_{i = 0}^{t-1} c_{i} \HH^{i} \gg \right) = \sum_{i = 0}^{t-1} c_{i} \HH^{i+1} \gg \\ 
		&= \sum_{i = 0}^{t-1} c_{i} \HH^{i} \left(\HH \UU \UU^{\T} + \HH \UUp \UUp^{\T}\right) \gg \\
		&= \sum_{i = 0}^{t-1} c_{i} \HH^{i} \HH \UU \UU^{\T} \gg = \HH \left( \sum_{i = 0}^{t-1} c_{i} \HH^{i} \UU \UU^{\T}  \gg \right) \\
		&= \HH \UU \UU^{\T} \left( \sum_{i = 0}^{t-1} c_{i} \HH^{i} \gg \right) = \HH \UU \UU^{\T} \zz.
	\end{align*}
\end{proof}

\subsubsection{Complexity of Finding Approximate Solution}
\label{sec:MINRES_complexity_solution}
So long as the NPC condition \cref{eq:NPC} has not been detected, the iterates of \cref{alg:MINRES} are guaranteed to yield decent in both $ f(\xx) $ and $ \vnorm{\gg}^{2} $. We now investigate the complexity of \cref{alg:MINRES} for obtaining an approximate solution to \cref{eq:lieast-squares}. 
Clearly, if $ \ggk \in \Null(\HHk) $, then $ \ggk $ is declared an NPC direction at the first iteration of \cref{alg:MINRES}. So, in obtaining the convergence rate, we can safely assume that $ \ggk \not \in \Null(\HHk) $. \cref{lemma:MINRES_complexity_Sol} will give the complexity of MINRES for obtaining a solution, which satisfies the inexactness condition \cref{eq:inexactness}.

\begin{lemma}
	\label{lemma:MINRES_complexity_Sol}
	Suppose $ \gg \notin \Null(\HH) $, i.e., $ \psip+\psin \geq 1 $. Consider any $ 1 \leq i \leq \psip $ and $\psip+\psiz+1 \leq j \leq \psi$. For any $ 0 < \theta < 1 $, after at most 
	\begin{align*}
		t \geq \frac{\sqrt{\max\{\kappap_{i},\kappan_{j}\}}}{4} \log \left(\frac{4}{ \theta}\right),
	\end{align*}
	iterations of \cref{alg:MINRES}  (without the NPC detection mechanism of \hyperref[alg:MINRES:NPC]{Step 11}), we have 
	\begin{align*}
		\vnorm{\rrt}^{2} &\leq \vnorm{\left( \eye - \begin{bmatrix}\UU_{i+} & \UU_{j-} \end{bmatrix} \begin{bmatrix}\UU_{i+} & \UU_{j-} \end{bmatrix}^{\intercal} \right) \gg}^2 \\
		& + \theta \vnorm{\left(\begin{bmatrix}\UU_{i+} & \UU_{j-} \end{bmatrix} \begin{bmatrix}\UU_{i+} & \UU_{j-} \end{bmatrix}^{\intercal}\right)\gg}^{2},
	\end{align*}
	and in particular,
	\begin{align*}
		\vnorm{\UU \UUT \rrt}^{2} &\leq \vnorm{\left( \UU\UUT - \begin{bmatrix}\UU_{i+} & \UU_{j-} \end{bmatrix} \begin{bmatrix}\UU_{i+} & \UU_{j-} \end{bmatrix}^{\intercal} \right) \gg}^2 \\
		&+ \theta \vnorm{\left(\begin{bmatrix}\UU_{i+} & \UU_{j-} \end{bmatrix} \begin{bmatrix}\UU_{i+} & \UU_{j-} \end{bmatrix}^{\intercal}\right)\gg}^{2},
	\end{align*}
	where 
	\begin{align*}
		\kappap_{i} \defeq \frac{\lambda_{1}}{\lambda_{i}}, \quad \text{and}, \quad \kappan_{j} \defeq \frac{\lambda_{\psi}}{\lambda_{j}},
	\end{align*}
	and $ \UU$, $\UU_{i+} $, and $\UU_{j-}$ are as in \cref{eq:H_decomp_PSD,eq:H_decomp_PSD_ij}.
	If $ \psin = 0 $, then the statement holds with $ \kappap_{i} $ instead of $\max\{\kappap_{i},\kappan_{j}\}$.  The statement is similarly modified if $ \psip = 0 $.
\end{lemma}
\begin{proof}  
	Recalling the $ t\th $ iteration of MINRES as in \cref{eq:MINRES}, it follows that
	\begin{align*}
		&\min_{\ss \in \kry(\HH, \gg)} \vnorm{\HH \ss + \gg}^2 = \min_{\ss \in \kry(\HH, \gg)} \vnorm{\HH \ss + \left(\UU \UU^{\T} + \UUp \UUp^{\T}\right) \gg}^2 \\
		&= \min_{\ss \in \kry(\HH, \gg)} \vnorm{\HH \ss + \UU \UU^{\T} \gg}^2 + \vnorm{\UUp \UUp^{\T}\gg}^2 \\
		&= \min_{\ss \in \kry(\HH, \gg)} \vnorm{\HH \UU \UU^{\T} \ss + \UU \UU^{\T} \gg}^2 + \vnorm{\UUp \UUp^{\T}\gg}^2 \\
		&= \min_{\ss \in \UU \UU^{\T} \kry(\HH, \gg)} \vnorm{\HH \UU \UU^{\T} \ss + \UU \UU^{\T} \gg}^2 + \vnorm{\UUp \UUp^{\T}\gg}^2 \\
		&= \min_{\ss \in \UU \UU^{\T} \kry(\HH, \gg)} \vnorm{\UU \UU^{\T} \HH  \ss + \UU \UU^{\T} \gg}^2 + \vnorm{\UUp \UUp^{\T}\gg}^2 \\
		&= \min_{\ss \in \UU \UU^{\T} \kry(\HH, \gg)} \left(\vnorm{\UU_{+} \UU_{+}^{\T} \HH  \ss + \UU_{+} \UU_{+}^{\T} \gg}^2 + \vnorm{\UU_{-} \UU_{-}^{\T} \HH  \ss + \UU_{-} \UU_{-}^{\T} \gg}^2 \right) + \vnorm{\UUp \UUp^{\T}\gg}^2 \\
		&= \min_{\ss \in \UU \UU^{\T} \kry(\HH, \gg)} \left(\vnorm{ \HH \UU_{+} \UU_{+}^{\T} \ss + \UU_{+} \UU_{+}^{\T} \gg}^2 + \vnorm{ \HH \UU_{-} \UU_{-}^{\T} \ss + \UU_{-} \UU_{-}^{\T} \gg}^2 \right) + \vnorm{\UUp \UUp^{\T}\gg}^2 \\
		&= \min_{\pp \in \UU_{+} \UU_{+}^{\T} \kry(\HH, \gg)} \vnorm{ \HH \UU_{+} \UU_{+}^{\T} \pp + \UU_{+} \UU_{+}^{\T} \gg}^2 + \min_{\qq \in \UU_{-} \UU_{-}^{\T} \kry(\HH, \gg)} \vnorm{ \HH \UU_{-} \UU_{-}^{\T} \qq + \UU_{-} \UU_{-}^{\T} \gg}^2  \\
		&\quad \quad \quad \quad + \vnorm{\UUp \UUp^{\T}\gg}^2 \\
		&= \min_{\pp \in \UU_{+} \UU_{+}^{\T} \kry(\HH, \gg)} \vnorm{ \Lambda_{+} \UU_{+}^{\T} \pp + \UU_{+}^{\T} \gg}^2 + \min_{\qq \in \UU_{-} \UU_{-}^{\T} \kry(\HH, \gg)} \vnorm{ \Lambda_{-}  \UU_{-}^{\T}  \qq + \UU_{-}^{\T} \gg}^2  + \vnorm{\UUp \UUp^{\T}\gg}^2 \\
		&= \min_{\pp \in \UU_{+} \kry(\Lambda_{+}, \UU_{+}^{\T}\gg)} \vnorm{ \Lambda_{+} \UU_{+}^{\T} \pp + \UU_{+}^{\T} \gg}^2 + \min_{\qq \in \UU_{-} \kry(\Lambda_{-}, \UU_{-}^{\T} \gg)} \vnorm{ \Lambda_{-}  \UU_{-}^{\T}  \qq + \UU_{-}^{\T} \gg}^2  + \vnorm{\UUp \UUp^{\T}\gg}^2 \\
		&= \min_{\hat{\pp} \in \kry(\Lambda_{+}, \UU_{+}^{\T}\gg)} \vnorm{ \Lambda_{+} \hat{\pp} + \UU_{+}^{\T} \gg}^2 + \min_{\hat{\qq} \in \kry(\Lambda_{-}, \UU_{-}^{\T} \gg)} \vnorm{ \Lambda_{-} \hat{\qq} + \UU_{-}^{\T} \gg}^2  + \vnorm{\UUp \UUp^{\T}\gg}^2.
	\end{align*}
	Now, for $ 1 \leq i \leq \psip $ and $\psip+\psiz+1 \leq j \leq \psi$, we have 
	\begin{align*}
		\min_{\ss \in \kry(\HH, \gg)} \vnorm{\HH \ss + \gg}^2 &= \min_{\tilde{\pp} \in \kry\left(\Lambda_{i+}, \UU_{i+}^{\T}\gg\right)}  \vnorm{ \Lambda_{i+} \tilde{\pp}  + \UU_{i+}^{\T} \gg}^2 \\
		&+ \min_{\bar{\pp} \in \kry\left(\Lambda^{\complement}_{i+}, \left[\UU^{\complement}_{i+}\right]^{\T}\gg\right)}  \vnorm{ \Lambda^{\complement}_{i+} \bar{\pp} + \left[\UU_{i+}^{\complement}\right]^{\T} \gg}^2 \\
		&+ \min_{\tilde{\qq} \in \kry\left(\Lambda_{j-}, \UU_{j-}^{\T}\gg\right)}  \vnorm{ \Lambda_{j-} \tilde{\qq}  + \UU_{j-}^{\T} \gg}^2 \\
		&+ \min_{\bar{\qq} \in \kry\left(\Lambda^{\complement}_{j-}, \left[\UU^{\complement}_{j-}\right]^{\T}\gg\right)}  \vnorm{ \Lambda^{\complement}_{j-} \bar{\qq} + \left[\UU_{j-}^{\complement}\right]^{\T} \gg}^2  + \vnorm{\UUp \UUp^{\T}\gg}^2 \\
		&\leq \min_{\tilde{\pp} \in \kry\left(\Lambda_{i+}, \UU_{i+}^{\T}\gg\right)}  \vnorm{ \Lambda_{i+} \tilde{\pp}  + \UU_{i+}^{\T} \gg}^2 + \min_{\tilde{\qq} \in \kry\left(\Lambda_{j-}, \UU_{j-}^{\T}\gg\right)}  \vnorm{ \Lambda_{j-} \tilde{\qq}  + \UU_{j-}^{\T} \gg}^2 \\
		& \quad + \vnorm{\left[\UU_{i+}^{\complement}\right]^{\T} \gg}^2 + \vnorm{\left[\UU_{j-}^{\complement}\right]^{\T} \gg}^2  + \vnorm{\UUp^{\T}\gg}^2.
	\end{align*}
	Since $ \Lambda_{i+} \succ \zero $ and $ \Lambda_{j-} \prec \zero $, from the standard convergence rate of MINRES for definite matrices \cite[(3.12)]{greenbaum1997iterative}, we get 
	\begin{align*}
		\vnorm{ \Lambda_{i+} \tilde{\pp} + \UU_{i+}^{\T} \gg} &\leq 2 \vnorm{\UU_{i+} \UU_{i+}^{\T} \gg} \left(\frac{\sqrt{\kappap_{i}} - 1}{\sqrt{\kappap_{i}} + 1}\right)^{t}, \\
		\vnorm{ \Lambda_{j-} \tilde{\qq} + \UU_{j-}^{\T} \gg}^2  &\leq 2 \vnorm{\UU_{j-} \UU_{j-}^{\T} \gg} \left(\frac{\sqrt{\kappan_{j}} - 1}{\sqrt{\kappan_{j}} + 1}\right)^{t},
	\end{align*}
	which implies that
	\begin{align*}
		\vnorm{\rrt}^{2} &\leq \vnorm{\left( \UU_{i+}^{\complement} \left[\UU_{i+}^{\complement}\right]^{\T} + \UU_{j-}^{\complement} \left[\UU_{j-}^{\complement}\right]^{\T} + \UUp \UUp^{\T} \right) \gg}^2 \\
		& \quad + 4 \vnorm{\UU_{i+} \UU_{+}^{\T} \gg}^{2} \left(\frac{\sqrt{\kappap_{i}} - 1}{\sqrt{\kappap_{i}} + 1}\right)^{2t}  + 4 \vnorm{\UU_{j-} \UU_{j-}^{\T} \gg}^{2} \left(\frac{\sqrt{\kappan_{j}} - 1}{\sqrt{\kappan_{j}} + 1}\right)^{2t} \\
		&\leq \vnorm{\left( \eye - \begin{bmatrix}\UU_{i+} & \UU_{j-} \end{bmatrix} \begin{bmatrix}\UU_{i+} & \UU_{j-} \end{bmatrix}^{\intercal} \right) \gg}^2 \\
		& + 4 \vnorm{\left(\UU_{i+} \UU_{i+}^{\T} + \UU_{j-} \UU_{j-}^{\T}\right) \gg}^{2} \left( \max \left\{\frac{\sqrt{\kappap_{i}} - 1}{\sqrt{\kappap_{i}} + 1}, \frac{\sqrt{\kappan_{j}} - 1}{\sqrt{\kappan_{j}} + 1}\right\} \right)^{2t},
	\end{align*}
	where we have used the fact that $ \begin{bmatrix}\UU_{i+} & \UU_{j-} \end{bmatrix} \begin{bmatrix}\UU_{i+} & \UU_{j-} \end{bmatrix}^{\intercal} = \eye -\UU_{i+}^{\complement} \left[\UU_{i+}^{\complement}\right]^{\T} - \UU_{j-}^{\complement} \left[\UU_{j-}^{\complement}\right]^{\T} - \UUp \UUp^{\T} $.
	Hence, using the fact that $ \log \max \{a,b\} = \max \log\{a,b\}, \; a > 0, b > 0$, if 
	\begin{align*}
		&t \geq \frac{\sqrt{\max\{\kappap_{i},\kappan_{j}\}}}{4} \log \left(4/\theta\right),
	\end{align*}
	then we have the desired result. 
\end{proof}

\begin{remark}
	\cref{lemma:MINRES_complexity_Sol} gives the convergence of MINRES involving arbitrary symmetric matrices. To our knowledge, it provides an improvement over the existing results on the convergence of MINRES applied to indefinite symmetric matrices. Indeed, available general convergence results for indefinite problems imply rates depending on $ \kappap $ and $ \kappan $ as opposed to $ \sqrt{\kappap} $ and $ \sqrt{\kappan} $, e.g., \cite{greenbaum1997iterative,xie2017convergence,simoncini2013superlinear,fischer2011polynomial}. While the existing convergence results in the literature are sub-optimal, \cref{lemma:MINRES_complexity_Sol} provides a convergence rate, which similar to positive-definite system, depends optimally on $ \sqrt{\kappap} $ and $ \sqrt{\kappan} $. The key to obtaining the improvements in \cref{lemma:MINRES_complexity_Sol} is by considering the interplay between $ \gg $ and $ \HH $ by separating positive and negative $\gg$-relevant eigenvalues, and the projection of the iterates on the corresponding eigenspaces.  
\end{remark}

From the proof of \cref{lemma:MINRES_complexity_Sol}, it is easy to see that, for any $ 1 \leq i \leq \psip $ and $\psip+\psiz+1 \leq j \leq \psi$, we have
\begin{align*}
	\min_{\ss \in \kry(\HH, \gg)} \vnorm{\HH \ss + \gg}^2 &\leq \min \left\{\min_{\hat{\pp} \in \kry(\Lambda_{i+}, \UU_{j+}^{\T}\gg)} \vnorm{ \Lambda_{i+} \hat{\pp} + \UU_{i+}^{\T} \gg}^2 + \vnorm{\eye - \UU_{i+} \UU_{i+}^{\T}\gg}, \right. \\
	& \quad \quad \quad \quad \left. \min_{\hat{\qq} \in \kry(\Lambda_{j-}, \UU_{j-}^{\T} \gg)} \vnorm{ \Lambda_{j-} \hat{\qq} + \UU_{j-}^{\T} \gg}^2 + \vnorm{\eye -  \UU_{j-} \UU_{j-}^{\T} \gg}^2 \right\}, 
\end{align*}
from which, following similar steps as in the proof of \cref{lemma:MINRES_complexity_Sol}, we get the following corollary.
\begin{corollary}
	\label{cor:MINRES_complexity_Sol}
	Suppose $ \psip \geq 1 $. Consider any $ 1 \leq i \leq \psip $. For any $ 0 < \theta < 1 $, after at most 
	\begin{align*}
		t \geq \frac{\sqrt{\kappap_{i}}}{4} \log \left(\frac{4}{ \theta}\right),
	\end{align*}
	iterations of \cref{alg:MINRES}  (without the NPC detection mechanism of \hyperref[alg:MINRES:NPC]{Step 11}), we have 
	\begin{align*}
		\vnorm{\rrt}^{2} &\leq \vnorm{\left( \eye - \UU_{i+} \UU_{i+}^{\intercal} \right) \gg}^2 + \theta \vnorm{\UU_{i+} \UU_{i+}^{\intercal} \gg}^{2},
	\end{align*}
	and in particular,
	\begin{align*}
		\vnorm{\UU \UU^{\T} \rrt}^{2} &\leq \vnorm{\left( \UU\UUT - \UU_{i+}\UU_{i+}^{\T} \right) \gg}^2 + \theta \vnorm{\UU_{i+} \UU_{i+}^{\intercal} \gg}^{2}.
	\end{align*}
	If $ \psin \geq 1 $, then letting $\psip+\psiz+1 \leq j \leq \psi$, the statement also holds with $ \UU_{j-} $ and $ \kappan_{j}$, instead of $ \UU_{i+} $ and $\kappap_{i}$, respectively.
\end{corollary}
\cref{cor:MINRES_complexity_Sol} only involves the positive (or negative) $\gg$-relevant spectrum of $ \HH $ and hence is more appealing than \cref{lemma:MINRES_complexity_Sol} when either of $ \kappan $ (or $ \kappap $) is very large. Of course, this comes at the cost of a larger residual error on the right-hand.

The following Lemma will be used in conjunction with \cref{lemma:MINRES_complexity_Sol,cor:MINRES_complexity_Sol} to establish the iteration complexity of \cref{alg:MINRES} to obtain a solution which satisfies the inexactness condition \cref{eq:inexactness}.
\begin{lemma}
	\label{lemma:inexactness}
	Let $ \UU $ be as in \cref{eq:H_decomp_PSD}. If for any $ \eta > 0 $, we have
	\begin{align*}
		\vnorm{\UU\UUT\rrtp}^{2} \leq \frac{\eta^{2}}{\displaystyle \left( \vnorm{\UUT\HH\UU}^{2} + \eta^{2} \right)} \vnorm{\UU\UUT\gg}^{2},
	\end{align*}
	then the inexactness condition \cref{eq:inexactness} holds.
\end{lemma}
\begin{proof}
	First recall that $ \rrtp \in \Klv_{t}(\HH,\gg) $. We have 
	\begin{align*}
		&\vnorm{\UUT\HH\UU}^{2} \vnorm{\UU\UUT\rrtp}^{2} \leq \eta^{2} \left(\vnorm{\UU\UUT\gg}^{2} - \vnorm{\UU\UUT\rrtp}^{2}\right).
	\end{align*}
	We also get $ \vnorm{\HH\sstp}^{2} = \vnorm{\HH\UU\UUT\sstp}^{2} = \vnorm{\UU\UUT\gg}^{2} - \vnorm{\UU\UUT\rrtp}^{2} $. This, coupled with the fact that $\vnorm{\HH\UU\UUT\rrtp} \leq  \vnorm{\UUT\HH\UU} \vnorm{\UU\UUT \rrtp} $ and noting that by \cref{lemma:HUUS} we have $ \vnorm{\HH\rrtp} = \vnorm{\HH \UU\UUT\rrtp} $, implies the desired result.
\end{proof}
For example, suppose $ \eta $ is large enough such that 
\begin{align*}
	\frac{\eta^{2}}{\displaystyle  \left( \vnorm{\UUT\HH\UU}^{2} + \eta^{2} \right)} > \frac{\vnorm{\left( \UU\UUT - \UU_{i+}\UU_{i+}^{\T} \right) \gg}^2}{\vnorm{\UU\UUT\gg}^2}.
\end{align*}
From \cref{lemma:MINRES_complexity_Sol}, after at most 
\begin{align*}
	t \geq \frac{\sqrt{\max\{\kappap_{i},\kappan_{j}\}}}{4} \log \left(\frac{4}{{\eta^{2}}/{\displaystyle  \left( \vnorm{\UUT\HH\UU}^{2} + \eta^{2} \right)} - \vnorm{\left( \UU\UUT - \UU_{i+}\UU_{i+}^{\T} \right) \gg}^2/\vnorm{\UU\UUT\gg}^2}\right) +1,
\end{align*}
iteration we have a solution $ \sstp $, which satisfies \cref{eq:inexactness}.

\subsubsection{Complexity of Finding NPC Direction}
\label{sec:MINRES_complexity_NPC}

\cref{thm:T_indefinite} shows that as long as $ \HH $ contains any non-positive $\gg$-relevant eigenvalues, MINRES is guaranteed to detect a NPC direction. We now establish a worst-case complexity of MINRES for detecting such directions. To detect an NPC direction, we do not necessarily need to estimate the smallest $\gg$-relevant eigenvalue of $ \HH $. As long as we can crudely approximate \emph{any} negative $\gg$-relevant eigenvalue, we can safely extract the corresponding NPC direction. This observation is basis in the derivation of \cref{lemma:Lanczos_convergence_T_indefinite}. The proof idea below draws upon that in \cite[Sections 4.4 and 6.6]{saad2011numerical}. 

\begin{lemma}
	\label{lemma:Lanczos_convergence_T_indefinite}
	Suppose $ \psip \geq 1 $, $ \psin \geq 1 $, and let $ \zeta_{t} $ be the smallest eigenvalue of $ \TT_{t}$ for $ t \leq g $. For any $ \psip+\psiz+1 \leq j \leq \psi $, we have
	\begin{align*}
		\zeta_{t} - \lambda_{j} \leq  4 \left(\frac{1 - \nuj}{\nuj} \right) \left(\lambda_{1} - \lambda_{j}\right) \left(\frac{\sqrt{\kappa_{j}+1} - 1}{\sqrt{\kappa_{j}+1} + 1}\right)^{2(t-1)},
	\end{align*}
	where
	\begin{align}
		\label{eq:tkappa}
		\kappa_{j} \triangleq \frac{\lambda_1}{-\lambda_{j}},
	\end{align}
	and
	\begin{align}
		\label{eq:nuj}
		\nuj \defeq \frac{\vnorm{\UU_{j-}\UU_{j-}^{\T} \gg}^2}{\vnorm{ \displaystyle \UU\UUT \gg}^2},
	\end{align}
	and $ \lambda_{i} $, $ \UU $ and $\UU_{j-}$ are as in \cref{eq:H_decomp_PSD,eq:H_decomp_PSD_ij}.
\end{lemma}
\begin{proof}   
	We first note that
	\begin{align*}
		\zeta_{t} = \min_{p \in \mathcal{P}_{t-1}} \frac{\dotprod{p(\HH)\gg,\HH p(\HH)\gg}}{\dotprod{p(\HH)\gg,p(\HH)\gg}} \leq \min_{p \in \mathcal{P}_{t-1}} \frac{\dotprod{p(\HH) \UU\UUT\gg,\HH p(\HH)\UU\UUT\gg}}{\dotprod{p(\HH)\UU\UUT\gg,p(\HH)\UU\UUT\gg}},
	\end{align*}
	where $ \mathcal{P}_{t-1} $ is the space of all polynomials of degree not exceeding $ t-1 $ and the inequality follows by \cref{lemma:HUUS} and the fact that $ \vnorm{\UU\UUT p(\HH)\gg} \leq \vnorm{p(\HH)\gg} $. 
	Denote
	\begin{align*}
		\mathcal{I} \defeq \left\{1,\ldots,\psip,\psip+\psiz+1,\ldots,\psi\right\}.
	\end{align*}
	Note that $ \abs{\mathcal{I}} = \psi $ if $ \psiz = 0 $, otherwise $ \abs{\mathcal{I}} = \psi-1 $.
	With $ \UU $ as in \cref{eq:H_decomp_PSD} and defining
	\begin{align*}
		c_i \defeq \frac{\dotprod{\uu_{i},\gg}}{\vnorm{\displaystyle \UU\UUT\gg}},
	\end{align*}
	we have 
	\begin{align*}
		\UU\UUT\gg = \sum_{i \in \mathcal{I}} \dotprod{\uu_{i},\gg} \uu_{i} = \vnorm{\UU\UUT\gg} \sum_{i \in \mathcal{I}} c_i \uu_{i}.
	\end{align*}
	Noting that $ \lambda_{i} < \lambda_{j} \leq 0 $ for any $ \psip+\psiz+1 \leq j < i < \psi$, it follows that for any $ \psip+\psiz+1 \leq j \leq \psi $, we have
	\begin{align*}
		\zeta_{t} - \lambda_{j} &\leq \min_{p \in \mathcal{P}_{t-1}}  \left( \frac{\displaystyle \dotprod{p(\HH) \UU\UUT \gg, \HH p(\HH) \UU\UUT \gg}}{\displaystyle \dotprod{p(\HH) \UU\UUT \gg, p(\HH) \UU\UUT \gg}} - \lambda_{j} \right)  = \min_{p \in \mathcal{P}_{t-1}} \left(\frac{\displaystyle \sum_{i \in \mathcal{I}} \lambda_i c_{i}^{2} p^{2}(\lambda_i)}{\displaystyle \sum_{i \in \mathcal{I}} c_{i}^{2} p^{2}(\lambda_i)} - \lambda_{j}\right) \\
		&= \min_{p \in \mathcal{P}_{t-1}} \left(\frac{\displaystyle \sum_{i \leq \psip} \lambda_i c_{i}^{2} p^{2}(\lambda_i) + \sum_{i \geq \psip+\psiz+1} \lambda_i c_{i}^{2} p^{2}(\lambda_i)}{\displaystyle \sum_{i \in \mathcal{I}} c_{i}^{2} p^{2}(\lambda_i)} - \lambda_{j}\right) \\
		&\leq \min_{p \in \mathcal{P}_{t-1}} \left(\frac{\displaystyle \sum_{i \leq \psip} \lambda_i c_{i}^{2} p^{2}(\lambda_i) + \sum_{j \leq i \leq \psi} \lambda_i c_{i}^{2} p^{2}(\lambda_i)}{\displaystyle \sum_{i \in \mathcal{I}} c_{i}^{2} p^{2}(\lambda_i)} - \lambda_{j}\right)\\
		&\leq \min_{p \in \mathcal{P}_{t-1}} \left(\frac{\displaystyle \sum_{i \leq \psip} \lambda_i c_{i}^{2} p^{2}(\lambda_i) + \lambda_{j} \sum_{j \leq i \leq \psi} c_{i}^{2} p^{2}(\lambda_i)}{\displaystyle \sum_{i \leq \psip} c_{i}^{2} p^{2}(\lambda_i) + \sum_{j \leq i \leq \psi} c_{i}^{2} p^{2}(\lambda_i)} - \lambda_{j}\right)\\
		&= \min_{p \in \mathcal{P}_{t-1}} \left(\frac{\displaystyle \sum_{i \leq \psip} \left(\lambda_{i} -\lambda_{j}\right) c_{i}^{2} p^{2}(\lambda_i)}{\displaystyle \sum_{i \leq \psip} c_{i}^{2} p^{2}(\lambda_i) + \sum_{j \leq i \leq \psi} c_{i}^{2} p^{2}(\lambda_i)}\right)\\
		&\leq \left(\lambda_{1} -\lambda_{j}\right)  \left(\frac{\displaystyle \sum_{i \leq \psip} c_{i}^{2} p^{2}(\lambda_i)}{\displaystyle \sum_{i \leq \psip} c_{i}^{2} p^{2}(\lambda_i) + \sum_{j \leq i \leq \psi} c_{i}^{2} p^{2}(\lambda_i)}\right), \quad \forall p \in \mathcal{P}_{t-1}.
	\end{align*}
	Now, take 
	\begin{align*}
		p(\lambda) = C_{t-1}\left(1 + 2 \left( \frac{\lambda_{\psip} - 2\lambda}{2\lambda_1 - \lambda_{\psip}} \right)\right),
	\end{align*}
	where $ C_{t} $ is the Chebyshev polynomial of the first kind of degree $ t $ defined as 
	\begin{align*}
		C_{t}(x) \triangleq \begin{cases}
			\cos(t \cos^{-1}(x)), \quad \abs{x} \leq 1, \\ \\
			\hf \left(\left(x + \sqrt{x^2 - 1}\right)^{t} + \left(x + \sqrt{x^2 - 1}\right)^{-t}\right), \quad \abs{x} > 1.
		\end{cases} 
	\end{align*}
	Since $ \abs{p(\lambda_{i})} \leq 1, \; 1 \leq i \leq \psip $, we have
	\begin{align*}
		\zeta_{t} - \lambda_{j} &\leq \left(\lambda_{1} -\lambda_{j}\right)  \left(\frac{\displaystyle \sum_{i \leq \psip} c_{i}^{2}}{\displaystyle \sum_{i \leq \psip} c_{i}^{2} p^{2}(\lambda_i) + \sum_{j \leq i \leq \psi} c_{i}^{2} p^{2}(\lambda_i)}\right) \leq \left(\lambda_{1} -\lambda_{j}\right)  \left(\frac{\displaystyle \sum_{\substack{{i \in \mathcal{I}} \\ {i \leq j-1}}} c_{i}^{2}}{\displaystyle \sum_{j \leq i \leq \psi} c_{i}^{2} p^{2}(\lambda_i)}\right).
	\end{align*}
	It is easy to show that $ C_{t}(x) $ is increasing on $ x > 1 $. Indeed, we have 
	\begin{align*}
		\frac{\df C_{t}(x)}{\df x} &= \hf \left(1 + \frac{2 x}{\sqrt{x^2 - 1}}\right)\left(t \left(x + \sqrt{x^2 - 1}\right)^{t-1} - t \left(x + \sqrt{x^2 - 1}\right)^{-t-1}\right) \\
		&= \frac{t}{2} \left(1 + \frac{2 x}{\sqrt{x^2 - 1}}\right) \frac{\left(x + \sqrt{x^2 - 1}\right)^{2t} - 1}{\left(x + \sqrt{x^2 - 1}\right)^{t+1}} > 0.
	\end{align*}
	Hence, $ p(\lambda_{j}) < p(\lambda_{j+1}) < \ldots < p(\lambda_{\psi}) $. It follows that 
	\begin{align*}
		\zeta_{t} - \lambda_{j} \leq \frac{(\lambda_1 - \lambda_{j})}{p^{2}(\lambda_{j})} \left(\frac{\displaystyle \sum_{\substack{{i \in \mathcal{I}} \\ {i \leq j-1}}} c_{i}^{2}}{\displaystyle \sum_{j \leq i \leq \psi} c_{i}^{2} }\right) = \frac{(\lambda_1 - \lambda_{j})}{p^{2}(\lambda_{j})} \left(\frac{\displaystyle \sum_{i \in \mathcal{I}} c_i^2 - \sum_{j \leq i \leq \psi} c_{i}^{2}}{\displaystyle \sum_{j \leq i \leq \psi} c_{i}^{2} }\right).
	\end{align*}
	From
	\begin{align*}
		\frac{\lambda_{\psip} - 2\lambda_{j}}{2\lambda_1 - \lambda_{\psip}} \geq \frac{-\lambda_{j}}{\lambda_1},
	\end{align*}
	we also have 
	\begin{align*}
		p(\lambda_{j}) &= C_{t-1}\left( 1 + 2 \left( \frac{\lambda_{\psip} - 2\lambda_{j}}{2\lambda_1 - \lambda_{\psip}} \right) \right) \geq C_{t-1}\left(1 - \frac{2\lambda_{j}}{\lambda_{1}}\right) = C_{t-1}\left(1 + \frac{2}{\kappa_{j}}\right) \\
		&= \hf \left(\left(1 + \frac{2}{\kappa_{j}} + \sqrt{\left(1 + \frac{2}{\kappa_{j}}\right)^2 - 1}\right)^{t-1} + \left(1 + \frac{2}{\kappa_{j}} + \sqrt{\left(1 + \frac{2}{\kappa_{j}}\right)^2 - 1}\right)^{1-t}\right) \\
		&= \hf \left(\left(\frac{\kappa_{j}+2 + 2\sqrt{\kappa_{j}+1}}{\kappa_{j}}\right)^{t-1} + \left(\frac{\kappa_{j}+2 + 2\sqrt{\kappa_{j}+1}}{\kappa_{j}}\right)^{1-t}\right) \\
		&= \hf \left(\left(\frac{\sqrt{\kappa_{j}+1} + 1}{\sqrt{\kappa_{j}+1} - 1}\right)^{t-1} + \left(\frac{\sqrt{\kappa_{j}+1} - 1}{\sqrt{\kappa_{j}+1} + 1}\right)^{1-t}\right) \geq \hf \left(\frac{\sqrt{\kappa_{j}+1} + 1}{\sqrt{\kappa_{j}+1} - 1}\right)^{t-1},
	\end{align*}
	which gives our desired result.
\end{proof}

As it can be seen, the bound in \cref{lemma:Lanczos_convergence_T_indefinite} is controlled by the proportion \cref{eq:nuj} and also the ``condition number'' $ \kappa_{j} $, which shed light on the factors that affect the performance of MINRES in detecting a negative curvature direction. Indeed, suppose $\psip \geq 1$, $ \psin \geq 1 $ and $ j \geq \psip+\psiz+1 $. The ratio $ \nuj $ measures the portion of $ \gg $ that lies in the relevant sub-space $ \range(\UU_{j-}) $ compared with the ``left-over'' portion that lies in the orthogonal complement in $ \range(\UU) $. It is natural to expect that $ \lambda_{j} $ can be estimated faster when $ \UU_{j-}\UU_{j-}^{\T} \gg $ constitutes a relatively large portion of $ \UU\UUT \gg $. Also, the ``condition number'' $ \kappa_j $ indicates that MINRES can detect a NPC direction faster in cases where $ \lambda_{j} $ is sufficiently larger, in magnitude, compared to $ \lambda_{1} $. In fact, the term ``condition number'' is somewhat a misnomer since $ \kappa_{j} $ can indeed be smaller than one if $ \abs{\lambda_{j}} \geq \lambda_{1}$.

\cref{lemma:Lanczos_convergence_T_indefinite} implies that, for any $ \epsilon > 0 $, after at most  
\begin{align*}
	t = \min \left\{\left \lceil \left(\frac{\sqrt{\kappa_{j}+1}}{4} \right)\log\left({4 \left(1 - \sum_{i=j}^{\psi} c_{i}^{2}\right)}/{\left(\epsilon \sum_{i=j}^{\psi} c_{i}^{2}\right)}\right)+1 \right \rceil, g \right\}
\end{align*}
iterations, for any $ 2 \leq j \leq \psi$, we have 
\begin{align}
	\label{eq:teps_min_np} 
	\zeta_t - \lambda_{j} \leq \epsilon \left(\lambda_{1} - \lambda_{j}\right).
\end{align}
Indeed, it suffices to find $ t $ large enough such that
\begin{align*}
	4 \left( \frac{1 - \sum_{i=j}^{\psi} c_{i}^{2}}{\sum_{i=j}^{\psi} c_{i}^{2}}\right) \left(\frac{\sqrt{\kappa_{j}+1} - 1}{\sqrt{\kappa_{j}+1} + 1}\right)^{2(t-1)} \leq \epsilon.
\end{align*}
If $ 4 \left(1 - \sum_{i=j}^{\psi} c_{i}^{2} \right)/\left(\sum_{i=j}^{\psi} c_{i}^{2}\right) \leq \epsilon$, this holds trivially for any $ t $. Otherwise, this is satisfied if  we have 
\begin{align*}
	t &\geq \hf \log\left(\frac{\epsilon \sum_{i=j}^{\psi} c_{i}^{2}}{4(1-\sum_{i=j}^{\psi} c_{i}^{2})}\right)/ \log \left(\frac{\sqrt{\kappa_{j}+1} - 1}{\sqrt{\kappa_{j}+1} + 1}\right) + 1  \\
	&= \hf \log\left(\frac{4(1-\sum_{i=j}^{\psi} c_{i}^{2})}{\epsilon \sum_{i=j}^{\psi} c_{i}^{2}}\right)/ \log \left(1 + \frac{2}{\sqrt{\kappa_{j}+1} - 1}\right) + 1.
\end{align*}
Now, the last bit is done by noting that for $ b > 0 $ we have
\begin{align*}
	\log\left(1 + \frac{1}{b}\right) \geq \frac{2}{2b + 1}.
\end{align*}

If $ \psin = \psiz = 0 $, then by \cref{thm:T_indefinite}, no NPC direction is detected within MINRES iterations.
Suppose, $ \psin = 0 $ but $ \psiz = 1 $. In this case, as per \cref{lemma:T_PSD}, a zero curvature direction is detected for the first time only at the very last iteration.
However, if $ \psip = 0 $, i.e., $ \HH $ has no positive $\gg$-relevant eigenvalues, then a NPC direction is detected at the very first iteration. Indeed, letting 
$ \rr_{0} = \gg = \sum_{i=1}^{\psi} \xi_{i} \uu_{i} $, we have 
\begin{align*}
	\dotprod{\rr_{0}, \HH \rr_{0}} &= \sum_{i=1}^{\psi} \xi_{i}^{2} \lambda_{i} \dotprod{\uu_{i}, \uu_{i}} \leq 0.
\end{align*}
\cref{cor:MINRES_complexity_NC} gives a complexity for detecting the NPC Condition \cref{eq:NPC} for non-trivial cases where $ \psin \geq 1 $ and $ \psip \geq 1 $.

\begin{corollary}
	\label{cor:MINRES_complexity_NC}
	Suppose $ \HH $ has at least one negative and one positive $\gg$-relevant eigenvalues, i.e., $ \psin \geq 1 $ and $ \psip \geq 1 $. 
	After at most $t = \min \left\{ T , g \right\}$ iterations of \cref{alg:MINRES} (without the inexactness mechanism of \hyperref[alg:MINRES:SOL]{Step 7}), the NPC Condition \cref{eq:NPC} is satisfied, where
	\begin{align*}
		T &\defeq \min \left\{ T_{j} \; \mid \; \psip+\psiz+1 \leq j \leq \psi  \right\}, \\
		T_{j} &\defeq \left \lceil \left(\frac{\sqrt{\kappa_{j}+1}}{4} \right)\log\left[4 \left(\kappa_{j} + 1 \right) \left(\frac{1-\nuj}{\nuj}\right)\right] + 1 \right \rceil,
	\end{align*}
	where $ \UU $, $\UU_{j-} $, $ \kappa_{j} $ and $\nu_{j} $ are as in \cref{eq:H_decomp_PSD,eq:H_decomp_PSD_ij,eq:tkappa,eq:nuj}, respectively.
\end{corollary}

\begin{proof}   
	Let $ \lambda_{j} < 0 $ be a negative $\gg$-relevant eigenvalue of $ \HH $ for some $ \psip+\psiz+1 \leq j \leq \psi $. Letting $ \epsilon \leq - \lambda_{j}/(\lambda_1 - \lambda_{j})$ in \cref{eq:teps_min_np}, we get
	\begin{align*}
		\zeta_t \leq \epsilon \left(\lambda_{1} - \lambda_{j}\right) + \lambda_{j} \leq 0.
	\end{align*}
	Now, the result follows by taking the minimum over $ \psip+\psiz+1 \leq j \leq \psi  $.
\end{proof}

\section{Newton-MR: Algorithms and Convergence Analysis}
\label{sec:NewtonMR}
Building upon the properties of \cref{alg:MINRES} in \cref{sec:MINRES_main}, we now present two variants of Newton-MR for optimization of \cref{eq:obj}. The first variant, discussed in \cref{sec:NewtonMR_first} and depicted in \cref{alg:NewtonMR_1st}, offers first-order approximate optimality guarantee as in \cref{eq:small_g}. In \cref{sec:NewtonMR_second}, we then propose a slightly more refined variant, which comes equipped with second-order approximate optimality guarantees as in \cref{eq:termination_second_order}. 

\subsection{Blanket Assumptions}
To carry out the analysis of this section, we make the following blanket assumptions regarding smoothness of the function $ f $. 

\begin{assumption}[$\gg$-relevant Hessian Boundedness]
	\label{assmpt:Lg}
	There exists a $ 0 \leq \Lg < \infty $ such that, for any $ \xx \in \real^{d} $, we have $\vnorm{\UU^{\T}\HH \UU} \leq \Lg$, where $ \UU $ is as in \cref{eq:H_decomp_PSD}
\end{assumption}
\cref{assmpt:Lg} implies that $ \max\{\lambda_{1},\abs{\lambda_{\psi}}\} \leq \Lg $, i.e., only the $\gg$-relevant subset of the spectrum of $ \HH $ is required to be bounded. In other words, it implies a directional smoothness with respect to a restricted sub-space. Clearly, \cref{assmpt:Lg} is a relaxation of the usual Lipschitz continuity assumption of $ \gg $. Indeed, the widely-used condition 
\begin{align*}
	\vnorm{\nabla f(\xx) - \nabla f(\yy)} &\leq \Lg \vnorm{\xx - \yy}, \quad \forall \xx,\yy \in \real^{d},
\end{align*}
for twice-continuously differentiable $ f $, implies $ \vnorm{\HH} \leq  \Lg$, which is stronger than \cref{assmpt:Lg}.

\begin{assumption}[Smoothness]
	\label{assmpt:lipschitz}
	The function $ f $ is twice continuously differentiable with Lipschitz continuous Hessian. That is, there exist a constant $0 \leq \LH < \infty $ such that for all $ (\xx,\yy) \in \real^{\dn \times \dn}$ we have 
	\begin{align}
		\vnorm{\nabla^2 f(\xx) - \nabla^2 f(\yy)} &\leq \LH \vnorm{\xx - \yy}. \label{eq:lipschitz_H}
	\end{align}
\end{assumption}

\cref{assmpt:lipschitz} is the only uniform smoothness assumption we make in this paper. Furthermore, unlike the analysis in \cite{royer2018complexity,royer2020newton,yao2021inexact}, we make no assumption on the boundedness of $ \gg $.

Recall the familiar relationship $ \TTt = \VVtT \HH \VVt $ where $ \TTt \in \real^{t \times t}$ is the symmetric tridiagonal matrix  obtained in the $ t\th $ iteration of \cref{alg:MINRES}  and $ \VVt \in \real^{d \times t} $ is the basis obtained from the Lanczos process for $ \Klvt(\HH,\gg) $; see \cref{sec:MINRES_details} for more details.
\cref{lemma:MINRES_NPC_detector} guarantees that as long as a NPC direction has not been encountered, we have $ \TTt \succ \zero $. 
\cref{assmpt:T_regularity} entails assuming that there is a small enough $ \sigma > 0 $ such that, as long as the NPC condition \cref{eq:NPC} has not been detected, $ \TTt $ remains uniformly positive definite. 
\begin{assumption}[Krylov Subspace Regularity Property I]
	\label{assmpt:T_regularity}
	There is a constant $ \sigma > 0 $ small enough such that for any $ \xx \in \real^{d} $, as long as the NPC condition \cref{eq:NPC} has not been detected, we have $\TTt \succeq \sigma \eye$. 
\end{assumption}
Clearly, \cref{assmpt:T_regularity} is equivalent to having $ \dotprod{\ss, \HH \ss} \geq \sigma \vnorm{\ss}^2, \; \forall \ss \in \Klv_{t}(\HH,\gg)$ as long as the NPC condition has not been detected, that is as long as $ \Klv_{t}(\HH,\gg) $ does not contain any NPC direction for $ \HH $. Since by \cref{lemma:MINRES_NPC_detector},  for all $ \xxk $, prior to the detection of the NPC direction we have $ \dotprod{\ss, \HHk \ss} > 0, \, \forall \ss \in \Klv_{t}(\HHk,\ggk)$, it follows that \cref{assmpt:T_regularity} trivially holds if we further assume that the sub-level set $ \mathcal{S}(\xx_{0};f) \defeq \left\{\xx \in \real^{d} \mid f(\xx) \leq f(\xx_{0})\right\} $ is compact, which is often made in similar literature, e.g., \cite{royer2018complexity,royer2020newton,yao2021inexact}.
%
We also note that \cref{assmpt:T_regularity} implies a local convexity property for $ f $ in a uniformly small ball within the Krylov subspace $ \Klvt(\HH,\gg) $ prior to the detection of the NPC condition \cref{eq:NPC}.
\begin{corollary}
	\label{coro:rel_conv}
	For any $ \xx \in \real^{d} $, suppose $ 1 \leq t \leq g $ be such that $ \Klvt(\HH(\xx),\gg(\xx)) $ does not contain any NPC direction for $ \HH(\xx) $. Under \cref{assmpt:T_regularity,assmpt:lipschitz},
	\begin{align*}
		f(\xx+\dd) \geq f(\xx) + \dotprod{\gg(\xx), \dd}, \quad \forall \dd \in \mathcal{D}(\xx),
	\end{align*}
	where $ \mathcal{D}(\xx) $ is a small ball within $ \Klvt(\HH(\xx),\gg(\xx)) $, defined as 
	\begin{align*}
		\mathcal{D}(\xx) \defeq \left\{\dd \in \Klvt(\HH(\xx),\gg(\xx)) \mid \vnorm{\dd} \leq \frac{2 \sigma}{\LH}\right\}.
	\end{align*}
\end{corollary}
\begin{proof}
	Recall that \cref{assmpt:T_regularity} is equivalent to having $ \dotprod{\dd, \HHk \dd} \geq \sigma \vnorm{\dd}^{2} $ for any $ \dd \in \mathcal{K}_{t}(\HH(\xx), \gg(\xx)) $. It follows that for any $ \dd \in \mathcal{D}(\xx) $,
	\begin{align*}
		f(\xx + \dd) - f(\xx) &=    \dotprod{\nabla f(\xx), \dd} + \hf    \int_{0}^{1} \dotprod{\dd, \nabla^2 f(\xx + t\dd) \dd} \df t \\
		&= \dotprod{\nabla f(\xx), \dd} + \hf    \int_{0}^{1} \dotprod{\dd, \left(\nabla^2 f(\xx + t\dd) - \nabla^2 f(\xx) + \nabla^2 f(\xx)\right) \dd} \df t \\
		&= \dotprod{\nabla f(\xx), \dd} + \hf    \int_{0}^{1} \dotprod{\dd, \nabla^2 f(\xx) \dd} \df t \\
		& + \hf    \int_{0}^{1} \dotprod{\dd, \left(\nabla^2 f(\xx + t\dd) - \nabla^2 f(\xx)\right) \dd} \df t \\
		&\geq \dotprod{\nabla f(\xx), \dd} + \frac{\sigma }{2} \vnorm{\dd}^2 - \frac{\LH}{4} \vnorm{\dd}^3 \geq \dotprod{\nabla f(\xx), \dd}.
	\end{align*}
\end{proof}

\subsection{Newton-MR With First-order Complexity Guarantee}
\label{sec:NewtonMR_first}
In this section, we present our first variant of Newton-MR, detailed in \cref{alg:NewtonMR_1st}, which offers first-order approximate optimality guarantee as in \cref{eq:small_g}. 
Recall from \cref{eq:sufficient_decrease_dTg} that, as long as $ \ggk \neq \zero $, \cref{alg:MINRES} returns a search direction, $ \ddk $, which is always guaranteed to be a descent direction. As a result, the step-size, $ \alphak $, can be chosen using the Armijo-type line-search by finding the largest $ \alphak $ that satisfies 
\begin{align}
	\label{cond:Armijo}
	f(\xxk + \alphak \ddk) \leq f(\xxk) + \rho \alpha_{k} \dotprod{\ggk, \ddk},
\end{align}
where $ \rho > 0 $ is some appropriately chosen line-search parameter. This is often approximately achieved using a back-tracking line-search strategy as in \cref{alg:line_search}. 

\begin{algorithm}[!htbp]
	\caption{Backward Tracking Line-Search}
	\label{alg:line_search}
	\begin{algorithmic}[1]
		\vspace{1mm}
		\STATE \textbf{Input:} $ \alpha = 1 $, $ 0 < \zeta < 1 $
		\vspace{1mm}
		\WHILE {Line-search criterion is not satisfied with $ \alpha $}
		\vspace{1mm}
		\STATE $ \alpha = \zeta \alpha $
		\vspace{1mm}
		\ENDWHILE
		\vspace{1mm}
		\RETURN $ \alpha $
	\end{algorithmic}
\end{algorithm}

In \cref{alg:NewtonMR_1st}, when $ \dt =$ ``SOL'', we set $ \alpha = 1 $ as the initial trial step-size for the backtracking line-search procedure. However, when $ \dt =$ ``NPC'', the proof of  \cref{lemma:NC_distance_opt} below highlights the fact that the chosen step-size may in fact be much larger. In this light, when $ \ddk $ is a direction of non-positive curvature, instead of back-tracking line-search, we can employ a forward-tracking procedure, depicted in \cref{alg:line_search_forward}, to search for larger step-sizes that satisfy the line-search criterion \cref{cond:Armijo}. We note that the notion of forward-tracking procedure has been considered in the literature before in various contexts, e.g., \cite{gould2000exploiting}.

\begin{algorithm}[!htbp]
	\caption{Forward/Backward Tracking Line-Search}
	\label{alg:line_search_forward}
	\begin{algorithmic}[1]
		\vspace{1mm}
		\STATE \textbf{Input:} $ \alpha > 0 $, $ 0 < \zeta < 1 $
		\vspace{1mm}
		\IF {Line-search criterion is not satisfied with $ \alpha $} 
		\vspace{1mm}
		\STATE \text{Call \cref{alg:line_search}}
		\vspace{1mm}
		\ELSE
		\vspace{1mm}
		\WHILE {Line-search criterion is satisfied with $ \alpha $}
		\vspace{1mm}
		\STATE $ \alpha = \alpha/\zeta $
		\vspace{1mm}
		\ENDWHILE
		\vspace{1mm}
		\RETURN $ \zeta \alpha $
		\vspace{1mm}
		\ENDIF
	\end{algorithmic}
\end{algorithm}

\begin{remark}
	The initial trial step-size for \cref{alg:line_search} is set as $ \alpha = 1$. The motivation for this choice is based on the following observation. Suppose MINRES returns a solution direction $ \ddk $ prior to detection of the NPC condition. Consider the second-order Taylor expansion of $ f $ at $ \xxk $ as
	\begin{align*}
		\phi(\alpha) \defeq f(\xxk) + \alpha \dotprod{\ddk, \ggk} + \frac{\alpha^2}{2} \dotprod{\ddk, \HHk \ddk},
	\end{align*}
	which is often taken as a local approximation to $ f $ in many Newton-type algorithms.
	By the \cref{lemma:sTr}, which is a unique property of MINRES, we can upper bound the mismatch between $ \phi(\alpha) $ and $ f(\xxk) $ as
	\begin{align*}
		\phi(\alpha) - f(\xxk) \leq  \left( \frac{\alpha^2}{2} - \alpha \right) \dotprod{\ddk, \HHk \ddk}.
	\end{align*}
	Clearly, $ \alpha = 1 $ is the minimizer of this upper bound, which implies that $ \alpha = 1 $ may, in some sense, be an optimal choice for the initial step-size in Newton-MR. This choice is also often made for the Newton-CG and many quasi-Newton methods. However, the motivation for such a choice for those algorithms is entirely different than the discussion above and is based on achieving quadratic/super-linear local convergence rate. This also raises an interesting research direction to quantify the local convergence of Newton-MR with unit step-size.
\end{remark}

\begin{algorithm}[!htbp]
	\caption{Newton-MR With First-order Complexity Guarantee}
	\label{alg:NewtonMR_1st}
	\begin{algorithmic}[1]
		\vspace{1mm}
		\STATE \textbf{Input:} 
		\vspace{1mm}
		\begin{itemize}[leftmargin=*,wide=0em, noitemsep,nolistsep,label = {\bfseries -}]
			\item Initial point:    $ \xx_{0} $
			\item Approximate first-order optimality tolerance: $ 0 < \epsg \leq 1 $
			\item Inexactness tolerance: $ \theta > 0 $
		\end{itemize}
		\vspace{1mm}
		\STATE $ k = 0 $
		\vspace{1mm}
		\WHILE{$ \| \ggk \| > \epsg $}
		\vspace{1mm}
		\STATE Call \cref{alg:MINRES} as $\left[\ddk,\dt\right] = \text{MINRES}(\HHk, \ggk, \theta \sqrt{\epsg})$
		\vspace{1mm}
		\IF{ $ \dt = \text{`SOL'} $}
		\vspace{1mm}
		\STATE Find the largest $ 0 < \alphak \leq 1$ satisfying \cref{cond:Armijo} with $ 0 < \rho < 1/2 $ using \cref{alg:line_search}
		\ELSE
		\vspace{1mm}
		\STATE Find the largest $ \alphak > 0$ satisfying \cref{cond:Armijo} with $ 0 < \rho < 1 $ using \cref{alg:line_search_forward}
		\vspace{1mm}
		\ENDIF
		\vspace{1mm}
		\STATE $ \xx_{k+1} = \xxk + \alphak \ddk $
		\vspace{1mm}
		\STATE $ k = k + 1 $
		\vspace{1mm}
		\ENDWHILE
		\vspace{1mm}
		\STATE \textbf{Output:} $ \xxk $ satisfying first-order optimality \cref{eq:small_g}
	\end{algorithmic}
\end{algorithm}

\subsubsection{Optimal Iteration Complexity}
\label{sec:opt}
In this section, we provide the first-order iteration complexity analysis of \cref{alg:NewtonMR_1st} and show that it achieves the optimal rate of $ \bigO{\epsg^{-3/2}} $. 

Typically, one might expect the size of the update direction to be directly correlated with the magnitude of the gradient. Indeed, from \cite[Lemma 3.11]{liu2021convergence}, we have $ \| \HHk \sskt \| \leq \| \ggk \| $. Hence, by \cref{assmpt:T_regularity}, 
\begin{align*}
	\vnorm{\sskt} \leq \frac{\dotprod{\sskt, \HHk \sskt}}{\sigma \vnorm{\sskt}} \leq \frac{\vnorm{\HHk \sskt}}{\sigma} \leq \frac{\vnorm{\ggk}}{\sigma},
\end{align*}
as long as no NPC direction has been detected. An important ingredient of our analysis is to establish a converse, which will then allow us to obtain a bound on a worst-case decrease in the objective value.
Suppose MINRES returns an iterate $ \ddk = \ssktp $ that satisfies the inexact condition \cref{eq:inexactness}. \cref{lemma:s_norm_opt} shows that magnitude of the direction serves as an estimate on the gradient norm at the point $ \xxkn = \xxk + \ddk $.
\begin{lemma}
	\label{lemma:s_norm_opt}
	Under \cref{assmpt:Lg,assmpt:lipschitz,assmpt:T_regularity}, if \cref{eq:inexactness} is satisfied with $\eta = \theta \sqrt{\epsg} $ for some $ \theta > 0 $, we have
	\begin{align*}
		\vnorm{\ssktp} \geq c_0  \min\left\{\frac{\vnorm{\nabla f(\xxk + \ssktp)}}{\sqrt{\epsg}}, \sqrt{\epsg}\right\},
	\end{align*}
	where
	\begin{align*}
		c_0 \defeq \left( \frac{2 \sigma}{\theta^{2} \Lg^{2} + \sqrt{\theta \Lg + 2 \sigma^{2} L_{\HH}}}\right),
	\end{align*}
	and $ \Lg$, $\LH $, and $ \sigma $ are as in \cref{assmpt:Lg,assmpt:lipschitz,assmpt:T_regularity}.
\end{lemma}
\begin{proof}   
	Firs, we note that since $ \rrktp \in \Klv_{t}(\HHk,\ggk) $ and the NPC condition \cref{eq:NPC} has not yet been detected at iteration $ t $, \cref{assmpt:T_regularity} implies that
	\begin{align*}
		\vnorm{\rrktp}^{2} \leq \dotprod{\rrktp,\HHk \rrktp}/\sigma \leq \vnorm{\rrktp} \vnorm{\HHk\rrktp}/\sigma,
	\end{align*}
	which gives $ \vnorm{\rrktp} \leq \vnorm{\HHk\rrktp}/\sigma $.
	Denote $ \gg_{k+1} \defeq \nabla f(\xxk + \ssktp) $. From \cref{assmpt:Lg,assmpt:lipschitz,eq:inexactness}, we have
	\begin{align*}
		\vnorm{\gg_{k+1}} &= \vnorm{\gg_{k+1} - \ggk - \HHk \ssktp - \rrktp} \leq \vnorm{\gg_{k+1} - \ggk - \HHk \ssktp} + \vnorm{\rrktp} \\
		&\leq \frac{\LH}{2} \vnorm{\ssktp}^{2} + \frac{\vnorm{\HHk \rrktp}}{\sigma} \leq \frac{\LH}{2} \vnorm{\ssktp}^{2} + \frac{\eta \vnorm{\HHk \ssktp}}{\sigma} \\
		&\leq \frac{\LH}{2} \vnorm{\ssktp}^{2} + \frac{\eta \Lg \vnorm{\ssktp}}{\sigma} = \frac{\LH}{2} \vnorm{\ssktp}^{2} + \frac{\theta \sqrt{\epsg}\Lg \vnorm{\ssktp}}{\sigma}
	\end{align*}
	Now, by solving
	\begin{align*}
		\sigma L_{\HH} \vnorm{\ssktp}^2 + 2 \theta \sqrt{\epsg} \Lg \vnorm{\ssktp} - 2\sigma \vnorm{\gg_{k+1}}\geq 0,
	\end{align*}
	we obtain
	\begin{align*}
		\vnorm{\ssktp} &\geq \frac{- 2 \theta \sqrt{\epsg} \Lg + \sqrt{4 \theta^{2} \epsg \Lg^{2} + 8 \sigma^{2} L_{\HH} \vnorm{\gg_{k+1}}}}{2 \sigma L_{\HH}} \\
		&= \left( \frac{- \theta \Lg + \sqrt{\theta^{2} \Lg^{2} + 2\sigma^{2} L_{\HH} \vnorm{\gg_{k+1}}/\epsg}}{\sigma L_{\HH}}\right) \sqrt{\epsg} \\
		&= \left( \frac{2 \sigma \vnorm{\ggkk}/\epsg}{\theta \Lg + \sqrt{\theta^{2} \Lg^{2} + 2 \sigma^{2} L_{\HH} \vnorm{\gg_{k+1}}/\epsg}}\right) \sqrt{\epsg} \\
		&\geq \left( \frac{2 \sigma}{\theta \Lg + \sqrt{\theta^{2} \Lg^{2} + 2 \sigma^{2} L_{\HH}}}\right) \sqrt{\epsg} \min\left\{\vnorm{\gg_{k+1}}/\epsg, 1\right\}.
	\end{align*}    
\end{proof}

The next lemma gives the worst-case amount of descent obtained by using a solution direction satisfying \cref{eq:inexactness}.
\begin{lemma}
	\label{lemma:SOL_distance_opt}
	Suppose \cref{alg:MINRES} returns $ \dt = \text{`SOL'} $. Under \cref{assmpt:Lg,assmpt:lipschitz,assmpt:T_regularity}, in \cref{alg:NewtonMR_1st}, we have
	\begin{align*}
		f(\xxkn) \leq f(\xxk) - \min \left\{c_{1}, c_{2} \epsg, c_{2} {\vnorm{\ggkk}^{2}}/{\epsg} \right\},
	\end{align*}
	where
	\begin{align*}
		c_1 &\defeq \rho \sigma \left(\frac{3 \sigma \left( 1 - 2 \rho\right)}{\LH}\right)^{2}, \quad c_2 \defeq \rho \sigma c_{0}^{2},
	\end{align*}
	$ c_0 $, $ \LH $, and $ \sigma $ are defined, respectively, in \cref{lemma:s_norm_opt,assmpt:T_regularity,assmpt:lipschitz}, and $ 0 < \rho < 1/2 $ is the line-search parameter. 
\end{lemma}
\begin{proof}
	Since $ \dt = \text{`SOL'} $, we have $ \ddk = \ssktp $. 
	From \cref{assmpt:lipschitz}, for any $ 0 < \alpha \leq 1 $, 
	\begin{align*}
		f(\xxk + \alpha \ddk) - f(\xxk) -\alpha \rho \dotprod{\ddk,\ggk} &\leq \alpha \dotprod{\ddk, \ggk} + \frac{\alpha^2}{2} \dotprod{\ddk, \HH \ddk}  + \frac{\LH}{6} \alpha^3 \vnorm{\ddk}^3 - \alpha \rho \dotprod{\ddk,\ggk}\\
		&\leq \alpha (1-\rho) \dotprod{\ddk, \ggk} + \frac{\alpha}{2} \dotprod{\ddk, \HH \ddk}  + \frac{\LH}{6} \alpha^3 \vnorm{\ddk}^3 \\
		&\leq \frac{\alpha(1-2\rho)}{2} \dotprod{\ddk, \ggk} + \frac{\alpha}{2} \left( \dotprod{\ddk, \ggk} + \dotprod{\ddk, \HH \ddk} \right) \\
		& + \frac{\LH}{6} \alpha^3 \vnorm{\ddk}^3 \\
		&\leq \frac{\alpha(1-2\rho)}{2} \dotprod{\ddk, \ggk} + \frac{\LH}{6} \alpha^3 \vnorm{\ddk}^3 \\
		&\leq -\frac{\alpha(1-2\rho)}{2} \dotprod{\ddk, \HHk\ddk} + \frac{\LH}{6} \alpha^3 \vnorm{\ddk}^3\\
		&\leq -\frac{\alpha \sigma (1-2\rho)}{2} \vnorm{\ddk}^{2} + \frac{\LH}{6} \alpha^3 \vnorm{\ddk}^3,
	\end{align*}
	where the last two inequalities follows from \cref{lemma:sTr,assmpt:T_regularity}, respectively.
	To satisfy the line-search criterion \cref{cond:Armijo}, we require that the right hand side is negative, which implies that for the largest $ \alphak $ satisfying the line-search condition \cref{cond:Armijo}, we must have 
	\begin{align*}
		\alphak \geq \min\left\{1,\sqrt{\frac{3 \sigma \left( 1 - 2 \rho\right)}{\LH \vnorm{\ddk}}}\right\}.
	\end{align*}    
	If 
	\begin{align*}
		\vnorm{\ddk} \geq \frac{3 \sigma \left( 1 - 2 \rho\right)}{\LH},
	\end{align*}
	it follows that
	\begin{align*}
		f(\xxk + \alphak \ddk) &\leq f(\xxk) + \alphak \rho \dotprod{\ddk,\ggk} \leq f(\xxk) - \rho \sqrt{\frac{3 \sigma \left( 1 - 2 \rho\right)}{\LH \vnorm{\ddk}}} \dotprod{\ddk,\HHk\ddk} \\
		&\leq f(\xxk) - \rho \sigma  \sqrt{\frac{3 \sigma \left( 1 - 2 \rho\right)}{\LH}} \vnorm{\ddk}^{3/2} \leq f(\xxk) - \rho \sigma \left(\frac{3 \sigma \left( 1 - 2 \rho\right)}{\LH}\right)^{2}.
	\end{align*} 
	Otherwise, we have $ \alphak = 1 $, and 
	\begin{align*}
		f(\xxk + \ddk) &\leq f(\xxk) - \rho \sigma  \vnorm{\ddk}^{2} \leq f(\xxk) - \rho \sigma c_{0}^{2} \min\left\{\frac{\vnorm{\nabla f(\xxk + \ddk)}^{2}}{\epsg}, \epsg\right\},
	\end{align*}
	where the last inequality follows from \cref{lemma:s_norm_opt}.
\end{proof}

Suppose the NPC condition \cref{eq:NPC} is detected at iteration $ t $, before the inexactness condition \cref{eq:inexactness} with some $ \eta = \theta \sqrt{\epsg} > 0 $ is satisfied. By $ \| \HH \rrtp \| > \eta \| \HH \sstp \| $ and by \cref{assmpt:Lg,eq:Hs}, we must have $ \|\rrktp\| > \eta/\sqrt{\Lg^2 + \eta^2} \vnorm{\ggk} $. Our choice of the inexactness tolerance $ \eta = \theta \sqrt{\epsg} $ could raise suspicion that we must have $ \|\rrktp\| \in \Omega(\sqrt{\epsg}) $. However, we now argue that a lower bound for NPC direction must be independent of $ \epsg $. For the NPC direction, $ \rrktp $, we clearly have $ \|\rrktp\| > 0 $. In fact, $ \|\rrktp\| > \|(\eye - \HHk\HHdk)\ggk \| $, and hence as long as $ \ggk \notin \range(\HHk)$, we obtain a lower-bound on $ \|\rrktp\|/\vnorm{\ggk} $, which is independent of the inexactness tolerance as well as the magnitude of $ \ggk $. More generally, from the construction of \cref{alg:MINRES}, we can clearly see that all of the underlying quantities are built from $ \HHk $ and $ \vv_1 = \ggk/\vnorm{\ggk} $. That is, for all $ t \geq 1 $, the relevant factors such as $ \talpha_{t} $, $ \tbeta_{t+1} $, and hence $ \gamma_{t} $, $ s_t $, and $ c_{t} $, detailed in \cref{sec:MINRES_review}, are all independent of inexactness tolerance $ \eta $ and the magnitude of $ \ggk $. As a result, the relative residual $ \| \rrktp \| / \| \ggk \| = \phi_{t}/\phi_{0} = s_1 s_2 \ldots s_{t-1} $ and the NPC condition \cref{eq:NPC} (or equivalently \cref{eq:NPC_cond}) all enjoy the same independence. In addition, the maximum number of iterations to detect an NPC direction, i.e., $ \TN $ in \cref{eq:complexity_NC} below, is also independent of these two factors.
Let us further investigate these observations with some examples. 
\begin{example}
	\label{ex:residual_NPC}
	Consider the following example 
	\begin{align*}
		\HH = \begin{bmatrix}
			\Lg & \\
			& -\mu
		\end{bmatrix}, \quad \gg = -\begin{bmatrix}
			1 \\
			\beps
		\end{bmatrix}, \quad \uu_1 = \begin{bmatrix}
			1 \\
			0
		\end{bmatrix}, \quad \uu_2 = \begin{bmatrix}
			0 \\
			1
		\end{bmatrix}.
	\end{align*}
	Clearly, as $ \beps \rightarrow 0 $, $ \gg $ becomes increasingly orthogonal to $ \uu_2 $. On the other hand, when $ \beps \rightarrow \pm \infty $, $ \gg $ become more aligned with $ \uu_{2} $. Following \cref{alg:MINRES}, we have
	\begin{align*}
		\rr_{0} = - \gg = \begin{bmatrix}
			1 \\
			\beps
		\end{bmatrix}, \quad \vv_1 = \frac{1}{\sqrt{1+\beps^2}}\begin{bmatrix}
			1 \\
			\beps
		\end{bmatrix}.
	\end{align*}
	For $ t = 1 $, we have 
	\begin{align*}
		\talpha_{1} = \vv_1^\T \HH \vv_1 = \frac{\Lg - \mu \beps^2}{1+\beps^2} = \gamma_1, \quad \tbeta_2 = \frac{\beps(\Lg + \mu)}{1+\beps^2}, \quad \vv_2 = \frac{1}{\sqrt{1+\beps^2}}\begin{bmatrix}
			\beps \\
			-1
		\end{bmatrix}
	\end{align*}
	and
	\begin{align*}
		\gamma_{1}^{[2]} = \sqrt{\frac{\Lg^2 + \beps^2 \mu^2}{1+\beps^2}}, \quad c_1 = \frac{\Lg - \mu \beps^2}{\sqrt{(1+\beps^2)(\Lg^2 + \beps^2 \mu^2)}}, \quad s_1 = \frac{\beps(\Lg + \mu)}{\sqrt{(1+\beps^2)(\Lg^2 + \beps^2 \mu^2)}}
	\end{align*}
	Now, considering $ t = 2 $, we get
	\begin{align*}
		\delta_2 = \tbeta_2, \quad \talpha_2 = \vv_2^\T \HH \vv_2 = \frac{\Lg \beps^2 - \mu}{1+\beps^2}, \quad \gamma_2 = s_1 \delta_2 - c_1 \talpha_2 = \Lg \mu \sqrt{\frac{1+\beps^2}{\Lg^2 + \beps^2 \mu^2}},
	\end{align*}
	and $ \tbeta_3 = 0 $, which implies the algorithm terminates after two iterations. Consider the NPC condition \cref{eq:NPC} (or equivalently \cref{eq:NPC_cond}) for $ t=1 $ and $ t=2 $ as
	\begin{align*}
		-c_0 \gamma_1 = \frac{\Lg - \mu \beps^2}{1+\beps^2}, \quad -c_1 \gamma_2 = -\left(\Lg - \mu \beps^2\right) \frac{\Lg \mu}{\Lg^2 + \beps^2 \mu^2}.
	\end{align*}
	For $ \rr_0 $ to be a NPC direction, we need to have $ \beps^2 \geq \Lg / \mu $. In this case, $ \rr_{0} = \gg $ and we have the relative residual $ \| \rr_0 \| / \| \gg \| = 1 $. Otherwise, if $ \beps^2 \leq \Lg / \mu $, then $ \rr_{1} $ is a NPC direction. In this case, the angle between the direction of $ \gg $ and $ \uu_2 $ can be represented as
	\begin{align*}
		\angle_2 = \abs{\dotprod{\frac{\gg}{\vnorm{\gg}}, \uu_2}} = \abs{\frac{\beps}{\sqrt{1+\beps^2}}}.
	\end{align*}
	Hence, we have $ \beps^2 \geq {\angle_2^2}/{(1-\angle_2^2)} $, and 
	\begin{align*}
		\frac{\vnorm{\rr_1}}{\vnorm{\gg}} = s_1 = \frac{\Lg + \mu}{\sqrt{(1/\beps^2+1)(\Lg^2/\beps^2 + \mu^2)}} \geq \frac{\angle_2(\Lg + \mu)}{\sqrt{\Lg^2(1-\angle_2^2)/\angle_2^2 + \mu^2}},
	\end{align*}
	which is independent of either $ \eta $ or the magnitude of $ \gg $.
\end{example}

\begin{example}
	\label{exp:MR}
	Consider a random matrix $ \HH $ with $ \dn = 20 $, containing $ 18 $ positive eigenvalues, $ 1 $ negative eigenvalue, and a one dimensional null-space. We consider the following vectors. Let $ \gg_1 = \mathbf{1} $ be the vector of all ones, $ \gg_2 $ be a random vector drawn from the multivariate normal distribution $\mathcal{N}(0,1) $. We scale $ \gg_{2} $ appropriately such that $ \| \gg_2 \| = \| \gg_1 \| $. Also, let $ \gg_3 = 10 \times \gg_1 $ and $ \gg_4 = 10 \times \gg_2 $. 
	
	According to our observations above, the smallest eigenvalue of $ \TTt $ using either $ \gg_1 $ or $ \gg_3 $ must be identical. Similarly, $ \gg_2 $ and $ \gg_4 $ generate identical Krylov sub-spaces, and hence must detect a NPC direction in an identical manner. \cref{fig:mr} depicts the results of \cref{alg:MINRES} using these four vectors. 
	While the residual itself depends on the choice of $ \gg_1 $ or $ \gg_3 $, the relative residual and $ \lambda_{\min}(\TTt) $ behave identically, resulting in the detection of the NPC condition \cref{eq:NPC} at the same iterations for both. The same is also seen to be true for $ \gg_2 $ and $ \gg_4 $. These examples verify our observation that $ \vnorm{\rrtp}/\vnorm{\gg} $ is independent of the magnitude of $ \gg$.  
	Note that a zero curvature direction is detected at the last iteration for all four examples, which is expected as per \cref{lemma:T_PSD}-\labelcref{lemma:T_PSD_ii}. 
\end{example}
\begin{figure}[!htbp]
	\centering
	\includegraphics[width=1\textwidth]{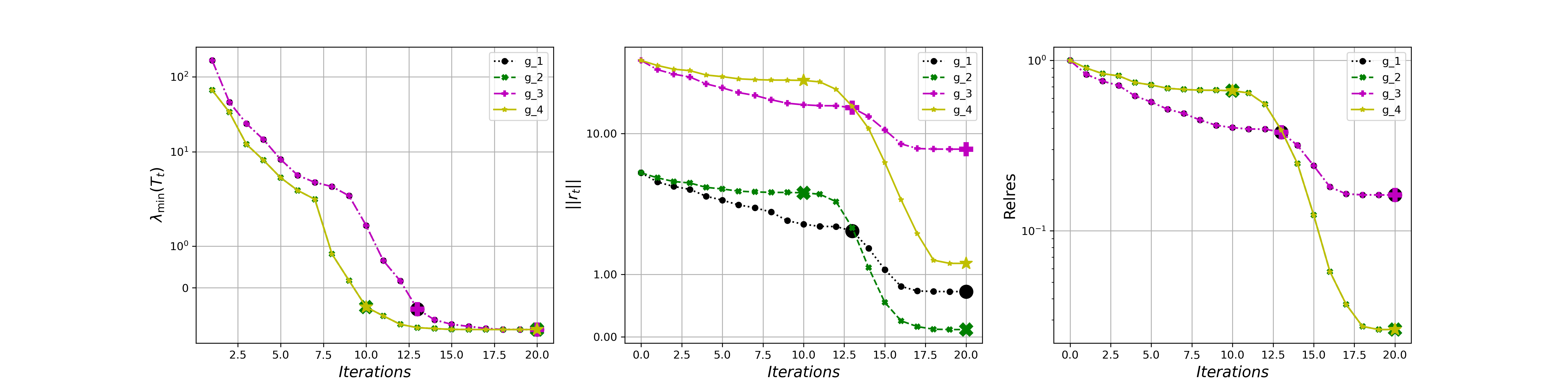}
	\caption{The behavior of \cref{alg:MINRES} in \cref{exp:MR}. Iterations where a NPC direction is detected are highlighted by enlarged marks. \label{fig:mr}}
\end{figure}

These observations lead us to safely conclude that there must exist $ \omega > 0 $, not depending on $ \eta $ or $ \ggk $, for which we have $ \vnorm{\rrktp} \geq \omega \vnorm{\ggk} $. \cref{assmpt:residual_NPC} entails the existence of a uniform $ \omega $ for all iterates of \cref{alg:NewtonMR_1st}.
\begin{assumption}[Krylov Subspace Regularity Property II]
	\label{assmpt:residual_NPC}
	There exist some $ \omega > 0 $ such that for any $ \xxk $ from \cref{alg:NewtonMR_1st}, the NPC direction satisfies $ \vnorm{\rrktp} \geq \omega \vnorm{\ggk} $.
\end{assumption}

Similar to \cref{lemma:SOL_distance_opt}, \cref{lemma:NC_distance_opt} gives a worst-case estimate on the reduction in $ f $ obtained with a NPC direction.
\begin{lemma}
	\label{lemma:NC_distance_opt}   
	Suppose \cref{alg:MINRES} returns $ \dt = \text{`NPC'} $. 
	Under \cref{assmpt:Lg,assmpt:lipschitz,assmpt:residual_NPC}, in \cref{alg:NewtonMR_1st}, we have
	\begin{align*}
		f(\xxkn) - f(\xxk) \leq -c_3 \vnorm{\ggk}^{3/2},
	\end{align*}
	where
	\begin{align*}
		c_3 \triangleq \rho \omega^{3/2} \sqrt{\frac{6  \left( 1-\rho\right)}{\LH}},
	\end{align*}
	$ \LH $, $ \omega $ are defined in \cref{assmpt:lipschitz,assmpt:residual_NPC}, respectively, and $ 0 < \rho < 1 $ is the line-search parameter. 
\end{lemma}
\begin{proof}       
	Since $ \dt = \text{`NPC'} $, we have $ \ddk = \rrktp$ and we must also have $\vnorm{\ddk} \geq \omega \vnorm{\ggk}$. From, \cref{eq:rTg}, we get
	\begin{align*}
		\dotprod{\ddk, \ggk} = \dotprod{\rrktp, \ggk} = - \vnorm{\rrktp}^2 = -\vnorm{\ddk}^{2}.
	\end{align*}
	From \cref{assmpt:lipschitz}, it follows that, for any $ \alpha > 0 $,
	\begin{align*}
		f(\xxk + \alpha \ddk) - f(\xxk) &\leq \alpha \dotprod{\ddk, \ggk} + \frac{\alpha^2}{2} \dotprod{\ddk, \HHk \ddk}  + \frac{\LH}{6} \alpha^3 \vnorm{\ddk}^3 \\
		&\leq - \alpha \vnorm{\ddk}^{2} + \frac{\LH}{6} \alpha^3 \vnorm{\ddk}^3.
	\end{align*}
	To satisfy the line-search criterion \cref{cond:Armijo}, we require that $ \alpha > 0 $ satisfies
	\begin{align*}
		- \alpha \vnorm{\ddk}^{2} + \frac{\LH}{6} \alpha^3 \vnorm{\ddk}^3 &\leq \alpha \rho \dotprod{\ddk,\ggk} = - \alpha \rho \vnorm{\ddk}^2,
	\end{align*}
	which implies that for the largest $ \alphak $ satisfying the line-search condition \cref{cond:Armijo}, we must have 
	\begin{align*}
		\alphak \geq \sqrt{\frac{6 \left( 1-\rho\right)}{\LH \vnorm{\ddk}}}.
	\end{align*}    
	It follows that
	\begin{align*}
		f(\xxk + \alphak \ddk) &\leq f(\xxk) + \alphak \rho \dotprod{\ddk,\ggk} \leq f(\xxk) - \rho \sqrt{\frac{6  \left( 1-\rho\right)}{\LH \vnorm{\ddk}}} \vnorm{\ddk}^{2}\\
		&= f(\xxk) - \rho \sqrt{\frac{6  \left( 1-\rho\right)}{\LH}} \vnorm{\ddk}^{3/2} \\
		&\leq f(\xxk) - \rho \omega^{3/2} \sqrt{\frac{6  \left( 1-\rho\right)}{\LH}} \vnorm{\ggk}^{3/2}.
	\end{align*} 
\end{proof}
Note that the result of \cref{lemma:NC_distance_opt} holds as long as the direction $ \ddk $ is a non-positive curvature direction. In other words, not only does a negative curvature direction can yield descent in $ f $, but also any direction of zero curvature can also be used to obtain reduction in the objective value.

Now, altogether, we can obtain the following optimal complexity result for \cref{alg:NewtonMR_1st}.

\begin{theorem}[Optimal Iteration Complexity of \cref{alg:NewtonMR_1st}]
	\label{thm:complexity_first_alg_opt}
	Under \cref{assmpt:Lg,assmpt:lipschitz,assmpt:T_regularity,assmpt:residual_NPC}, after at most
	\begin{align*}
		K \defeq \frac{\left(f(\xx_{0}) - f^{\star}\right) \epsg^{-3/2}}{\min \left\{ c_{1}, c_{2}, c_3\right\}},
	\end{align*}
	iterations of \cref{alg:NewtonMR_1st}, the approximate first-order optimality \cref{eq:small_g} is satisfied. Here, $c_{1}$, $c_{2}$ and $c_{3}$, are defined in \cref{lemma:SOL_distance_opt,lemma:NC_distance_opt}, and $ f^{\star} \defeq \min f(\xx) $.
\end{theorem}
\begin{proof}
	Suppose $ \| \ggk \| \leq \epsg $, for the first time, at $ k = K + 1 $. Then, we must have $ \vnorm{\gg_{k}} > \epsg, \; k = 0, \ldots, K $. Suppose at some iteration $ 0 \leq k \leq K-1 $, we have $ \dt = \text{`SOL'} $ in which case by \cref{lemma:SOL_distance_opt}, we get 
	\begin{align*}
		f(\xx_{k}) - f(\xx_{k+1}) \geq \min \left\{c_{1}, c_{2} \epsg, c_{2} {\vnorm{\ggkk}^{2}}/{\epsg} \right\} \geq  \min \left\{c_{1}, c_{2}\right\} \epsg.
	\end{align*}    
	Alternatively, if $ \dt = \text{`NPC'} $, then by \cref{lemma:NC_distance_opt}, we have 
	\begin{align*}
		f(\xx_{k}) - f(\xx_{k+1}) \geq c_3 \vnorm{\ggk}^{3/2} \geq c_3 \epsg^{3/2}.
	\end{align*}    
	It follows that 
	\begin{align*}
		f(\xx_{0}) - f(\xx_{K}) &\geq \sum_{k=0}^{K-1} f(\xx_{k}) - f(\xx_{k+1}) \\
		&> \sum_{k=0}^{K} \min \left\{\min \left\{ c_{1}, c_{2}\right\} \epsg, c_3 \epsg^{3/2}\right\} \\
		&> K \min \left\{ c_{1}, c_{2},c_3\right\} \epsg^{3/2} > f(\xx_{0}) - f^{\star},
	\end{align*}
	This implies $ f^{\star} > f(\xx_{K}) $, which is a contradiction. 
\end{proof}

\subsubsection{Optimal Operation Complexity}
\label{sec:op_comp}
In order to establish a complexity bound for \cref{alg:MINRES} to return a direction that either satisfies the inexactness condition \cref{eq:inexactness} or the NPC condition \cref{eq:NPC}, we will now make a few assumptions to uniformly bound some of the quantities that appear in \cref{sec:MINRES_complexity}.

\begin{assumption}
	\label{assmpt:prop}
	Recall the decomposition in \cref{eq:H_decomp_PSD,eq:H_decomp_PSD_ij}. There exists $\mu>0$, and $ \Lg^{2}/\left(\Lg^{2} + \theta^{2} \epsg \right) \leq \nu \leq 1 $ such that, for any $ \xx \in \real^{d} $, if $ \gg \notin \Null(\HH) $, then \underline{at least one} of the following holds. 
	\begin{enumerate}[label = {\bfseries (\roman*)}]
		\item \label{assmpt:prop:pseudo} If $ \psip \geq 1 $ and $ \psin \geq 1 $, there exists some $ 1 \leq j \leq \psip$ and $ \psip+\psiz+1 \leq j \leq \psi$ for which 
		\begin{subequations}
			\begin{align}
				\min\left\{\lambda_{i}, \abs{\lambda_{j}}\right\} &\geq \mu  \label{eq:property_pseudo}\\
				\vnorm{\left(\UU_{i+} \UU_{i+}^{\T} + \UU_{j-} \UU_{j-}^{\T}\right) \gg}^{2} &\geq \nu \vnorm{\UU\UUT\gg}^{2} \label{eq:property_null}
			\end{align}
			\item \label{assmpt:prop:pos} If $ \psip \geq 1 $, there exists some $ 1 \leq i \leq \psip$, for which 
			\begin{align}
				\lambda_{i} &\geq \mu \label{eq:property_pseudo_pos}\\
				\vnorm{\UU_{i+} \UU_{i+}^{\T} \gg}^{2} &\geq \nu \vnorm{\UU\UUT\gg}^{2}   \label{eq:property_null_pos}
			\end{align}
			\item \label{assmpt:prop:neg} If $ \psin \geq 1 $, there exists some $ \psip+\psiz+1 \leq j \leq \psi$, for which 
			\begin{align}
				\abs{\lambda_{j}} &\geq \mu \label{eq:property_gap}\\ 
				\vnorm{\UU_{j-} \UU_{j-}^{\T} \gg}^{2} &\geq \nu \vnorm{\UU\UUT \gg}^{2}   \label{eq:property_null_neg}
			\end{align}
		\end{subequations}
	\end{enumerate}
\end{assumption}
\cref{assmpt:prop} is relatively mild. In essence, \cref{assmpt:prop} only necessitates that $ \HH $ has \emph{some} $\gg$-relevant eigenvalue which is sufficiently large, and $ \gg $ has non-trivial projection on the corresponding eigen-space. 
In fact, \cref{eq:property_pseudo,eq:property_null} are significant relaxations of \cite[Assumptions 3 and 4]{roosta2018newton}. Indeed, the regularity conditions \cref{eq:property_pseudo,eq:property_null} are only required to hold for some $ 1 \leq i \leq \psip$ and $ \psip+\psiz+1 \leq j \leq \psi$ among the $\gg$-related spectrum of $ \HH $, whereas in \cite{roosta2018newton}, such conditions are assumed for the entire spectrum of $ \HH $. 
More importantly, \cite[Assumption 4]{roosta2018newton} partitions $ \real^{d} $ into to complementary spaces, namely $ \range(\HH) $ and its orthogonal complement, $ \Null(\HH) $, and then relates the magnitude of the projection of $ \gg $ on these two sub-spaces. Assumption \cref{eq:property_null} does a similar partitioning but restricted to the sub-space $ \range(\UU) $, as opposed to the entire $ \real^{d} $. More specifically, \cref{eq:property_null} merely relates the magnitude of the projection of $ \gg $ on the sub-space $ \UU_{i+} $ and/or $ \UU_{j-} $ to that on the complement of this space with respect to $ \UU $. 
In fact, \cref{eq:property_null} is always trivially satisfied with $ \nu = 1 $, $ i=\psip $ and $ j = \psip+\psiz+1 $. As a result, replacing the entirety of \cref{assmpt:prop} with a stronger requirement that $ \vnorm{\UUT \HHd \UU} \leq 1/\mu $ (or equivalently $ \min\left\{\lambda_{\psip}, \abs{\lambda_{\psip+\psiz+1}}\right\}  \geq \mu$) is entirely sufficient to establish all the results of this paper with $ \nu =1 $. 
%
We also note that none of the properties \labelcref{assmpt:prop:pseudo,assmpt:prop:pos,assmpt:prop:neg} in \cref{assmpt:prop} is strictly weaker/stronger than the other two. In fact, we can see that either condition \cref{eq:property_null_pos} or \cref{eq:property_null_neg} implies \cref{eq:property_null}, where as \cref{eq:property_pseudo} is stronger than either of \labelcref{eq:property_pseudo_pos} or \labelcref{eq:property_gap}.

Under Assumptions \ref{assmpt:Lg} and \ref{assmpt:prop}-\labelcref{assmpt:prop:neg}, \cref{cor:MINRES_complexity_NC} implies that if $ \psin \geq 1 $, then after at most 
\begin{align}   
	\label{eq:complexity_NC}
	\TN &\defeq \min \left\{\left \lceil \left(\frac{\sqrt{2(\Lg +\mu)/\mu}}{4} \right)\log\left[\frac{2 (\Lg +\mu) (1-\nu)}{\mu\nu}\right] +1\right \rceil, g \right\},
\end{align}
iterations of \cref{alg:MINRES}, the NPC condition \cref{eq:NPC} must be detected. Of course, if $ \gg \in \Null(\HH) $, then $ \gg $ itself is immediately declared as the zero curvature direction at the outset. If $ \nu = 1 $, then Assumptions \ref{assmpt:prop}-\labelcref{assmpt:prop:neg} implies that all non-zero $ \gg- $relevant eigenvalues of $ \HH $ are in fact negative, which in turn implies $ \TN = 1 $, i.e., the NPC direction is detected in a single iteration.
Also, under \cref{assmpt:prop}, we have from \cref{lemma:MINRES_complexity_Sol,cor:MINRES_complexity_Sol,lemma:inexactness} that after at most 
\begin{align}
	\label{eq:complexity_Sol}
	\TS &\defeq \min \left\{ \left \lceil\frac{\sqrt{\Lg/\mu}}{4} \log \left[ {4}/\left({\frac{\eta^{2}}{\left( \Lg^{2} + \eta^{2} \right)} - (1-\nu)}\right)\right] +1\right \rceil, g\right\},
\end{align}
iterations, the inexactness condition \cref{eq:inexactness} is satisfied. 

Putting this all together, we conclude that if Property  \labelcref{assmpt:prop:neg} of \cref{assmpt:prop} holds, \cref{alg:MINRES} returns a direction in at most $ \min\{\TN,\TS\} $ iterations. Indeed, in this case, it takes at most $ \TN $ iterations to detect the NPC condition \cref{eq:NPC}. It also takes at most $ \TS $ iterations to satisfy the inexactness condition \cref{eq:inexactness}. \cref{alg:MINRES} terminates if either of these occurs first. If $ \psin = 0 $, however, no negative curvature is detected across iterations, and \cref{alg:MINRES} terminates after st most $ \TS $ iteration. The returned direction in this case always satisfies \cref{eq:inexactness}.

We can now combine \cref{thm:complexity_first_alg_opt} with the complexity analysis of MINRES in \cref{sec:MINRES_complexity} to obtain optimal first-order operation complexity, i.e., the total number of gradient and Hessian-vector product evaluations for \cref{alg:NewtonMR_1st} to find a point that satisfies \cref{eq:small_g}. Note that, unlike the iteration complexity of \cref{thm:complexity_first_alg_opt}, to obtain operation complexity, we need to consider \cref{assmpt:prop}.

\begin{theorem}[Optimal Operation Complexity of \cref{alg:NewtonMR_1st}]
	\label{thm:complexity_first_Hv_opt}
	Suppose $ \dn $ is sufficiently large. Under \cref{assmpt:Lg,assmpt:prop,assmpt:T_regularity,assmpt:lipschitz,assmpt:residual_NPC}, after at most $\bigO{\epsg^{-3/2}}$ gradient and Hessian-vector product evaluations, \cref{alg:NewtonMR_1st} satisfies the approximate first-order optimality \cref{eq:small_g}. 
\end{theorem}
\begin{proof}
	Every iterations of \cref{alg:MINRES} requires one Hessian-vector product. By the proof of \cref{thm:complexity_first_alg_opt}, we see that at most 
	\begin{align*}
		\KS \triangleq \frac{\left(f(\xx_{0}) - f^{\star}\right) \epsg^{-1}}{\min \left\{ c_{1}, c_{2}\right\}},
	\end{align*}
	iterations of \cref{alg:NewtonMR_1st} uses directions with $ \dt = \text{``SOL''} $ from \cref{alg:MINRES}. Similarly, at most 
	\begin{align*}
		\KN \triangleq \frac{\left(f(\xx_{0}) - f^{\star}\right) \epsg^{-3/2}}{c_3},
	\end{align*}
	iterations employs a NPC direction. 
	the total number of gradient or Hessian-vector products for \cref{alg:NewtonMR_1st} to reach a solution satisfying \cref{eq:small_g} is at most
	\begin{align*}
		\TS \cdot \KS + \TN \cdot \KN \in \tbigO{1} \bigO{\epsg^{-1}} + \bigO{1} \bigO{\epsg^{-3/2}} \in \bigO{\epsg^{-3/2}}.
	\end{align*} 
\end{proof}

Note that, for the first-order operation complexity, Newton-MR \cref{alg:NewtonMR_1st} achieves the optimal\footnote{Recall that the optimal first-order iteration complexity is also $ \bigO{\epsg^{-3/2}} $ \cite{cartis2022evaluation}.} rate of $ \bigO{\epsg^{-3/2}} $, which is a significant improvement over the rate $ \tbigO{\epsg^{-7/4}} $ in Newton-CG variants \cite{royer2020newton,curtis2021trust}. This is  thanks to the fact that, unlike the Newton-CG variants in \cite{royer2020newton,curtis2021trust}, which employ a $\epsH$-perturbed version of the Hessian matrix,  the sub-problems of \cref{alg:NewtonMR_1st}  involve unperturbed Hessian. We end this section by noting that without \cref{assmpt:prop}, or when $ d $ is small, the total number of operations with gradients and Hessian-vector products for \cref{alg:NewtonMR_1st} to obtain \cref{eq:small_g} is at most $ \dn \bigO{\epsg^{-3/2}} $.

\subsection{Newton-MR With Second-order Complexity Guarantee}
\label{sec:NewtonMR_second}
\cref{alg:NewtonMR_1st} provides a first-order complexity guarantee \cref{eq:small_g} and hence the iterations are terminated once $ \vnorm{\ggk} \leq \epsg $. 
In non-convex optimization, however, $ \vnorm{\ggk} \leq \epsg $, in particular $ \ggk=\zero $, does not necessarily imply that the algorithm has converged to (a vicinity of) a local minimum. Hence, it if often desired to provide stronger convergence guarantees in the form of both \cref{eq:small_g} and \cref{eq:small_eig}. \cref{alg:NewtonMR_2nd} depicts a variant of Newton-MR for which we can establish second-order complexity guarantees. As it can be seen, for the most part, \cref{alg:NewtonMR_2nd} is identical to \cref{alg:NewtonMR_1st}. The main difference lies in the cases of $ \vnorm{\ggk} \le \epsg $ where \cref{alg:NewtonMR_1st} terminates, while \cref{alg:NewtonMR_2nd} aims to either find a negative curvature direction or to certify \cref{eq:small_eig}.

\begin{algorithm}
	\caption{Newton-MR With Second-order Complexity Guarantee}
	\label{alg:NewtonMR_2nd}
	\begin{algorithmic}[1]
		\vspace{1mm}
		\STATE \textbf{Input:} 
		\vspace{1mm}
		\begin{itemize}[leftmargin=*,wide=0em, noitemsep,nolistsep,label = {\bfseries -}]
			\item Initial point:    $ \xx_{0} $
			\item Approximate second-order optimality tolerance: $ 0 < \epsg \leq 1 $ and $ 0 < \epsH \leq 1 $
			\item Inexactness tolerance: $ \theta >  0$
		\end{itemize}
		\STATE $ k = 0 $
		\vspace{1mm}
		\WHILE {not terminate}
		\vspace{1mm}
		\IF{$ \| \ggk \| > \epsg $}
		\vspace{1mm}
		\STATE Call \cref{alg:MINRES} as $\left[\ddk,\dt\right] = \text{MINRES}(\HHk, \ggk, \theta \sqrt{\epsg})$
		\vspace{1mm}
		\IF{ $ \dt = \text{`SOL'} $}
		\vspace{1mm}
		\STATE Find the largest $ 0 < \alphak \leq 1$ satisfying \cref{cond:Armijo} with $ 0 < \rho < 1/2 $ using \cref{alg:line_search}
		\ELSE
		\vspace{1mm}
		\STATE Find the largest $ \alphak > 0$ satisfying \cref{cond:Armijo} with $ 0 < \rho < 1 $ using \cref{alg:line_search_forward}
		\vspace{1mm}
		\ENDIF
		\vspace{1mm}
		\ELSE
		\vspace{1mm}
		\STATE Randomly generate $ \tgg $ from the uniform distribution on the unit sphere
		\vspace{1mm}
		\STATE \label{alg:NewtonMR_2nd_Lanczos} Call \cref{alg:MINRES} as $\displaystyle \left[\ddk,\dt\right] = \text{MINRES}\left(\HHk+\frac{\varepsilon_H}{2} \eye, \tgg, 0\right)$
		\vspace{1mm}
		\IF{ $ \dt = \text{`NPC'} $}
		\vspace{1mm}
		\STATE \label{alg:NewtonMR_2nd_Lanczos_sign} $ \ddk = - \sign(\dotprod{\ggk, \ddk}) \ddk / \| \ddk \| $ \quad (NB: ``$\sign(0)$'' can be arbitrarily chosen as ``$\pm 1$'') 
		\vspace{1mm}
		\STATE Find the largest $ \alphak > 0$ satisfying \cref{cond:Armijo_2nd} with $ 0 < \rho < 1 $ using \cref{alg:line_search_forward}
		\vspace{1mm}
		\ELSE
		\vspace{1mm}
		\STATE Terminate
		\vspace{1mm}
		\ENDIF
		\vspace{1mm}
		\ENDIF
		\vspace{1mm}
		\STATE $ \xx_{k+1} = \xxk + \alphak \ddk $
		\vspace{1mm}
		\STATE $ k = k + 1 $
		\vspace{1mm}
		\ENDWHILE
		\vspace{1mm}
		\STATE \textbf{Output:} $ \xxk $ satisfying second-order optimality \cref{eq:termination_second_order}
	\end{algorithmic}
\end{algorithm}

When $ \ggk $ is too small, calling \cref{alg:MINRES} with such $ \ggk $ may cease to provide beneficial utility in obtaining a descent direction. In fact, in the extreme case where $ \ggk = \zero $, \cref{alg:MINRES} returns $ \ddk = \zero $ and the optimization algorithm stagnates. 
As a remedy, and in the context of Newton-CG, \cite{royer2020newton} replaces $ \ggk $ with some appropriately chosen random vector $ \tgg $, namely drawn from the uniform distribution over the unit sphere. In addition, in lieu of the true Hessian $ \HHk $, some perturbation as $ \tHHk = \HHk + 0.5 \epsH \eye $ is used within the standard CG method. If $ \lambda_{\min}(\tHH) \leq -\epsH/2 $, then (with high probability) a negative curvature direction for $ \tHHk $ is encountered in certain numbers of iterations. Otherwise, it is concluded that $ \tHHk \succeq -\epsH/2 $ (with high probability), which in turn implies \cref{eq:small_eig}. In our Newton-MR setting, we adopt a similar approach.

\subsubsection{Optimal Iteration Complexity}    
We can use the complexity results from \cref{sec:NewtonMR_first} for all the iterations where $ \vnorm{\ggk} \geq \epsg $. The only missing piece is the implication of using the direction obtained from MINRES in the case where $ \vnorm{\ggk} \leq \epsg $.             
For small $ \ggk $, the line-search condition \cref{cond:Armijo} no longer provides a meaningful criterion for choosing the step-size. Hence, we replace the line-search \cref{cond:Armijo} with 
\begin{align}
	\label{cond:Armijo_2nd}
	f(\xxk + \alphak \ddk) \leq f(\xxk) + \hf \rho \alpha_{k}^2 \dotprod{\ddk, \HHk \ddk},
\end{align}
for some $ 0 < \rho < 1 $. Note that we can evaluate the term ``$ \dotprod{\ddk, \HHk \ddk} $'' in \cref{cond:Armijo_2nd} without any additional Hessian-vector product. Indeed, from \hyperref[alg:NewtonMR_2nd_Lanczos_sign]{Step 15} of \cref{alg:NewtonMR_2nd}, we have $ \ddk = - \sign{\dotprod{\ggk,\ddk}} \ddk / \| \ddk \| $, which using \cref{eq:NPC_curve}, gives
\begin{align*}
	\dotprod{\ddk, \HHk \ddk} &= \dotprod{\ddk, \tHHk \ddk} - \frac{\epsH}{2} \vnorm{\ddk}^2 \\
	&= \dotprod{\frac{\rrktp}{\vnorm{\rrktp}}, \frac{\tHHk \rrktp}{\vnorm{\rrktp}}} - \frac{\epsH}{2} = -c_{t-1} \gamma_{t} - \frac{\epsH}{2},
\end{align*}
where the residual $ \rrktp = -\tgg - \tHHk \ssktp$ and the quantity $c_{t-1} \gamma_{t}$ are those from calling \cref{alg:MINRES} as $\text{MINRES}\left(\HHk+0.5\epsH \eye, \tgg, 0\right)$.

\cref{lemma:NC_distance_pert} gives a worst-case decrease in $ f $ using the direction from \hyperref[alg:NewtonMR_2nd_Lanczos_sign]{Step 15} of \cref{alg:NewtonMR_2nd}.
\begin{lemma}
	\label{lemma:NC_distance_pert}  
	Suppose $ \lambda_{\min}(\HHk) \leq -\epsH $ and  \hyperref[alg:NewtonMR_2nd_Lanczos]{Step 13} of \cref{alg:NewtonMR_2nd} successfully returns a NPC direction. Under \cref{assmpt:lipschitz}, in \cref{alg:NewtonMR_2nd}, we have
	\begin{align*}
		f(\xxkn) - f(\xxk) \leq - c_4 \epsH^3,
	\end{align*}
	where
	\begin{align*}
		c_4 \triangleq - \frac{9 \rho  \left(1- \rho\right)^{2}}{16 \LH^{2}},
	\end{align*}
	$ \LH $, and $ \rho $ are as in \cref{assmpt:lipschitz,cond:Armijo_2nd}. 
\end{lemma}
\begin{proof} 
	By assumption, \cref{alg:MINRES} returns a NPC direction $ \rrktp = -\tgg - \tHHk \ssktp$ for some $ 1 \leq t \leq d $. Since $ \ddk = - \sign{\dotprod{\rrktp, \ggk}} \rrktp / \vnorm{\rrktp} $, we have 
	\begin{align*}
		\dotprod{\ddk, \ggk} &\leq 0, \quad \text{and} \quad \dotprod{\ddk, \HHk \ddk} \leq - \frac{\epsH}{2} \vnorm{\ddk}^2 = - \frac{\epsH}{2}.
	\end{align*}    
	Similar to the proof of \cref{lemma:NC_distance_opt}, we have 
	\begin{align*}
		f(\xxk + \alpha \ddk) - f(\xxk) &\leq \alpha \dotprod{\ddk, \ggk} + \frac{\alpha^2}{2} \dotprod{\ddk, \HHk \ddk}  + \frac{\LH}{6} \alpha^3 \vnorm{\ddk}^3 \\
		&\leq \frac{\alpha^{2}}{2} \dotprod{\ddk, \HHk \ddk} + \frac{\LH}{6} \alpha^3.
	\end{align*}
	and it follows that for the largest $ \alphak $ satisfying the line-search condition \cref{cond:Armijo_2nd}, we must have 
	\begin{align*}
		\alphak \geq \frac{3 \left(1- \rho\right) \epsH}{2 \LH}.
	\end{align*}    
	So, it follows that
	\begin{align*}
		f(\xxk + \alphak \ddk) - f(\xxk) \leq \hf \rho \alpha_{k}^2 \dotprod{\ddk, \HHk \ddk} \leq - \frac{1}{4} \rho \alpha_{k}^2 \epsH < - \frac{9 \rho  \left(1- \rho\right)^{2} \epsH^{3}}{16 \LH^{2}}.
	\end{align*}
\end{proof}

\cref{lemma:NC_distance_pert}  assumes that \hyperref[alg:NewtonMR_2nd_Lanczos]{Step 13} of \cref{alg:NewtonMR_2nd} successfully returns a NPC direction, which can be guaranteed to hold with probability one over the draws of $ \tgg $. 
Indeed, by randomly drawing $ \tgg $ from the uniform distribution on the unit sphere, we ensure that, with probability one, the grade of $ \tgg $ with respect to $ \tHHk $ is $ |\Theta(\tHHk)| $, i.e., $ \Psi(\tHH, \tgg) = \Theta(\tHH) $. Hence, \cref{thm:H_PSD} implies that, with probability one, $ \tHHk \not \succeq \zero $ if and only if \cref{alg:MINRES} detects a NPC direction for $ \tHHk $. This, in turn, amounts to either obtaining a direction of negative curvature for $ \HHk $ or else guaranteeing that $ \tHHk \succ \zero $, which is in fact a certificate for \cref{eq:small_eig}. As a result, if $ \lambda_{\min}(\HHk) \leq -\epsH $, \hyperref[alg:NewtonMR_2nd_Lanczos]{Step 13} of \cref{alg:NewtonMR_2nd} successfully returns a NPC direction with probability one.

We are now in a position to give the iteration and operation complexities of \cref{alg:NewtonMR_2nd} to achieve approximate second-order optimality condition \cref{eq:termination_second_order}.

\begin{theorem}[Optimal Iteration Complexity of \cref{alg:NewtonMR_2nd}]
	\label{thm:complexity_second_alg}
	Define 
	\begin{align*}
		K \defeq \left \lceil\frac{2 \left(f(\xx_{0}) - fs\right)}{\min \left\{ c_{1}, c_{2}, c_{3}, c_4\right\}} \max\left\{\epsg^{-3/2}, \epsH^{-3}\right\} + 1 \right \rceil,
	\end{align*}
	where $ c_1 $, $ c_2 $, $ c_3 $ and $ c_4 $ are as in \cref{lemma:SOL_distance_opt,lemma:NC_distance_opt,lemma:NC_distance_pert}, and $ f^{\star} \defeq \min f(\xx) $. Under \cref{assmpt:Lg,assmpt:lipschitz,assmpt:T_regularity,assmpt:residual_NPC}, after at most $ K $ iterations of \cref{alg:NewtonMR_2nd}, the approximate second-order optimality \cref{eq:termination_second_order} is satisfied with probability one.
\end{theorem}
\begin{proof}
	We provide the analysis, conditioned on the event that \hyperref[alg:NewtonMR_2nd_Lanczos]{Step 13} of \cref{alg:NewtonMR_2nd} is successful. Incorporating failure probabilities is straightforward.
	Suppose $ \| \ggk \| \leq \epsg $ and $ \lambda_{\min}(\HHk) \geq -\epsH $, for the first time, at $ k = K + 1 $. Then, we must have $ \vnorm{\gg_{k}} > \epsg$ or $\lambda_{\min}(\HHk) \leq -\epsH $ for any $k = 0, \ldots, K $.  
	Any iteration $k = 0, \ldots, K $ must belong to one of these three sets
	\begin{align*}
		\mathcal{K}_1 &= \{0 \leq k \leq K \mid \min \left\{\vnorm{\ggk}, \vnorm{\ggkk}\right\} > \epsg \}, \\
		\mathcal{K}_2 &= \{0 \leq k \leq K \mid \vnorm{\ggk} \leq \epsg \} = \{0 \leq k \leq K \mid \vnorm{\ggk} \leq \epsg \; \text{and} \; \lambda_{\min}(\HHk) \leq -\epsH \} \\
		\mathcal{K}_3 &= \{0 \leq k \leq K \mid  \vnorm{\ggk} > \epsg \geq \vnorm{\ggkk}\}.
	\end{align*}
	So, it follows that
	\begin{align*}
		f(\xx_{0}) - f^{\star} &\geq f(\xx_{0}) - f(\xx_{K+1}) = \sum_{k=0}^{K} f(\xx_{k}) - f(\xx_{k+1}) \\
		&\geq \sum_{k \in \mathcal{K}_1} f(\xx_{k}) - f(\xx_{k+1}) \geq \abs{\mathcal{K}_1} \min \left\{ c_{1}, c_{2}, c_{3}\right\} \epsg^{3/2},
	\end{align*}
	where the last inequality follows from \cref{lemma:SOL_distance_opt,lemma:NC_distance_opt}.
	Similarly, from \cref{lemma:NC_distance_pert}, we have
	\begin{align*}
		f(\xx_{0}) - f^{\star} &\geq \sum_{k=0}^{K} f(\xx_{k}) - f(\xx_{k+1}) \geq \sum_{k \in \mathcal{K}_2} f(\xx_{k}) - f(\xx_{k+1}) \geq \abs{\mathcal{K}_2} c_{4}  \epsH^{3}.
	\end{align*}
	Lastly, if $ k \in \mathcal{K}_3 $, then $ k+1 \in \mathcal{K}_2 $, and hence $ \abs{\mathcal{K}_3} \leq \abs{\mathcal{K}_2} + 1 $.
	
	Putting this all together, we have
	\begin{align*}
		\abs{\mathcal{K}_1} + \abs{\mathcal{K}_2} + \abs{\mathcal{K}_3} \leq \frac{\left(f(\xx_{0}) - f^{\star}\right)}{\min \left\{ c_{1}, c_{2}, c_{3}\right\}} \epsg^{-3/2} + \frac{2 \left(f(\xx_{0}) - f^{\star}\right)}{c_{4}} \epsH^{-3} + 1.
	\end{align*}
\end{proof}

\subsubsection{Optimal Operation Complexity}
We now discuss the complexity of \hyperref[alg:NewtonMR_2nd_Lanczos]{Step 13} of \cref{alg:NewtonMR_2nd}. 
Since \cref{alg:MINRES} builds on the exact same Krylov sub-space as the randomized Lanczos method, in the absence of any additional structural assumption on $ f $, we can use a similar complexity analysis as that in \cite{kuczynski1992estimating,royer2020newton}. Indeed, from \cite[Theorem 4.2]{kuczynski1992estimating} or \cite[Lemma 9]{royer2020newton} (adopted to our setting), with probability at least $ 1 - \delta $, after at most
\begin{align*}
	t \geq \min \left\{\left \lceil \left(\frac{1}{4\sqrt{\epsilon}} \right)\log\left[\frac{2.75 d}{\delta^{2}}\right] +1\right \rceil, d \right\},
\end{align*}
iterations, we have
\begin{align*}
	\lambda_{\min}(\tTTt) - \lambda_{\min}(\tHH) \leq \epsilon \left(\lambda_{\max}(\tHH) - \lambda_{\min}(\tHH)\right),
\end{align*}
where $ 0 < \delta < 1 $, $ \epsilon > 0 $, and $ \tTTt $ is the tridiagonal matrix as in \cref{eq:tridiagonal_T} obtained from $ \tHH = \HH + 0.5 \epsH \eye$ and $ \tgg $. Letting $ \epsilon = \epsH / (4\Lg) $, we get
\begin{align*}
	\lambda_{\min}(\tTTt) - \lambda_{\min}(\tHH) &\leq \frac{\epsH}{4 \Lg} \left(\lambda_{\max}(\tHH) - \lambda_{\min}(\tHH)\right) \\
	&= \frac{\epsH}{4 \Lg} \left(\lambda_{\max}(\HH) + \epsH + (-\lambda_{\min}(\HH) - \epsH)\right) \leq \frac{\epsH}{2},
\end{align*}
which implies $ \lambda_{\min}(\HH) = \lambda_{\min}(\tHH) - \epsH / 2 \geq \lambda_{\min}(\tTTt) - \epsH $. As a result, if after 
\begin{align}
	\label{eq:complexity_NC_Lanczos}   
	\TNL &\defeq \min \left\{\left \lceil \left(\frac{\sqrt{\Lg/\epsH}}{2} \right)\log\left[\frac{2.75 d}{\delta^{2}}\right] +1\right \rceil, d \right\},
\end{align}
iterations, no NPC direction is detected in \hyperref[alg:NewtonMR_2nd_Lanczos]{Step 13} of \cref{alg:NewtonMR_2nd}, i.e., $ \tTTt \succ \zero $ for all $t \leq \TNL$ (cf.\ \cref{lemma:MINRES_NPC_detector}), then we must have $ \lambda_{\min}(\HH) > - \epsH $, with probability $ 1-\delta $. In other words, if $ \lambda_{\min}(\HHk) \leq -\epsH $, \hyperref[alg:NewtonMR_2nd_Lanczos]{Step 13} of \cref{alg:NewtonMR_2nd} can find a direction of negative curvature for $ \tHHk $, with probability $ 1-\delta $, in at most $ \TNL $ iterations.


\hyperref[alg:NewtonMR_2nd_Lanczos]{Step 13} of \cref{alg:NewtonMR_2nd} is a probabilistic procedure. Hence, in order to guarantee success over the life of \cref{alg:NewtonMR_2nd}, we need to ensure a small failure probability for this step across all iterations. In particular, in order to get an overall and accumulative success probability of $ 1 - \delta $ for the entire $ K  $ iterations, the per-iteration failure probability is set as $(1- \sqrt[K]{(1-\delta)} )\in \mathcal{O}(\delta/K)$. This failure probability appears only in the ``log factor'' in \cref{eq:complexity_NC_Lanczos}, and so it is not a dominating cost. Hence, requiring that all $ K $ iterations are successful for a large $ K $, only necessitates a small (logarithmic) increase in worst-case iteration count of \cref{alg:MINRES} as part of \hyperref[alg:NewtonMR_2nd_Lanczos]{Step 13} of \cref{alg:NewtonMR_2nd}. For example, for $ K \in \bigO{\max\left\{\epsg^{-3/2}, \epsH^{-3}\right\}}$, we can set the per-iteration failure probability to $ \delta \cdot \min \left\{\epsg^{3/2}, \epsH^{3}\right\} $, and ensure that when Algorithm \ref{alg:NewtonMR_2nd} terminates, \hyperref[alg:NewtonMR_2nd_Lanczos]{Step 13} have been reliably successful, with an accumulative probability of $ 1-\delta $. 

\begin{theorem}[Operation Complexity of \cref{alg:NewtonMR_2nd}]
	\label{thm:complexity_second_Hv}
	Suppose $ \dn $ is sufficiently large. Under \cref{assmpt:Lg,assmpt:prop,assmpt:T_regularity,assmpt:lipschitz,assmpt:residual_NPC}, after at most $\max \left\{\tbigO{\epsg^{-3/2}}, \bigO{\epsH^{-7/2}} \right\}$ gradient and Hessian-vector product evaluations, \cref{alg:NewtonMR_2nd} satisfies the approximate second-order optimality \cref{eq:termination_second_order} with probability $ 1-\delta $. 
\end{theorem}
\begin{proof}
	Again, we provide a deterministic analysis, which can trivially be modified to incorporate failure probabilities.
	Every iterations of \cref{alg:MINRES} requires one Hessian-vector product. Hence, similar to the proof of \cref{thm:complexity_first_Hv_opt}, under \cref{assmpt:prop}, using \cref{eq:complexity_Sol,eq:complexity_NC_Lanczos}, the total number of Hessian-vector products is no more than
	\begin{align*}
		\TS \left( \abs{\mathcal{K}_{1}} + \abs{\mathcal{K}_{3}} \right) + \TNL \abs{\mathcal{K}_{2}} 
		&\in \tbigO{1} \left( \bigO{\epsg^{-3/2}} + \bigO{\epsH^{-3}} \right) + \bigO{\epsH^{-1/2}} \bigO{\epsH^{-3}}.
	\end{align*}
	
\end{proof}

Clearly, if $ \epsH^{2} = \epsg = \varepsilon $, the operation complexity of \cref{alg:NewtonMR_2nd}} to  achieve \cref{eq:termination_second_order} is $ \bigO{\epsH^{-7/2}} $, which matches those of alternative algorithms with similar state-of-the-art guarantees. Finally, without \cref{assmpt:prop}, or in small dimensional problems, the operation complexity of \cref{alg:NewtonMR_2nd} to obtain \cref{eq:termination_second_order} is at most $ \dn \max \left\{\bigO{\epsg^{-3/2}}, \bigO{\epsH^{-3}}\right\} $.

\paragraph{Benign Saddle Regions.} The bound given in \cref{eq:complexity_NC_Lanczos} is obtained directly from the complexity result of the randomized Lanczos method \cite{kuczynski1992estimating} without any additional structural assumption on $ f $ other than \cref{assmpt:Lg}. It turns out that as long as the regions near the saddle points of $ f $ exhibit sufficiently large negative curvature, one can obtain an improved operation complexity as compared with \cref{thm:complexity_second_Hv}. 
\begin{assumption}[Benign Saddle Property]
	\label{assmp:saddle}
	The function $ f $ has the $ (\iota, \mu, \varsigma) $-benign saddle property, i.e., there exists $ \iota > 0$, $\mu > 0$,  and $0 \leq \varsigma < \mu$, such that for any point $ \xx \in \real^{d} $, at least one of the following holds:
	\begin{enumerate}[label = {\bfseries (\roman*)}]
		\item \label{eq:saddle:grad} $ \| \nabla f(\xx) \| \geq \iota $, 
		\item \label{eq:saddle:strict} $ \lambda_{\min}(\nabla^2 f(\xx)) \leq - \mu $, or
		\item \label{eq:saddle:convex} $ \lambda_{\min}(\nabla^2 f(\xx)) \geq - \varsigma $. 
	\end{enumerate}
\end{assumption}
Clearly, if we allow $\varsigma = \mu$, then \cref{assmp:saddle} would be trivially satisfied by all twice differentiable functions. Non-triviality of \cref{assmp:saddle} lies in the discrepancy between $\varsigma$ and $\mu$. Coupled with Hessian continuity from \cref{assmpt:lipschitz}, the optimization landscape of functions satisfying \cref{assmp:saddle} is in essence structured such that going from the vicinity of saddle points to near local minima entails navigating steep regions with large enough gradients. 

It turns out that   
\cref{assmp:saddle} is in fact a relaxation of the strict saddle property, which has become a standard assumption in analyzing non-convex optimization algorithms that can escape saddle points, e.g., \cite{ge2015escaping,sun2015nonconvex,reich2010nonlinear,lee2016gradient,panageas2017gradient,panageas2019first,lee2019first,achour2021global}. Recall that a function is said to satisfy the $ (\iota, \mu, \vartheta, \Delta) $-strict saddle property, if for any $ \xx  \in \real^{d}$, we have either $ \| f(\xx) \| \geq \iota > 0 $, $ \lambda_{\min}(\nabla^2 f(\xx)) \leq - \mu < 0 $, or there is a local minimum $ \xxs $ such that $ \| \xx - \xxs \| \leq \Delta $ and $ \lambda_{\min}(\nabla^2 f(\xx)) \geq \vartheta > 0 $ in the neighborhood $ \| \xx - \xxs \| \leq 2 \Delta $. It has been shown that many interesting machine learning problems satisfy the strict saddle property, e.g., online tensor decomposition \cite{sun2015nonconvex}, dictionary recovery problems \cite{sun2016complete}, the (generalized) phase retrieval problems \cite{sun2018geometric}, and the phase synchronization and community detection problems \cite{boumal2016nonconvex,bandeira2016low}. It is easy to see that $ (\iota, \mu, \vartheta, \Delta) $-strict saddle property implies $ (\iota, \mu, 0) $-benign saddle property. In this light, there are many more functions that enjoy $ (\iota, \mu, 0) $-benign saddle property than those with a  $ (\iota, \mu, \vartheta, \Delta) $-strict saddle characteristic. 

Leveraging the benign saddle assumption, we will apply \cref{cor:MINRES_complexity_NC} to obtain an alternative bound to $ \TNL $ in \cref{eq:complexity_NC_Lanczos}. For this, we need to find an estimate on the projection of $\tgg$ on a given eigenspace of $\tHH$, i.e., $\nu_{j} $ as in \cref{eq:nuj}. Fortunately, the particular choice for generating $ \tgg $ allows us to do just that. Indeed, suppose $ d \geq 3 $ and let $ \tgg $ be randomly generated from a uniform distribution on the unit sphere, i.e., $ \tgg = [\tg_1, \tg_2, \ldots, \tgd]^{\T} / \| \tgg \| $, where $ \| \tgg \|^{2} = {\sum_{i=1}^{d} \tg_{i}^2} $ and $ \tg_{i} \overset{iid}{\sim} \mathcal{N}(0,1) $. Consider $ \tnu = \left(\dotprod{\tgg,\uu}\right)^{2} $ where $ \uu $ is any unit eigenvector of $ \tHH $. By the spherical symmetry, the distribution of this dot product is the same that of $\left(\dotprod{\tgg,\ee_{1}}\right)^{2}$, where $ \ee_{1} $ is the first column of the identity matrix, so we consider $ \tnu = \left(\dotprod{\ee_{1}, \tgg}\right)^{2} = \tg^{2}_{1} / \| \tgg \|^{2} $. Recall that $ \tnu \sim \mathcal{B}(1/2, {(d-1)}/{2}) $, where $ \mathcal{B} $ denotes the beta distribution. 
Hence, for any $ 0 < \nu < 1 $, we obtain
\begin{align*}
	\delta  \defeq \Pr(\tnu \leq \nu) =  \int_{0}^{\nu} c(d) t^{-1/2}(1 - t)^{(d-3)/2} d t \leq 2 c(d) \sqrt{\nu},
\end{align*}
where
\begin{align}
	\label{eq:gamma_fun}
	c(d) \triangleq \frac{\Gamma({d}/{2})}{\sqrt{\pi}\Gamma({(d-1)}/{2})},
\end{align}
and $ \Gamma $ is the Gamma function. So, it follows that, with probability $ 1-\delta $, we have
$ \tnu \geq \delta^{2}/\left(4 c^{2}(d)\right) $.

Putting this all together, consider any function with $ (\iota, \mu, \varsigma) $-benign saddle property. Letting $ 0 < \epsg \leq \iota $ and $ \varsigma < \epsH/2 <  \mu $, we can apply \cref{cor:MINRES_complexity_NC} with $ \tHH $ and $ \tgg $, to guarantee that, with probability $ 1-\delta $, in at most 
\begin{align}   
	\label{eq:complexity_NC_pert}
	\TP &\defeq \min \left\{\max\left\{ \left \lceil \left(\frac{1}{4}\sqrt{\frac{\Lg + \mu}{\mu - 0.5 \epsH}} \right)\log\left[\frac{4 (\Lg + \mu)}{(\mu - 0.5 \epsH)} \left(\frac{4 c^{2}(d)}{\delta^{2}} - 1\right)\right] +1\right \rceil, 3 \right\}, d \right\},
\end{align}
iterations, a NPC direction for $ \tHH $ is detected, where $ c(d) $ is as in \cref{eq:gamma_fun}. 
In other words, if a NPC direction is never detected in $ \TP $ iterations, then with probability $ 1 - \delta $, we have $ \HH \succeq -\varsigma$.
%
Finally, we obtain the following improved operation complexity for functions that enjoy benign saddle property, whose proof is almost identical to \cref{thm:complexity_second_Hv} and hence it omitted.
\begin{corollary}[Operation Complexity of \cref{alg:NewtonMR_2nd} for Function with Benign Saddle Regions]
	\label{cor:complexity_second_Hv_saddle}
	Suppose $ \dn $ is sufficiently large, and \cref{assmpt:Lg,assmpt:prop,assmpt:T_regularity,assmpt:lipschitz,assmpt:residual_NPC} hold. Further, suppose the function $ f $ also satisfied the $ (\mu, \iota, \varsigma) $-benign saddle property as in \cref{assmp:saddle}. Let $ 0 < \epsg \leq \iota $ and $ \varsigma < \epsH/2 < \mu $. After at most $\max \left\{\tbigO{\epsg^{-3/2}}, \tbigO{\epsH^{-3}}\right\}$ gradient and Hessian-vector product evaluations, \cref{alg:NewtonMR_2nd} finds a point $ \xx $ such that $ \vnorm{\gg(\xx)} \leq \epsg $ and $ \lambda_{\min}(\HH(\xx)) \geq - \varsigma $, with probability $ 1-\delta $.
\end{corollary}

\section{Numerical Experiments}
\label{sec:exp}

In this section, we will evaluate the performance of \cref{alg:NewtonMR_1st} on several examples\footnote{In our implementations, we only aim to reach an approximate first-order optimal point, and leave a thorough numerical evaluation of \cref{alg:NewtonMR_2nd} to a follow up empirical work.}\!\!, namely, non-linear least squares (\cref{sec:nls}), deep auto-encoders (\cref{sec:ae}), and a series of problems from the \texttt{CUTEst} test collection (\cref{sec:cute}). We compare \cref{alg:NewtonMR_1st} with the following alternative Newton-type methods. 
\begin{itemize}[label = -]
	\item \texttt{L-BFGS} \cite[Algorithm 7.5]{nocedal2006numerical}. A limit memory quasi-Newton method with strong Wolfe line-search \cite[Algorithm 3.5]{nocedal2006numerical}. 
	\item \texttt{Newton-CR} \cite[Algorithm 3]{dahito2019conjugate}. A line search Newton-CR method  with strong Wolfe line-search \cite[Algorithm 3.5]{nocedal2006numerical}.
	\item \texttt{Newton-CG-LS} \cite[Algorithm 3]{royer2020newton}. A line search Newton-CG method with small Hessian perturbations. 
	\item \texttt{Newton-CG-LS-FW}. This is identical to \texttt{Newton-CG-LS}, except that, when negative curvature directions are encountered, we incorporate  forward/backward tracking line-search, \cref{alg:line_search_forward}, within the framework of \cite[Algorithm 3]{royer2020newton}. This is mainly to create as much of a level playing field as possible among various methods, in particular for Newton-CG variants. We also note that this change is in fact consistent with the theoretical analysis of \cite{royer2020newton} and does not raise any theoretical concerns.
	\item \texttt{Newton-CG-TR-Steihaug} \cite[Algorithm 4.1]{nocedal2006numerical}. A trust-region method with CG-Steihaug sub-problem solver. 
	\item \texttt{Newton-CG-TR} \cite[Algorithm 4.1]{curtis2021trust}. A trust-region method based on Capped-CG algorithm and small Hessian perturbations.
\end{itemize}

We set the maximum number of iterations and the respective inexactness tolerance for all sub-problems solvers, i.e., CG/CR/MINRES, to be, respectively, $1,000$ and $ 0.1 $. For trust-region methods, the radius is enlarged by a factor of $ 3 $ when the reduction in the objective function is larger than $ 20\% $ of what is predicted by the underlying quadratic model. Otherwise, the radius is cut in half. The initial and the  maximum trust region radii are, respectively, chosen to be $ 1 $, and $ 10^{10} $.
For line-search methods, the line-search algorithm is initialized with the unit step-size, the Armijo parameter is set to $ 10^{-4} $, and the parameter for the Wolfe's curvature condition \cite{nocedal2006numerical} is $ 0.1 $. Also, the maximum iterations of the line-search is set to $ 1,000 $. We terminate the optimization algorithms if the norm of gradient falls below $ \epsg = 10^{-10} $, which signifies a successful termination. The algorithms fail to converge if they are terminated prematurely, i.e., if the total number of oracle calls exceeds $10^{5}$, or if the step-size/trust region radius shrink to less than $ 10^{-18} $.

For experiments of \cref{sec:ae,sec:nls}, we demonstrate the performance of the algorithms as measured by the objective value and gradient norm in light of the total number of calls to the function oracle (or equivalent operations) \cite{roosta2018newton} as well as the ``wall-clock'' time. In \cref{sec:cute}, however, to empirically evaluate various algorithms on CUTEst problem sets, we plot the performance profiles \cite{dolan2002benchmarking,gould2016note} for objective value and gradient norm metrics.

\begin{figure}[htb]
	\centering
	\includegraphics[width=.95\textwidth]{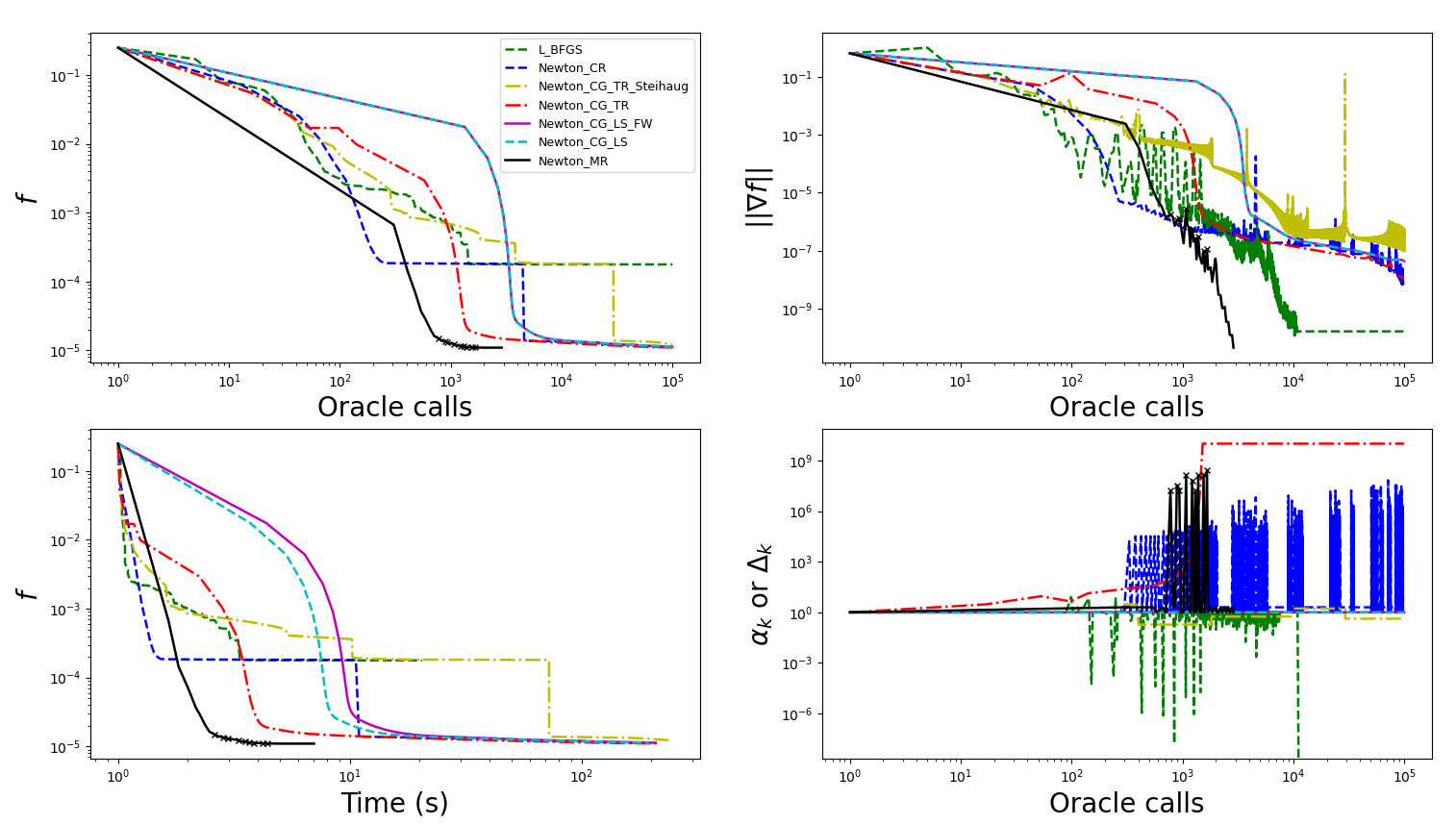}
	\caption{The performance of various algorithms on the model problem of \cref{sec:nls} using the \texttt{Gisette} dataset \cite{guyon2004result} ($ n=6,000 $ and $ d = 5,000 $). The special marks on the plots signify the iterations where a negative/non-positive curvature direction is detected and subsequently used within the respective algorithms. An ``oracle call'' refers to an operation that is equivalent, in terms of complexity, to a single function evaluation. Time is measured in seconds. \texttt{Newton-MR} achieves a solution faster than all other methods. Also, note that, for this problem, \texttt{Newton-MR} is the only method that leverages the NPC directions when they arise. \label{fig:gisette}}
\end{figure}

\begin{figure}[htb]
	\centering
	\includegraphics[width=.95\textwidth]{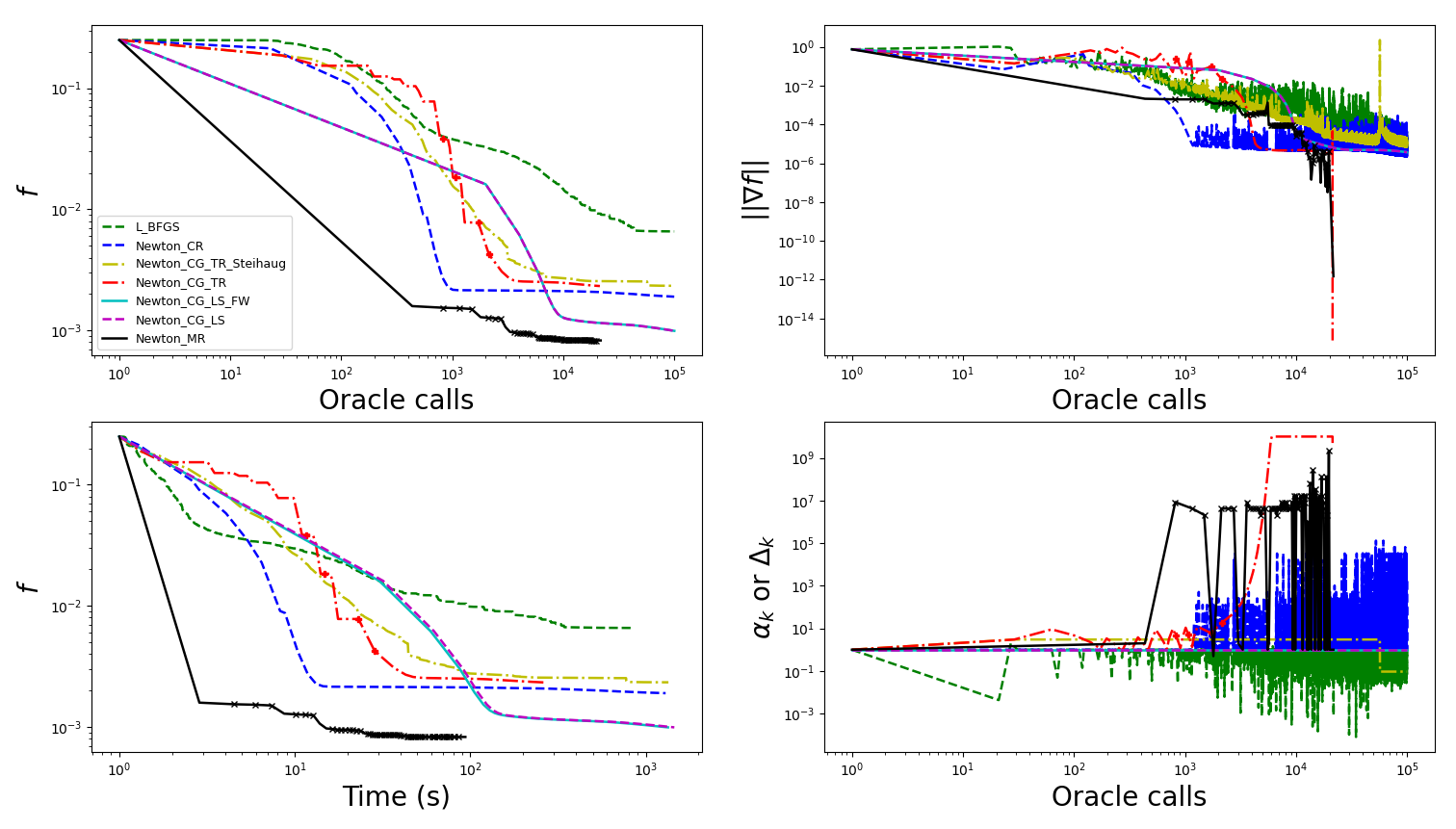}
	\caption{The performance of various algorithms on the model problem of \cref{sec:nls} using the \texttt{STL10} dataset \cite{coates2011analysis} ($ n=5,000 $ and $ d = 27,648$). The special marks on the plots signify the iterations where a negative/non-positive curvature direction is detected and subsequently used within the respective algorithms. An ``oracle call'' refers to an operation that is equivalent, in terms of complexity, to a single function evaluation. Time is measured in seconds. One can clearly see the superior performance of \texttt{Newton-MR} in achieving a better quality solution in less time/computational efforts. \label{fig:stl10}}
\end{figure}

\subsection{Non-linear Least-square Problem}
\label{sec:nls}
We first consider a regularized non-linear least squares problem,
\begin{align*}
	f(\xx) = \frac{1}{n}\sum_{i=1}^{n} \left( b_i - \frac{1}{\displaystyle 1+e^{-\dotprod{\aa_i, \xx}}} \right)^2 + \lambda \psi(\xx),
\end{align*}
where $\{(\aa_i, b_i)\}_{i=1}^n \subset \bbR^d \times \{0, 1\}$, and $ \psi(\xx) = \sum_{i=1}^{d} x_i^2/(1+x_i^2) $ is a non-convex regularization. We run the experiments on two datasets, namely \texttt{Gisette} and \texttt{STL10}\footnote{The \texttt{STL10} dataset contains colored images in ten classes. We relabel the even classes as ``$ 0 $'' and the odd one as ``$ 1 $''.}, and set $ \lambda $ to $ 10^{-7} $ and $ 10^{-6} $, respectively. All algorithms are initialized from the same instance drawn randomly from the standard normal distribution.

In \cref{fig:gisette}, we see that \texttt{Newton-MR} terminates successfully much faster than all other algorithms and it does so by achieving the lowest objective value. This might be related to the observation that, among all methods, \texttt{Newton-MR} is the only one method that leverages the NPC directions when they arise (emphasized by special marks on the plot). When NPC directions are not encountered, \texttt{Newton-CG-LS-FW} and \texttt{Newton-CG-LS} perform almost identically. \texttt{Newton-CR}, \texttt{L-BFGS}, and \texttt{Newton-CG-TR-Steihaug} stagnate around a saddle point for a long time. Eventually, \texttt{Newton-CR} and \texttt{Newton-CG-TR-Steihaug} escape the saddle region, while \texttt{L-BFGS} fails to make sufficient progress. The superior performance of \texttt{Newton-MR} in  \cref{fig:stl10} is even more pronounced. In fact, no other algorithm could obtain a similar quality solution in the same amount of time and computational efforts. Note that using a NPC direction, the step-size returned from forward/backward tracking line-search strategy \cref{alg:line_search_forward} within \texttt{Newton-MR} can reach $ 10^{7} $ to provide sufficient decrease in $ f $. Similar performance can also be seen in the following experiment.

\begin{figure}[htb]
	\centering
	\includegraphics[width=.95\textwidth]{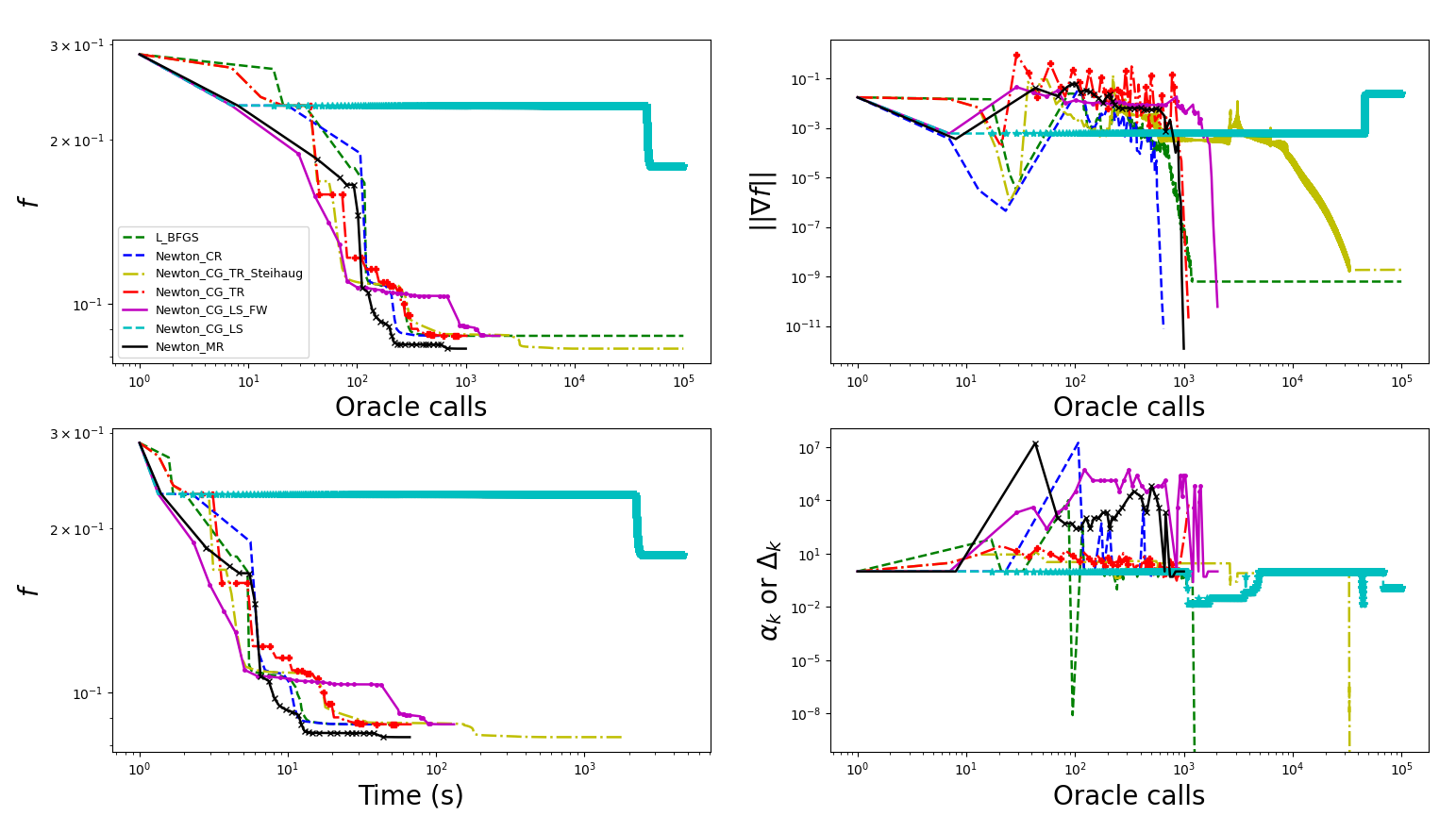}
	\caption{Deep auto-encoder with \texttt{CIFAR10} dataset \cite{krizhevsky2009learning}. The special marks on the plots signify the iterations where a negative/non-positive curvature direction is detected and subsequently used within the respective algorithms. An ``oracle call'' refers to an operation that is equivalent, in terms of complexity, to a single function evaluation. Time is measured in seconds. \texttt{Newton-MR} clearly outperforms all other methods. Notable is the poor performance of \texttt{Newton-CG-LS}, which is hindered by encountering NPC directions too often and yet not employing a forward tracking strategy.   \label{fig:CIFAR10}}
\end{figure}

\begin{figure}[htb]
	\centering
	\includegraphics[width=.95\textwidth]{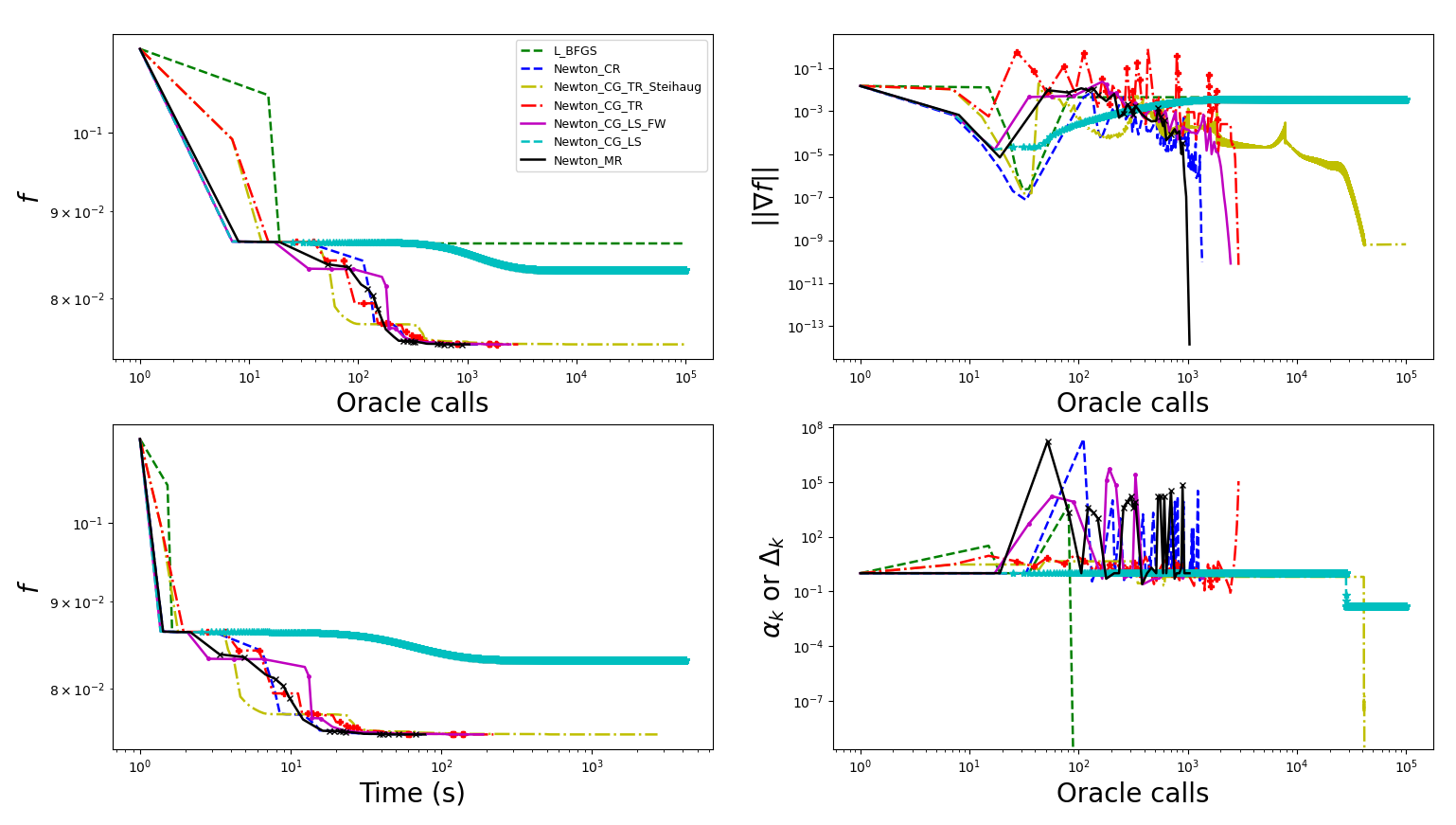}
	\caption{Deep auto-encoder with  \texttt{MNIST} dataset \cite{lecun1998gradient}.  The special marks on the plots signify the iterations where a negative/non-positive curvature direction is detected and subsequently used within the respective algorithms. An ``oracle call'' refers to an operation that is equivalent, in terms of complexity, to a single function evaluation. Time is measured in seconds.  The \texttt{MNIST} dataset gives rise to an arguably much simpler problem and hence we see that, with the exception of  \texttt{Newton-CG-LS}, all methods perform similarly. \label{fig:MNIST}}
\end{figure}

\subsection{Auto-encoder}
\label{sec:ae}
We now consider a non-convex deep auto-encoder problem as
\begin{align*}
	f(\xx_{\mathcal{E}},\xx_{\mathcal{D}}) = \frac{1}{n}\sum_{i=1}^{n} \vnorm{ \aa_i - \mathcal{D}(\mathcal{E}(\aa_{i}; \xx_{\mathcal{E}}); \xx_{\mathcal{D}}) }^{2} + \lambda \psi(\xx_{\mathcal{E}},\xx_{\mathcal{D}}),
\end{align*}
where the structures of the encoder $ \mathcal{E}: \real^{p} \to \real^{q}$ and the decoder $ \mathcal{D}: \real^{q} \to \real^{p}$ are as described in \cite{xuNonconvexEmpirical2017,martens2010deep,hinton2006reducing}. Here $ \xx \defeq [\xx^{\T}_{\mathcal{E}};\xx^{\T}_{\mathcal{E}}]^{\T} \in \real^{d} $. We add the same non-convex regularization function $ \psi  $ as in \cref{sec:nls} and set the regularization parameter to $ \lambda = 10^{-3} $. Unlike \cite{hinton2006reducing}, we use ``Tanh'' for the non-linear activation function, and our encoder-decoder structure is entirely symmetric, i.e., we do not apply an extra nonlinear activator at the end of the decoder network.  We run the experiments on two popular machine learning datasets, namely \texttt{CIFAR10} and \texttt{MNIST}; see \cref{table:para}. The algorithms are initialized from the same point, which is chosen randomly near the origin, i.e., the $ d $ components of the starting point are drawn independently from a normal distribution with standard deviation $ 10^{-8} $. The results are depicted in \cref{fig:CIFAR10,fig:MNIST}. 

In both experiments, \texttt{Newton-CG-LS} performs very poorly, which is mainly due to encountering NPC directions too often and yet not employing a forward tracking strategy. At the same time, the \texttt{Newton-CG-LS-FW} variant, which employs forward/backward tracking line-search strategy \cref{alg:line_search_forward}, shows significantly improved performance. In \cref{fig:CIFAR10}, among all methods, \texttt{Newton-CG-TR-Steihaug} and \texttt{Newton-MR} converge to the most optimal solution, albeit the \texttt{Newton-CG-TR-Steihaug} method converges relatively slower and the norm of the gradient never reaches the preset threshold. 
With the \texttt{MNIST} dataset \cref{fig:MNIST}, which arguably gives rise to a much easier problem than that using the \texttt{CIFAR10} dataset \cref{fig:CIFAR10}, with the exception of \texttt{Newton-CG-LS} and \texttt{L-BFGS}, all other methods converge to a similar solution within a comparable amount of time and computational effort. However, \texttt{Newton-MR} arrives at a point with a much smaller gradient.

\begin{table}[htb]
	\begin{center}
		\scalebox{1}{
			\begin{tabular}{|c|c|c|c|}
				\hline
				Dataset & n & d & Encoder network architecture \\ 
				\hline &&& \\ [-2ex]
				\texttt{CIFAR10} & 50,000 & 1,664,232 & 3,072-256-128-64-32-16-8 \\ [1ex]
				\hline &&& \\ [-2ex]
				\texttt{MNIST} & 60,000 & 1,154,784 & 784-512-256-128-64-32-16\\ [1ex]
				\hline
		\end{tabular}}
	\end{center}
	\caption{\small Auto-encoder model for \texttt{CIFAR10} and \texttt{MNIST} datasets. \label{table:para}}
\end{table}

\subsection{CUTEst test problems}
\label{sec:cute}

We now focus our efforts in evaluating the performance of the algorithms on a series of test problems from the  \texttt{CUTEst} test collection \cite{gould2015cutest}. In particular, we focused on the unconstrained problems, excluding the problems whose objective function is constant, linear, undefined, or unbounded below. This amounted to a test set of $ 237 $ problems. All algorithms are initialized by the same instance drawn randomly from the uniform distribution on the unit sphere. \cref{fig:CUTEST} depicts the performance profile \cite{dolan2002benchmarking,gould2016note} of various methods as measured by objective value and gradient norm. We found that the default maximum oracle call of $10^5$ was adequate to allow all methods to achieve their best possible outcome. To highlight the performance of the methods more clearly, in \cref{fig:CUTEST_cut}, the horizontal axis is capped at $ \tau = 100 $.

\begin{figure}[!htb]
	\centering
	\subfigure[Performance profile in terms of $ f(\xxk) $]
	{\includegraphics[width=.45\textwidth]{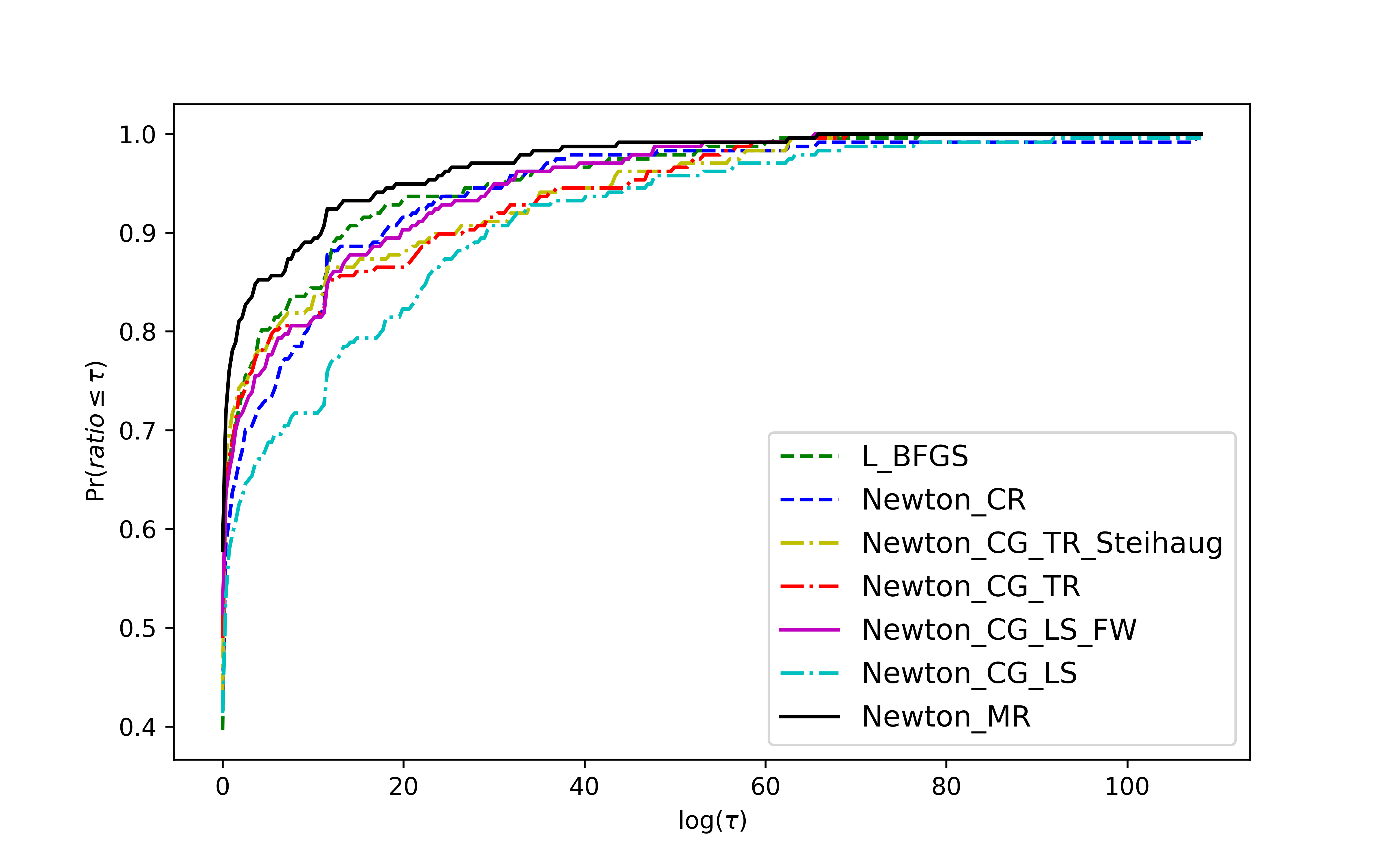}}
	\subfigure[Performance profile in terms of $ \vnorm{\nabla f(\xxk)} $]
	{\includegraphics[width=.45\textwidth]{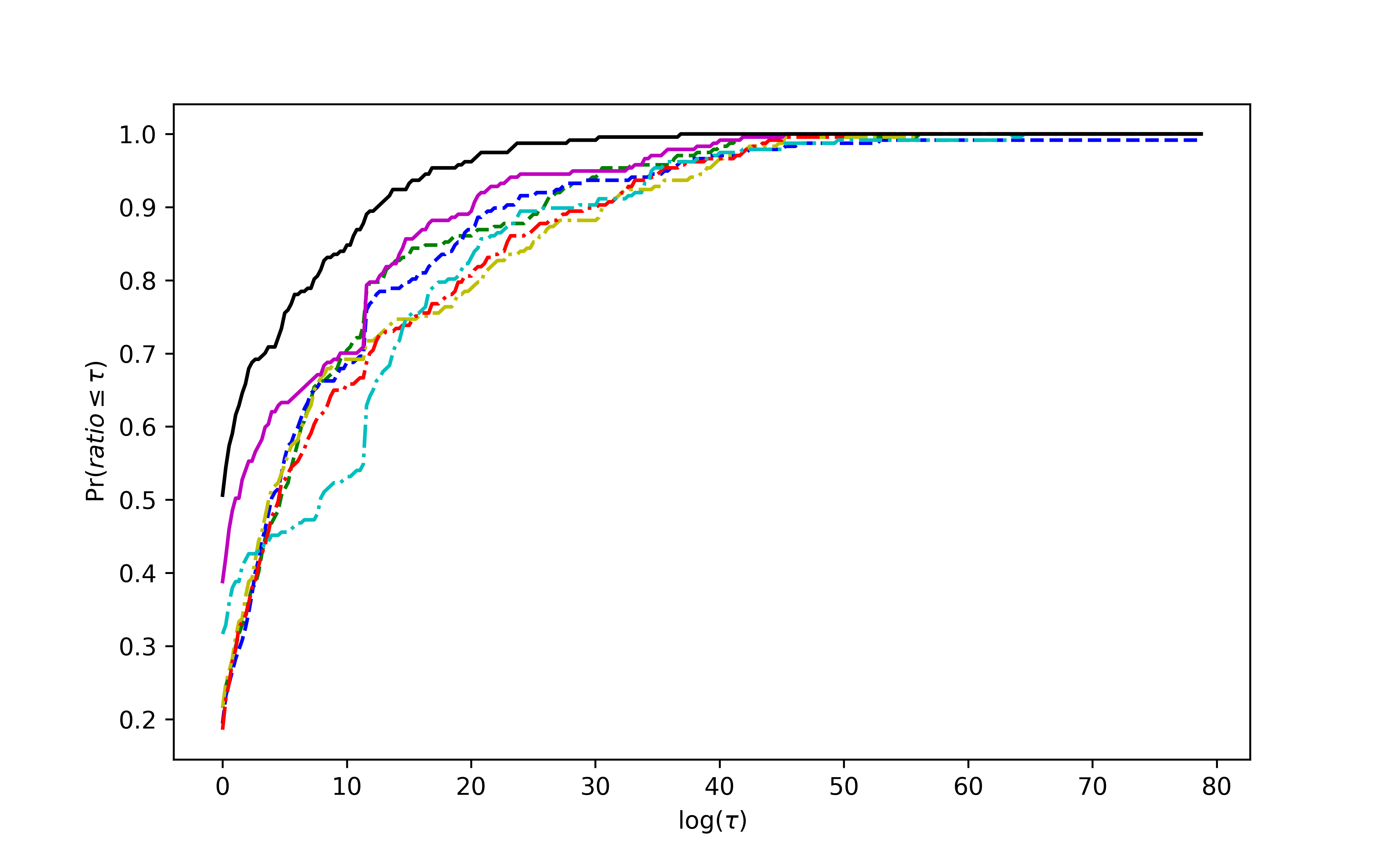}}
	\caption{The performance profile of various Newton-type methods on $ 237 $  \texttt{CUTEst} problems. For
		a given $ \tau $ in the x-axis, the corresponding value on the y-axis is the proportion of times that a given solver's performance lies within a factor $ \tau $ of the best possible performance over all runs. \label{fig:CUTEST}}
\end{figure}

From \cref{fig:CUTEST,fig:CUTEST_cut}, clearly \texttt{Newton-MR} outperforms all other algorithms both in terms of obtaining the lowest objective value and the smallest gradient norm. What is somewhat striking, however, is the relatively consistent poor performance of \texttt{Newton-CG-LS}. Of course, incorporating forward/backward tracking line-search in \texttt{Newton-CG-LS-FW} has helped improve the performance. Nonetheless, on these experiments, the advantages of employing MINRES as sub-problem solver as opposed to alternatives such as CG or CR is evidently clear.

\begin{figure}[!htb]
	\centering
	\subfigure[Performance profile in terms of $ f(\xxk) $]
	{\includegraphics[width=.45\textwidth]{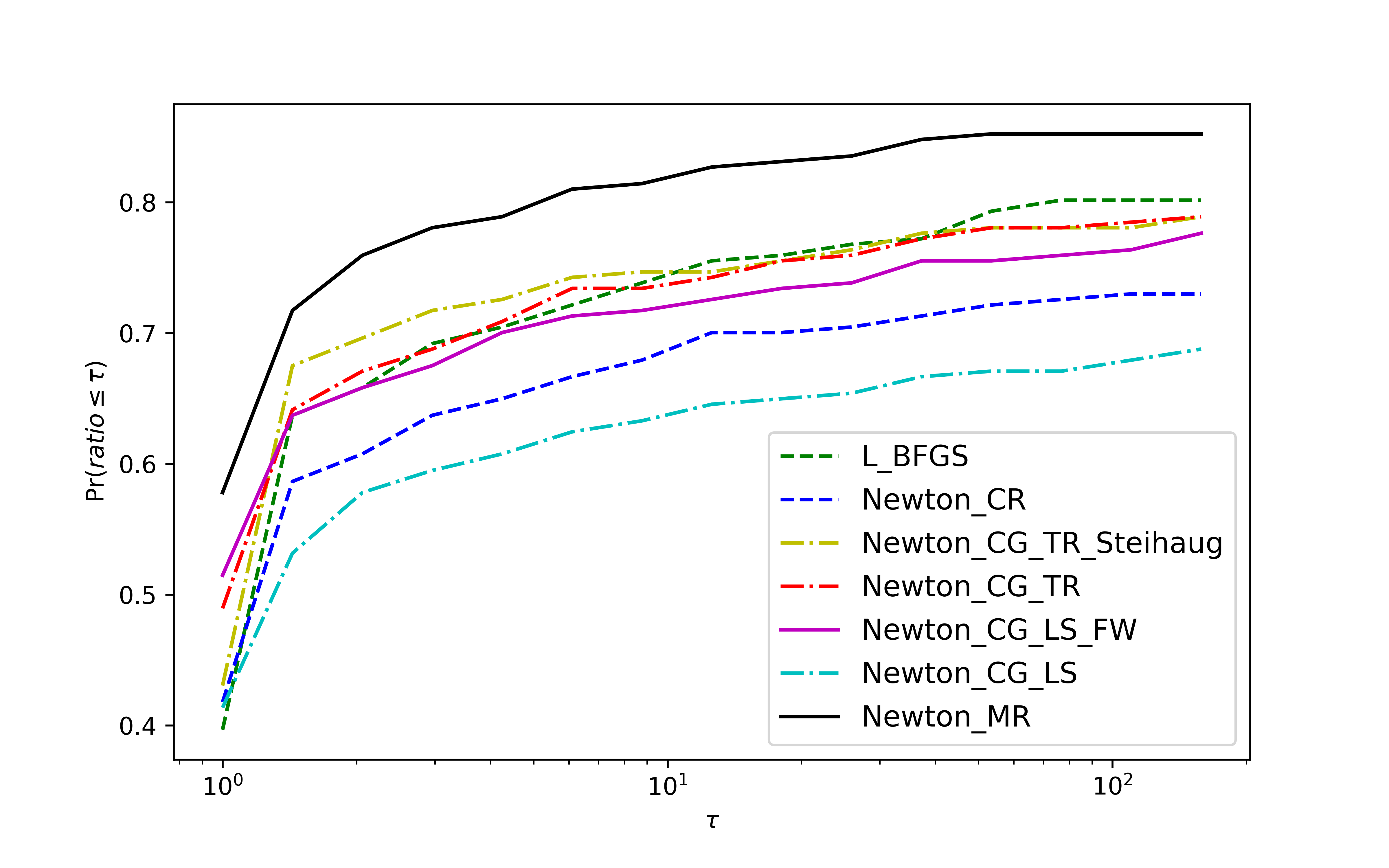}}
	\subfigure[Performance profile in terms of $ \vnorm{\nabla f(\xxk)} $]
	{\includegraphics[width=.45\textwidth]{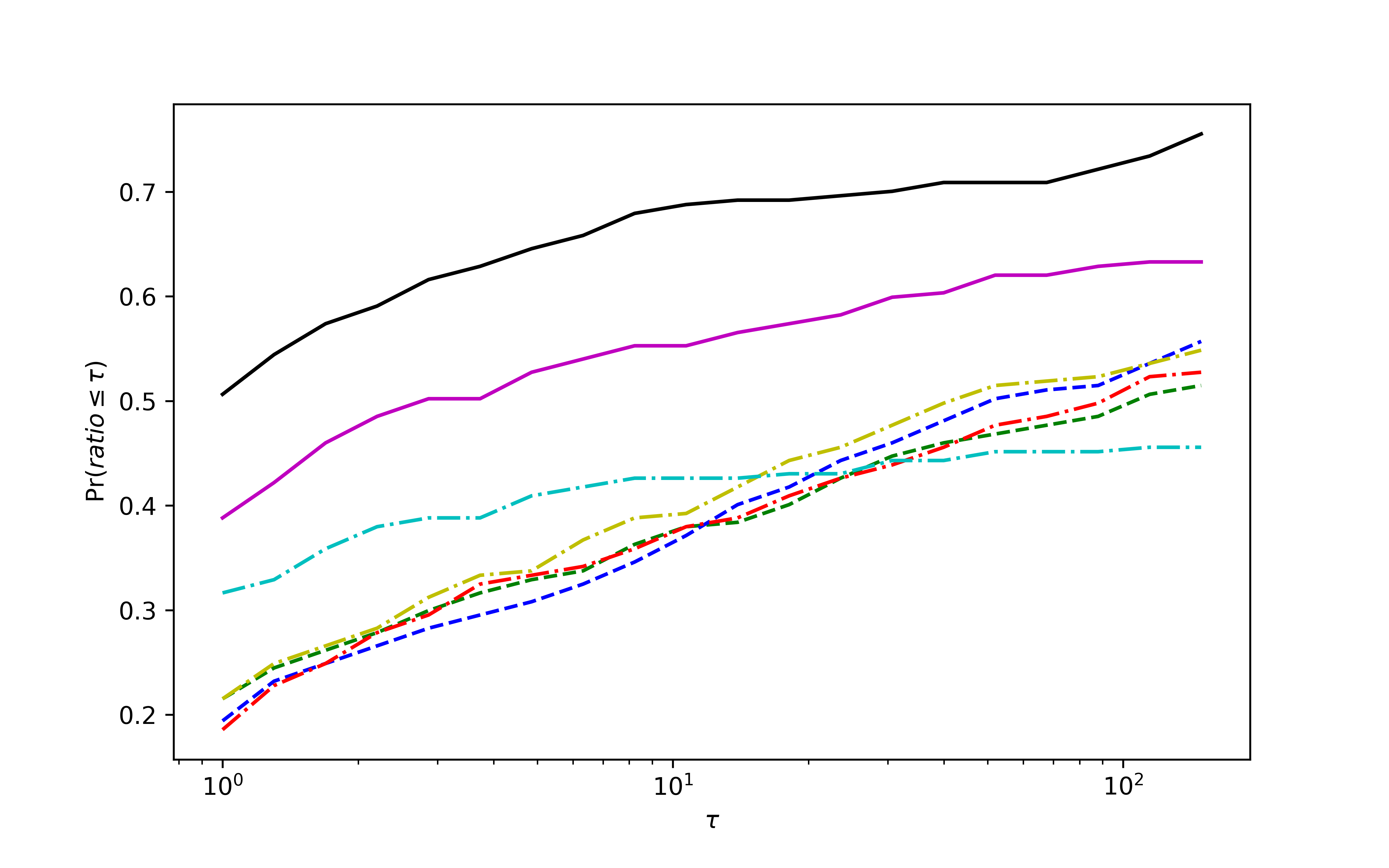}}
	\caption{The performance profile of various Newton-type methods on $ 237 $  \texttt{CUTEst} problems. For
		a given $ \tau $ in the x-axis, the corresponding value on the y-axis is the proportion of times that a given solver's performance lies within a factor $ \tau $ of the best possible performance over all runs. To distinguish the performance of the methods more clearly, $ \tau $ is capped at $100$. \label{fig:CUTEST_cut}}
\end{figure}

\section{Conclusions}
\label{sec:conc}
Building on recent results regarding various properties of MINRES \cite{liu2022minres}, we extended the Newton-MR algorithm, initially proposed in \cite{roosta2018newton} and limited to invex optimization problems, to more general non-convex settings. This is done by leveraging non-positive curvature directions, when they arise, as part of MINRES iterations. We established complexity guarantees for convergence to first and second-order approximate optimality, which are known to be optimal. To achieve this, we provided a novel convergence analysis for MINRES, which improves upon the existing bounds in terms of dependence on the spectrum for indefinite matrices. Furthermore, under the benign saddle property, a novel assumption which is weaker than the widely used strict saddle property, we were able to greatly improve the second-order complexity guarantee and obtain a rate that, to our knowledge, is the state-of-the-art. In contrast to similar alternative methods where CG iterations are greatly modified to extract NPC directions, our algorithms are simple in that the NPC directions, if they exist, arise naturally within MINRES iterations without additional algorithmic modifications. We also demonstrated the superior performance of our algorithm, as compared with several alternative Newton-type methods, on several non-convex problems.

\appendix 
\section{MINRES: Review and Further Details}
\label{sec:MINRES_review}

In this section, for the sake of completeness, we review MINRES in some details and highlight the useful theoretical results that are relevant to the analysis of this paper. In doing so, we follow the presentation of \cite{liu2022minres} almost verbatim, however, we adjust the notation to match our setting here. For more details on MINRES and its various theoretical properties, see \cite{liu2022minres} and references therein. 
In this section, for notational simplicity, we drop the dependence on MINRES iterations $ \sskt $ on the outer iterate $ \xxk$, i.e., we use $ \sst $ instead.

\subsection{Algorithmic Details}
\label{sec:MINRES_details}
Recall MINRES, depicted in \cref{alg:MINRES}, is a method using Krylov subspace methods to solve the symmetric linear least-squares problem \cref{eq:MINRES}. MINRES involves three major ingredients: Lanczos process, QR decomposition, and the update of its iterates.

\subsubsection*{Lanczos process}
With $\vv_{1} = \gg/\vnorm{\gg}$, recall that after $ t $ iterations of the Lanczos process, and in the absence of round-off errors, the Lanczos vectors form an orthogonal matrix $ \VV_{t+1} = [ \vv_1 \mid \vv_2 \mid \dots  \mid \vv_{t+1} ] \in \real^{\dn \times (t+1)} $, whose columns span $ \Klv_{t+1}(\HH,\gg) $ and satisfy the familiar relation $ \HH \VV_{t} = \VV_{t+1} \bTT_{t} $, where $ \bTT_{t} \in \real^{(t+1) \times t} $ is an upper-Hessenberg matrix of the form
\begin{align}
	\label{eq:tridiagonal_T}
	\bTT_{t} &= 
	\begin{bmatrix}
		\talpha_1 & \tbeta_2 & & & \\
		\tbeta_2 & \talpha_2 & \tbeta_3 & & \\
		& \tbeta_3 & \talpha_3 & \ddots & \\
		& & \ddots & \ddots & \tbeta_{t} \\
		& & & \tbeta_{t} & \talpha_{t} \\
		\hdashline
		& & & & \tbeta_{t+1} \\
	\end{bmatrix}
	\triangleq
	\begin{bmatrix}
		\TT_{t} \\
		\tbeta_{t+1}\ee^\TT_{t}
	\end{bmatrix}.
\end{align}
Subsequently, we get the three-term recursion
\begin{align*}
	\HH \vv_{t} = \tbeta_{t} \vv_{t-1} + \talpha_{t} \vv_{t} + \tbeta_{t+1} \vv_{t+1}, \quad t \geq 2,
\end{align*}
and the Lanczos process is terminated when $ \tbeta_{t+1} = 0 $. 

Noting that $ \mathcal{K}_{t}(\HH, \gg) = \range(\VV_{t}) = \Span\{\vv_1, \dots, \vv_{t}\} $, we let $ \ss_{t} = \VV_{t}\yy_{t} $ for some $ \yy_{t} \in \real^{t} $, and it follows that the residual $ \rr_{t} $ can be written as 
\begin{align*}
	\rr_{t} = - \gg - \HH \ss_{t} = - \gg - \HH \VV_{t} \yy_{t} = - \gg - \VV_{t+1} \bTT_{t} \yy_{t} = - \VV_{t+1}(\vnorm{\gg} \ee_1 + \bTT_{t} \yy_{t}).
\end{align*}
This gives rise to the well-known sub-problems of MINRES as
\begin{align}
	\label{eq:MINRES_sub}
	\min_{\yy_{t} \in \real^{t}} \vnorm{\tbeta_1 \ee_1 + \bTT_{t} \yy_{t}}, \quad \tbeta_{1} = \vnorm{\gg}.
\end{align}

\subsubsection*{QR decomposition}
\label{sec:QR}
Recall that \cref{eq:MINRES_sub} is solved using the QR factorization of $ \bTT_{t} $. Let $ \QQ_{t} \bTT_{t} = \bRR_{t} $ be the full QR decomposition of $ \bTT_{t} $ where $ \QQ_{t} \in \real^{(t+1) \times (t+1)} $ and $ \bRR_{t} \in \real^{(t+1) \times t} $. Typically, $\QQ_{t}$ is formed, implicitly, by the application of series of Householder reflections to transform $ \bTT_{t} $ to the upper-triangular matrix $ \bRR_{t} $. Recall that each Householder reflection only affects two rows of the matrix that is being triangularized. More specifically, two successive application of Householder reflections can be compactly written by only considering the elements of the matrix that are being affected as 
\begin{align*}
	\begin{bmatrix}
		1 & 0 & 0 \\
		0 & c_{i-1} & s_{i-1} \\
		0 & s_{i-1} & -c_{i-1}
	\end{bmatrix} \begin{bmatrix}
		c_{i-2} & s_{i-2} & 0\\
		s_{i-2} & -c_{i-2} & 0 \\
		0 & 0 & 1 
	\end{bmatrix} 
	\begin{bmatrix}
		\gamma_{i-2} & \delta_{i-1} & 0 & 0 \\
		\tbeta_{i-1} & \talpha_{i-1} & \tbeta_{i} & 0 \\
		0 & \tbeta_{i} & \talpha_{i} & \tbeta_{i+1}
	\end{bmatrix}&
	\nonumber \\ 
	= \begin{bmatrix}
		1 & 0 & 0 \\
		0 & c_{i-1} & s_{i-1} \\
		0 & s_{i-1} & -c_{i-1}
	\end{bmatrix} \begin{bmatrix}
		\gamma_{i-2}^{[2]} & \delta_{i-1}^{[2]} & \epsilon_{i} & 0 \\
		0 & \gamma_{i-1} & \delta_{i} & 0 \\
		0 & \tbeta_{i} & \talpha_{i} & \tbeta_{i+1} \\
	\end{bmatrix}& \nonumber \\ 
	= \begin{bmatrix}
		\gamma_{i-2}^{[2]} & \delta_{i-1}^{[2]} & \epsilon_{i} & 0 \\
		0 & \gamma_{i-1}^{[2]} & \delta_{i}^{[2]} & \epsilon_{i+1} \\
		0 & 0 & \gamma_{i} & \delta_{i+1} \\
	\end{bmatrix}&,
\end{align*}
where $ 3 \leq i \leq t-1 $ and 
\begin{align}
	\label{eq:c_s_r2}
	c_{j}   = \frac{\gamma_{j}}{\gamma_{j}^{[2]}}, \quad s_{j} = \frac{\tbeta_{j+1}}{\gamma_{j}^{[2]}}, \quad \gamma_{j}^{[2]} = \sqrt{(\gamma_{j})^2 + \tbeta_{j+1}^2} = c_{j} \gamma_{j} + s_{j} \tbeta_{j+1}, \quad \quad 1 \leq j \leq t.
\end{align}
Here, the $ 2 \times 2 $ sub-matrix made of $c_{j}$ and $s_{j}$ is the special case of a Householder reflector in dimension two \cite[p.\ 76]{trefethen1997numerical}.

Consequently, we can rewrite $\QQ_{t}$ and $ \bRR_{t} $ in block form as 
\begin{subequations}
	\label{eq:T_QR_decomp}
	\begin{align}
		\QQ_{t} \bTT_{t} = \bRR_{t} \triangleq 
		\begin{bmatrix}
			\RR_{t}\\
			\zero^{\T}
		\end{bmatrix}, \quad 
		\RR_{t} =
		\begin{bmatrix}
			\gamma_1^{[2]} & \delta_2^{[2]} & \epsilon_3 & \\
			&\gamma_2^{[2]} & \delta_3^{[2]} & \ddots \\
			& & \ddots & \ddots & \epsilon_{t} \\
			& & & \gamma_{t-1}^{[2]}  & \delta_{t}^{[2]} \\
			& & & & \gamma_{t}^{[2]} \\
		\end{bmatrix}&. \label{eq:block_R} \\
		\QQ_{t} \triangleq \prod_{i=1}^{t}\QQ_{i,i+1}, \quad \QQ_{i,i+1} \triangleq 
		\begin{bmatrix}
			\eye_{i-1} & & & \\
			& c_{t} & s_{t} &\\
			& s_{t} & -c_{t} & \\
			& & & \eye_{t-i}
		\end{bmatrix}&, \label{eq:block_Q}
	\end{align}
\end{subequations}
In fact, the same series of transformations are also simultaneously applied to $\tbeta_{1} \ee_1$ as 
\begin{align*}
	\QQ_{t}\tbeta_1\ee_1 = \tbeta_1
	\begin{bmatrix}
		c_1\\
		s_1 c_2\\
		\vdots\\
		s_1s_2\dots s_{t-1}c_{t}\\
		s_1s_2\dots s_{t-1}s_{t}
	\end{bmatrix} \triangleq
	\begin{bmatrix}
		\tau_1\\
		\tau_2\\
		\vdots\\
		\tau_{t}\\
		\phi_{t}
	\end{bmatrix} \triangleq \begin{bmatrix}
		\ttt_{t} \\
		\phi_{t}
	\end{bmatrix}.
\end{align*}

With these quantities available, we can solve \cref{eq:MINRES_sub} by noting that
\begin{align*}
	\min_{\yy_{t} \in \real^{t}}  \vnorm{\rr_{t}} &= \min_{\yy_{t} \in \real^{t}} \vnorm{\tbeta_1\ee_1 + \bTT_{t}\yy_{t}}  = \min_{\yy_{t} \in \real^{t}} \vnorm{\QQ_{t}^{\T}\left(\QQ_{t}\tbeta_1\ee_1 + \QQ_{t} \bTT_{t} \yy_{t}\right)} \\
	& = \min_{\yy_{t} \in \real^{t}} \vnorm{\QQ_{t}\tbeta_1\ee_1 + \QQ_{t} \bTT_{t} \yy_{t}} = \min_{\yy_{t} \in \real^{t}} \vnorm{\begin{bmatrix}
			\ttt_{t}\\
			\phi_{t}
		\end{bmatrix} + 
		\begin{bmatrix}
			\bRR_{t} \\
			\zero^{\T}
		\end{bmatrix}\yy_{t}}.
\end{align*}
Note that this in turn implies that $ \phi_{t} = \| \rr_{t} \| $. We also trivially have $ \phi_{0} = \tbeta_{1} = \vnorm{\gg} $.

\subsubsection*{Updates}
Let $ t < g $ and define $ \WW_{t} $ from solving the lower triangular system $ \RR_{t}^{\T} \WW_{t}^{\T} = \VV_{t}^{\T} $ where $ \RR_{t} $ is as in \cref{eq:block_R}. Now, letting $ \VV_{t} = [ \VV_{t-1} \mid \vv_{t}] $, and using the fact that $ \RR $ is upper-triangular, we get the recursion $ \WW_{t} = [ \WW_{t-1} \mid \ww_{t}] $ for some vector $ \ww_{t} $. As a result, using $ \RR_{t} \yy_{t} = \ttt_{t} $, one can update the iterate as 
\begin{align*}
	\ss_{t} = \VV_{t} \yy_{t} = \WW_{t} \RR_{t} \yy_{t} = \WW_{t} \ttt_{t} = \begin{bmatrix}
		\WW_{t-1} & \ww_{t}
	\end{bmatrix} \begin{bmatrix}
		\ttt_{t-1} \\
		\tau_{t}
	\end{bmatrix} = \ss_{t-1} + \tau_{t} \ww_{t}.
\end{align*}
We also always set $ \ss_{0} = \zero $. Furthermore, from $ \VV_{t} = \WW_{t} \RR_{t} $, i.e., 
\begin{align*}
	\begin{bmatrix}
		\vv_1 & \vv_2 & \ldots & \vv_{t}
	\end{bmatrix} = \begin{bmatrix}
		\ww_1 & \ww_2 & \ldots & \ww_{t}
	\end{bmatrix} \begin{bmatrix}
		\gamma_1^{[2]} & \delta_2^{[2]} & \epsilon_3 & & \\
		& \gamma_2^{[2]} & \ddots & \ddots & \\
		& & \ddots & \ddots & \epsilon_{t} \\
		& & & \gamma_{t-1}^{[2]}  & \delta_{t}^{[2]} \\
		& & & & \gamma_{t}^{[2]}
	\end{bmatrix},
\end{align*}
we get the following relationship for computing $ \vv_{t} $ as
\begin{align*}
	\vv_{t} = \epsilon_{t} \ww_{t-2} + \delta^{[2]}_{t} \ww_{t-1} + \gamma^{[2]}_{t} \ww_{t}.
\end{align*}
All of the above steps constitute MINRES method, which is given in \cref{alg:MINRES}. 

\subsection{Other Relevant Properties}
\label{sec:MINRES_property}
Beyond the descent implications discussed in \cref{sec:MINRES_Descent}, MINRES offers a plethora of relevant properties that we leverage in developing the algorithms of this paper and obtaining their convergence guarantees. We briefly mention these relevant properties here and invite the reader to consult \cite{liu2022minres} for further details and proofs.

The following lemma contains several useful facts about the iterations of MINRES.
\begin{lemma}
	\label{lemma:MINRES_properties}
	Let $g$ be the grade of $ \gg $ with respect to $ \HH $. In MINRES, for any $ 1 \leq t \leq g $, we have 
	\begin{subequations}        
		\begin{align}
			\dotprod{\rr_{t}, \HH \rr_{i}} &= 0, \quad i \neq t, \quad 1 \leq i \leq g, \label{eq:rTHr_conjugated} \\
			\dotprod{\rr_{t}, \gg} &= - \vnorm{\rr_{t}}^2, \label{eq:rTg} \\
			\dotprod{\rr_{t-1}, \HH \rr_{t-1}} &= - c_{t-1} \gamma_{t} \vnorm{\rr_{t-1}}^2, \label{eq:NPC_curve} \\
			\vnorm{\HH \sstp} &= \sqrt{\phi_{0}^{2} - \phi_{t-1}^{2}}, \label{eq:Hs} \\
			\vnorm{\HH \rrtp} &= \phi_{t-1} \sqrt{\gamma_{t}^{2} + (\delta_{t+1})^{2}} \label{eq:Hr}.
		\end{align}
	\end{subequations}
\end{lemma}
\begin{proof}
	We get \cref{eq:rTHr_conjugated,eq:rTg,eq:NPC_curve} directly from \cite[Lemma 3.1]{liu2022minres}. Applying \cref{eq:rTg}, we obtain \cref{eq:Hs} as 
	\begin{align*}
		\vnorm{\HH \sstp}^2 &= \vnorm{\rrtp + \gg}^2 = \vnorm{\gg}^{2} - \vnorm{\rrtp}^{2} = \phi_{0}^{2} - \phi_{t-1}^{2}.
	\end{align*}
	By \cite[Eqn (3.2)]{liu2022minres}, we can formulate $ \HH \rrtp $ as
	\begin{align*}
		\HH \rrtp = \phi_{t-1} (\gamma_{t} \vv_{t} + \delta_{t+1} \vv_{t+1}),
	\end{align*}
	which implies \cref{eq:Hr}.
\end{proof}

From \cref{eq:NPC_curve}, it is immediate that if for some $ t \leq g $, we have
\begin{align}
	\label{eq:NPC_cond}
	- c_{t-1} \gamma_{t} \leq 0,
\end{align}
then $ \rr_{t-1} $ must be a NPC direction for $ \HH $, i.e., \cref{eq:NPC} holds. The crucial question, answered in \cite{liu2022minres}, is whether or not the condition \cref{eq:NPC_cond} is both necessary and sufficient for a NPC direction to be available. This question was investigated by observing a tight connection between the tridiagonal symmetric matrix $ \TTt $ and the condition \cref{eq:NPC_cond}. Indeed, since $ \rrtp \in \Kt{\HH, \gg} $, we can $ \rrtp = \VVt \pp $ for some non-zero $ \pp \in \real^{t} $. Hence, 
\begin{align*}
	\frac{\dotprod{\pp, \TTt \pp}}{\vnorm{\pp}^2} = \frac{\dotprod{\VVt \pp, \HH \VVt \pp}}{\vnorm{\VVt \pp}^2} = \frac{\dotprod{\rr_{t-1}, \HH \rr_{t-1}}}{\vnorm{\rr_{t-1}}^2} = - c_{t-1} \gamma_{t}.
\end{align*}
Clearly, as long as $ \TT_{t} \succ \zero $, the condition \cref{eq:NPC_cond} cannot hold. The following result from \cite{liu2022minres} shows the converse also holds, i.e., as soon as  $ \TT_{t} \not \succ \zero $ for some $ t \leq g $, MINRES declares $ \rrktp $ as a NPC direction.
\begin{lemma}[{\!\!\cite[Theorem 3.3]{liu2022minres}}]
	\label{lemma:MINRES_NPC_detector}
	Let $ k \leq g $ where $ g $ is the grade of $ \gg $ with respect to $ \HH $. If $ \TTt \not \succ \zero $, then the NPC condition \cref{eq:NPC_cond} holds for some $ t \leq k $. In particular, if $ t \leq g $ is the first iteration where $ \TT_{t} \not \succ \zero $, then the NPC condition \cref{eq:NPC_cond} holds.
\end{lemma}

As long as the NPC condition \cref{eq:NPC_cond} has not been detected, MINRES enjoys additional properties regarding the signs of certain quadratic functions as well as the monotonicity of certain quantities, which are used in the analysis of this paper. In particular, before \cref{eq:NPC_cond} is detected, \cref{lemma:sTr} shows that, not only is $ \sst $ a descent direction for $ \vnorm{\gg}^{2} $, but it can also yield descent for the original objective $ f $. 
\begin{lemma}[{\!\!\cite[Theorems 3.8 and 3.11]{liu2022minres}}]
	\label{lemma:sTr}
	Let $g$ be the grade of $ \gg $ with respect to $ \HH $. As long as the NPC condition \cref{eq:NPC_cond} has not been detected for $1 \leq t \leq g$, we must have $ \dotprod{\ss_{t}, \gg} + \dotprod{\ss_{t}, \HH \ss_{t}}  \leq 0 $.
\end{lemma}

    \section*{Acknowledgments}
    Y. Liu is supported by the Hong Kong Innovation and Technology Commission (InnoHK Project CIMDA). Fred Roosta was partially supported by the Australian Research Council through an Industrial Transformation Training Centre for Information Resilience (IC200100022) as well as a Discovery Early Career Researcher Award (DE180100923).
    
    \bibliographystyle{plain}
    \bibliography{biblio}
	
\end{document}